\documentclass[11pt, a4paper]{amsart}

\usepackage[style=alphabetic]{biblatex}
\usepackage{amsmath, amsthm, amssymb,fullpage}

\usepackage{color}
\usepackage{xcolor}

\usepackage{enumitem}
\usepackage{hyperref}

\newcommand{\Bc}{\mathcal{B}}
\newcommand{\Fc}{\mathcal{F}}

\newcommand{\Pc}{\mathcal{P}}
\newcommand{\Qc}{\mathcal{Q}}

\theoremstyle{plain}
\newtheorem{theorem}{Theorem}[section]
\newtheorem{corollary}[theorem]{Corollary}
\newtheorem{lemma}[theorem]{Lemma}
\newtheorem{proposition}[theorem]{Proposition}
\newtheorem*{claim}{Claim}

\theoremstyle{definition}
\newtheorem{definition}[theorem]{Definition}

\theoremstyle{remark}
\newtheorem{remark}[theorem]{Remark}

\numberwithin{equation}{section}

\renewcommand{\P}{\mathbb P}
\newcommand{\C}{\mathbb C}
\newcommand{\R}{\mathbb R}
\newcommand{\lam}{\lambda}
\newcommand{\eps}{\epsilon}

\newcommand{\abs}[1]{\left|#1\right|}
\newcommand{\pa}[1]{\left(#1\right)}
\newcommand{\norm}[1]{\left\|#1\right\|}

\newcommand{\Cc}{\mathcal C}

\newcommand{\Ll}{\mathcal L}
\newcommand{\Hh}{\mathcal H}

\usepackage{bbold}
\newcommand{\1}{\mathbb 1}
\newcommand{\sca}[1]{\left\langle#1\right\rangle}
\newcommand{\set}[1]{\left\{#1\right\}}

\newcommand{\D}{\mathbb D}
\newcommand{\N}{\mathbb N}
\newcommand{\B}{\mathbb B}

\usepackage{tikz-cd}

\DeclareMathOperator{\Leb}{Leb}
\DeclareMathOperator{\dist}{dist}
\DeclareMathOperator{\supp}{supp}

\newcommand{\Ex}{\mathbb{E}}

\newcommand{\FS}{{\text{\rm \tiny FS}}}

\newcommand{\ddc}{{dd^c}}

\newcommand{\dbar}{{\overline\partial}}
\newcommand{\ddbar}{{\partial\overline\partial}}

\newcommand{\ap}{{p,\alpha}}
\newcommand{\sap}{{\langle p,\alpha\rangle}}
\newcommand{\sapg}{{\sap,\gamma}}

\newcommand{\h}{\mathbb{h}}

\DeclareMathOperator{\ent}{Ent}
\newcommand{\Ent}[2]{\ent_{#1} ({#2})}
\DeclareMathOperator{\ess}{ess}
\newcommand{\ress}{r_{\ess}}

\usepackage{enumitem}
\usepackage{hyperref}



\begin{document} 

\hyphenpenalty=10000

\title[Equilibrium states of  endomorphisms of $\P^k$ II]{Equilibrium states of endomorphisms of $\P^k$ II:\\
 spectral stability and limit theorems}

\begin{author}[F.~Bianchi]{Fabrizio Bianchi}
\address{ 
CNRS, Univ. Lille, UMR 8524 - Laboratoire Paul Painlev\'e, F-59000 Lille, France}
  \email{fabrizio.bianchi$@$univ-lille.fr}
\end{author}

\begin{author}[T.C.~Dinh]{Tien-Cuong Dinh}
\address{National University of Singapore, Lower Kent Ridge Road 10,
Singapore 119076, Singapore}
\email{matdtc$@$nus.edu.sg }
\end{author}

\maketitle

\begin{abstract}
 We
  establish the existence of a spectral gap for
 the transfer operator 
induced on $\P^k = \P^k (\C)$
 by a generic holomorphic endomorphism and a suitable continuous weight 
 and its perturbations on various
  functional spaces,
which is
 new even in dimension one.
The main issue to overcome
is the rigidity of the complex objects, since the transfer operator is a 
non-holomorphic
 perturbation
of the 
operator $f_*$.
The system is moreover
non-uniformly hyperbolic
and one may have
critical points on the Julia set.
The construction of our
norm requires the introduction and study of several intermediate new norms,
and a careful combination of ideas from pluripotential and interpolation theory.
As far as we know, this is the first time that
pluripotential methods have been applied to solve a mixed real-complex problem.

Thanks to  the spectral gap,
we establish 
an exponential speed of convergence for
 the equidistribution of the backward orbits of points
 towards the
conformal measure.
Moreover, we obtain a full list
of statistical properties for the equilibrium states:
 exponential mixing, 
CLT, Berry-Esseen theorem, local CLT,
ASIP, LIL, LDP, almost sure CLT.
Many of these properties are new even in dimension one, some even in the case
of zero weight function (i.e., for the measure of maximal entropy).
\end{abstract}

\setcounter{secnumdepth}{3}
\setcounter{tocdepth}{1}
\tableofcontents

\bigskip

\noindent
{\bf Notation.} 
Throughout the paper, $\P^k$ denotes the complex projective space of dimension $k$
endowed with the standard Fubini-Study form $\omega_\FS$. This is a K\"ahler $(1,1)$-form
normalized so that $\omega_\FS^k$ is a probability measure. We will use the metric and
distance $\dist(\cdot,\cdot)$ on $\P^k$ induced by $\omega_\FS$ and the standard ones
on $\C^k$ when we work on open subsets of $\C^k$. 
We denote by $\B_{\P^k}(a,r)$ (resp.\ $\B_r^k, \D(a,r), \D_r$) 
the ball of center $a$ and radius $r$ in $\P^k$  (resp.\ the ball of center 0 and radius $r$ in $\C^k$,
the disc of center $a$ and radius $r$ in $\C$, and the disc
of center $0$ and radius $r$ in $\C$). 
$\Leb$ denotes the standard Lebesgue measure on a
 Euclidean
 space or on a sphere.

The pairing $\langle \cdot,\cdot\rangle$ is used for the integral of a function with respect to a
measure or more generally the value of a current at a test form. If $S$ and
$R$ are two $(1,1)$-currents, we will write $|R|\leq S$ when $\Re (\xi R)\leq S$
for every function $\xi\colon\P^k\to\C$ with $|\xi|\leq 1$,
 i.e., all currents $S- \Re (\xi R)$ with $\xi$ as before are positive. Notice 
that this forces $S$ to be real and positive. We also write other inequalities such
as $|R|\leq |R_1|+|R_2|$ if $|R|\leq S_1+S_2$ whenever $|R_1|\leq S_1$ and $|R_2|\leq S_2$.
Recall that $d^c={i\over 2\pi}(\dbar -\partial)$ and $\ddc={i\over\pi}\ddbar$. 
The notations $\lesssim$ and $\gtrsim$ stand for inequalities up to a multiplicative constant.
The function identically equal to 1 is denoted by $\1$.
We also use the function $\log^\star (\cdot):=1+|\log (\cdot)|$. 

Consider a holomorphic endomorphism $f\colon\P^k\to\P^k$ of
 algebraic degree $d\geq 2$ satisfying the Assumption {\bf (A)} in the Introduction. 
Denote respectively by $T$, $\mu=T^k$, $\supp(\mu)$ the Green $(1,1)$-current, 
the measure of maximal entropy (also called
the Green measure or the
equilibrium measure), and the small Julia set of $f$. 
If $S$ is a positive closed $(1,1)$-current on $\P^k$, its dynamical potential is denoted by $u_S$ and is
defined in Section \ref{ss:dyn-pot}.

We also consider a weight $\phi$ which is a
continuous function on $\P^k$. 
We often assume that $\phi$ is real.
The
transfer operator (Perron-Frobenius operator) $\Ll=\Ll_\phi$ is introduced in the
Introduction together with the scaling ratio $\lambda= \lambda_\phi$, the conformal measure $m_\phi$,
the density function $\rho=\rho_\phi$, the equilibrium state $\mu_\phi=\rho m_\phi$, the pressure $P(\phi)$.
The measures $m_\phi$ and $\mu_\phi$ are probability measures.
The operator $L$ is a suitable modification of $\Ll$
and is introduced
 at the beginning of
 Section 
 \ref{s:statistic}. 

The oscillation $\Omega(\cdot)$, the modulus of continuity $m(\cdot,\cdot)$, the
semi-norms $\norm{\cdot}_{\log^p}$ and $\norm{\cdot}_*$ of a function
are defined in Section
  \ref{ss:def}.
Other
 norms and semi-norms $\norm{\cdot}_p$, $\norm{\cdot}_{\ap}$,
$\norm{\cdot}_{\sap}$,
 $\norm{\cdot}_{\sapg}$
 for $(1,1)$-currents and functions
are introduced in Section \ref{s:norm} and the norms
$\norm{\cdot}_{\diamond_1}, \norm{\cdot}_{\diamond_2}$ 
 in Section \ref{s:proof:t:goal}.
 The semi-norms we
consider are almost norms: they vanish only on constant functions.
It is easy to make them norms by adding
a suitable functional such as
$g \mapsto  \abs{\sca{m_\phi, g}}$. However, for
simplicity, it is more convenient to work directly
with these semi-norms. The versions of these semi-norms for currents are actually norms. 
The positive numbers $q_0,q_1,q_2$ are defined in Lemmas \ref{l:alpha-p-q}, \ref{l:cpt_apmixed}, \ref{l:logq-alpha-p-gamma} and the families of 
weights $\Pc(q,M,\Omega)$
and
$\Qc_0$ are introduced
in Sections
\ref{ss:new-recall}
and \ref{ss:proofs_gap}.

\section{Introduction and results} \label{s:Intro}

Let $f\colon \P^k\to \P^k$ be a holomorphic endomorphism of 
the complex projective space $\P^k=\P^k (\C)$, with $k\geq 1$, of algebraic degree $d\geq 2$.
Denote by $\mu$ the unique measure of maximal entropy
for the dynamical system $(\P^k, f)$
\cite{lyubich1983entropy,briend2009deux,dinh2010dynamics,berteloot2001rudiments}.
The support $\supp(\mu)$ of $\mu$ is called {\it the small Julia} set of $f$.
Given a \emph{weight}, i.e., a real-valued continuous
function, $\phi$ as above, we define the \emph{pressure} of $\phi$
as
$$P(\phi) := \sup \big\{ \Ent{f}{\nu} + \langle\nu,\phi\rangle \big\},$$
where the supremum is taken over all Borel $f$-invariant probability measures $\nu$ and $\Ent{f}{\nu} $
denotes the metric entropy of $\nu$.
An \emph{equilibrium state} for $\phi$ is then an invariant probability measure $\mu_\phi$ 
realizing a maximum in the above formula, that is, 
$$P(\phi) = \Ent{f}{\mu_\phi}+ \langle\mu_\phi,\phi\rangle.$$
Define  also
 the \emph{Perron-Frobenius} (or \emph{transfer}) operator $\Ll$
with weight $\phi$ as (we often drop the index $\phi$ for simplicity)
\begin{equation} \label{e:L}
\Ll g(y):=\Ll_\phi g (y):= \sum_{x \in f^{-1}(y)} e^{\phi (x)} g(x),
\end{equation}
where $g\colon\P^k \to \R$ is a continuous test function and the
points $x$ in the sum are counted with multiplicity. 
A \emph{conformal measure} is an eigenvector
for the dual operator $\Ll^*$ acting on positive measures.

 The following
result was obtained in \cite{bd-eq-states-part1}.
We refer to 
that paper
 for references to earlier
related
results.

\begin{theorem}\label{t:main}
Let $f$ be an endomorphism of $\P^k$ of algebraic degree $d\geq 2$ and
satisfying the Assumption {\bf (A)} below. Let $\phi$ be 
a real-valued $\log^q$-continuous function on $\P^k$, for some $q>2$, 
such that  $\Omega (\phi) :=\max \phi - \min \phi < \log d$. Then
$\phi$ admits a unique equilibrium state $\mu_\phi$, whose support is equal to  
the small Julia set of $f$. This measure $\mu_\phi$ is $K$-mixing
and mixing of all orders, and 
 repelling 
periodic 
 points of period $n$
 (suitably weighted) 
 are equidistributed with respect to 
 $\mu_\phi$ as $n$ goes to infinity.
Moreover, there is a unique conformal measure $m_\phi$
associated to $\phi$. We have $\mu_\phi=\rho m_\phi$ for some strictly positive continuous function $\rho$ on  $\P^k$
and the preimages of points by $f^n$ (suitably weighted)
 are equidistributed with respect to $m_\phi$ as $n$ goes to infinity.
\end{theorem}

Recall
that a function is $\log^q$-continuous if its oscillation
on a ball of radius $r$ is bounded by a constant times
$(\log^\star r)^{-q}$, 
see Section
\ref{ss:def}
 for details. 
We made use of the 
following technical assumption for $f$:

\medskip\noindent
{\bf (A)} \hspace{1cm} the local degree of the iterate $f^n:=f\circ\cdots\circ f$ ($n$ times) satisfies
$$\lim_{n\to\infty} {1\over n} \log\max_{a\in\P^k}\deg(f^n,a) =0.$$

\medskip\noindent
Here, $\deg(f^n,a)$ is the multiplicity of $a$ as a solution of the
equation $f^n(z)=f^n(a)$. Note that generic endomorphisms of $\P^k$ satisfy
 this condition, 
 see \cite{dinh2010equidistribution}.
We assume 
{\bf (A)} also throughout all the current paper.

\medskip

A reformulation of Theorem \ref{t:main} is the following:
given $\phi$ as in the statement, there exist a number $\lam >0$ and
a continuous function $\rho=\rho_\phi \colon \P^k \to \R$ 
such that, for every continuous function $g\colon \P^k\to \R$, 
the following uniform convergence holds:
\begin{equation}\label{e:intro-cv-rho}
\lam^{-n}\Ll^n  g (y) \to c_g  \rho
\end{equation}
for some constant $c_g$ depending on $g$. By duality,
this is equivalent to the convergence, uniform on probability measures $\nu$,
\begin{equation}\label{e:intro-cv-m}
\lam^{-n} (\Ll^*)^{n} \nu  \to m_\phi,
\end{equation}
where $m_\phi$ is a conformal measure associated to the weight
$\phi$.  The equilibrium state $\mu_\phi$ is then given by
$\mu_\phi = \rho m_\phi$, and we have $c_g = \langle m_\phi, g \rangle$. 

\medskip

 We aim here at establishing
an exponential speed of convergence in \eqref{e:intro-cv-rho},
when $g$ satisfies necessary regularity properties.
This requires to build a suitable
(semi-)norm for (or equivalently, a suitable functional space on)
which the operator $\lam^{-1} \Ll$ 
turns out to be a contraction. 
 Observe that condition
{\bf (A)} 
is necessary for the uniform convergence.

Establishing the following statement is then our main goal in the current paper. 
 As far as we know, this is the first time that the existence
 of a spectral gap for the perturbed Perron-Frobenius operator
 is proved in this context even in dimension 1, except for hyperbolic
 endomorphisms or for weights with ad-hoc
 conditions (see for instance \cite{ruelle1992spectral,makarov2000thermodynamics}).
 This is one of the most desirable properties in dynamics.

\begin{theorem}\label{t:goal} 
Let $f,q,\phi,\rho, m_\phi$ be as in  Theorem \ref{t:main} and 
$\Ll,\lam$ the above Perron-Frobenius operator and scaling factor
associated to $\phi$.  
Let $A>0$ and $0<\Omega<\log d$ be two constants. Then, for every
constant $0<\gamma\leq 1$, there exist two explicit equivalent
norms for functions 
on $\P^k$: $\norm{\cdot}_{\diamond_1}$, depending on  $f,\gamma,q$ and
 independent of $\phi$, and $\norm{\cdot}_{\diamond_2}$, depending on $f,\phi,\gamma,q$,  such that
$$\norm{\cdot}_\infty+\norm{\cdot}_{\log^q}
\lesssim \norm{\cdot}_{\diamond_1}
\simeq \norm{\cdot}_{\diamond_2}
\lesssim \norm{\cdot}_{\Cc^\gamma}.$$
Moreover, there are positive constants $c=c(f,\gamma,q,A,\Omega)$ and 
$\beta =\beta(f,\gamma,q,A,\Omega)$ with $\beta<1$, both independent of $\phi$ and $n$, such that
when $\norm{\phi}_{\diamond_1}\leq A$ and $\Omega(\phi)\leq\Omega$ we have 
$$\|\lam^{-n}\Ll^n\|_{\diamond_1}\leq c, \quad \norm{\rho}_{\diamond_1}\leq c,
 \quad \norm{1/\rho}_{\diamond_1}\leq c, \quad \text{and} \quad 
\big\|\lam^{-1} \Ll g \big\|_{\diamond_2} \leq \beta \norm{g}_{\diamond_2}$$
for every function $g\colon\P^k\to\R$ with $\sca{m_\phi, g}=0$.
Furthermore, given any constant $0<\delta<d^{\gamma/ (2 \gamma+2)}$, when $A$ is small enough,
the norm  $\norm{\cdot}_{\diamond_2}$
 can be chosen so that we can take $\beta = 1/\delta$. 
 \end{theorem}

The construction of the norms $\norm{\cdot}_{\diamond_1}$ and 
$\norm{\cdot}_{\diamond_2}$ is quite 
involved. 
We use here ideas from the theory of interpolation between Banach
spaces \cite{triebel1995interpolation}
combined with techniques from pluripotential theory and complex dynamics.
Roughly speaking, an idea from 
 interpolation theory allows us to reduce the problem to the case where $\gamma=1$.
The definition of the above norms in this case requires a control of 
 the derivatives of $g$ (in the distributional sense), and this is where we use techniques from
 pluripotential theory. 
 This also explains why these norms are bounded by the $\Cc^1$ norm.
 Note that we should be able to bound the derivatives of $\Ll g$ in a similar way.
A quick expansion of the derivatives of $\Ll g$
using \eqref{e:L} gives an idea of the difficulties that one faces.
The existence of these norms is still surprising to us.

{Let us highlight two among these difficulties. First,
the objects from complex analysis and geometry are too rigid for perturbations with a non-constant weight: none 
of the operators
$f_*$, $d$, and
$dd^c$ commutes with 
the operator $\Ll$. 
In particular, the $dd^c$-method developed
by the second author and Sibony
(see for instance \cite{dinh2010dynamics}) 
cannot be applied in this context, even for
small perturbations of the weight $\phi=0$.
Moreover, 
we have critical points on the support of the measure, which
cause a loss in the regularity
of functions under the operators $f_*$ and $\Ll$
(see Section \ref{s:norm}). Notice that
 we do not assume that our potential degenerates at the critical points.
}

Our solution to these problems is to define 
a new invariant functional space and norm in this mixed real-complex setting, 
that we call the 
\emph{dynamical Sobolev space} and \emph{semi-norm}, taking into account both the regularity of the function
(to cope with the rigidity of the complex objects)
and the action of $f$
(to take into account the critical dynamics), see Definitions \ref{def:norm_sap} and  \ref{d:norm-pag}.
The construction of this norm requires the definition of several 
intermediate semi-norms and the precise study of the action of the operator $f_*$
with respect to them, and is carried out in Section \ref{s:norm}.
Some of the intermediate estimates
already give new or more precise convergence properties
for the operator $f_*$ and the equilibrium measure $\mu$,
 see for instance Theorem \ref{t:like-equi}.

\medskip

A spectral gap for the Perron-Frobenius operator and its perturbations is one of the
most desirable properties in dynamics. It allows us to obtain several
statistical properties of the equilibrium state. In the present setting, we have the following result,
see Appendix \ref{a:abstract-result} for the definitions.

\begin{theorem}\label{t:statistic}
Let $f,\phi, \mu_\phi, m_\phi, \norm{\cdot}_{\diamond_1}$
 be as in Theorems \ref{t:main} and \ref{t:goal},
  $\lam$ the scaling ratio associated to $\phi$, and
 assume that $\norm{\phi}_{\diamond_1}<\infty$.
Then the equilibrium state $\mu_\phi$ 
 is exponentially mixing
for observables with bounded 
$\norm{\cdot}_{\diamond_1}$
norm 
and the preimages of points by $f^n$ 
(suitably weighted)
equidistribute
 exponentially fast towards $m_\phi$ as $n$ goes to infinity. 
The measure $\mu_\phi$
satisfies the LDP
for all observables with finite $\norm{\cdot}_{\diamond_1}$
norm, the ASIP, CLT, almost sure CLT, LIL for  
all observables with finite $\norm{\cdot}_{\diamond_1}$
norm 
which are not coboundaries,
and the local CLT for all observables 
with finite $\norm{\cdot}_{\diamond_1}$ 
norm 
which are not $(\norm{\cdot}_{\diamond_1}, \phi)$-multiplicative cocycles.
Moreover,
 the pressure
$P(\phi)=\log \lam$
is analytic in the following sense: 
for $\|\psi\|_{\diamond_1}<\infty$
 and $t$ sufficiently small, the
  function $t \mapsto P(\phi+t\psi)$ 
is analytic.
 \end{theorem}

In particular, all the properties in Theorem \ref{t:statistic} hold when the weight $\phi$ and 
the observable are H\"older continuous and satisfy
 the necessary
coboundary/cocycles requirements.
In this assumption, some of the above properties were previously
obtained with ad hoc arguments, see
\cite{przytycki1989harmonic,denker1991ergodic,denker1991existence,denker1996transfer, haydn1999convergence, dinh2007thermodynamics, szostakiewicz2015fine}
 when $k=1$, 
 \cite{urbanski2013equilibrium,szostakiewicz2014stochastics}
for mixing, CLT, LIL
when $k\geq 1$,
and
\cite{dupont2010bernoulli}
 for the ASIP when 
 $k\geq 1$ and $\phi=0$. 
Our results are more general, sharper, and with
better error control.
Note that the LDP
and the local CLT are new even for $\phi=0$ (for all $k\geq 1$ and for all $k >1$, respectively).
Our method of proof of 
Theorem \ref{t:statistic} 
exploits the spectral gap established in Theorem \ref{t:goal}
and is based on
the theory of perturbed operators. This approach was first developed
by Nagaev \cite{nagaev1957some} in the context of Markov chains, 
see also
\cite{rousseau1983theoreme,broise1996transformations,gouezel2015limit},
 and provides a unified treatment for all
the statistical study.
Observe in particular that the fine control given for instance
by the local CLT is simply impossible to prove using weaker arguments, such as martingales, see, e.g.,
\cite[p.\ 163]{gouezel2015limit}.

For the reader's convenience, the above statistic properties
(exponential mixing, LDP, ASIP, CLT, almost sure CLT, LIL, local CLT) 
and the notions of coboundary and multiplicative cocycle
will be recalled in Section \ref{s:statistic} and Appendix \ref{a:abstract-result}
at the end of the paper.

\medskip\noindent
{\bf Outline of the organization of the paper.} 
In Section \ref{s:preliminary}, we
 introduce some notations and 
 establish comparison principles for currents and potentials that 
 will be the technical key 
in the construction of our norms.
   In Section \ref{s:norm}, we introduce the main (semi-)norms
 that we will need, and study the action of the operator $f_*$ with respect to these (semi-)norms.
Section \ref{s:speed} is dedicated to the proof of Theorem
 \ref{t:goal}.
Finally, in Section \ref{s:statistic}, we develop the
statistical study of the equilibrium states. This section contains 
the proof of Theorem \ref{t:statistic} and more precise statements.
In Appendix \ref{a:abstract-result},
we recall statistical properties and criteria in abstract settings that we 
use to prove results in Section \ref{s:statistic}.

\medskip\noindent 
\textbf{Acknowledgements.}
The first author would like to thank the National University
of Singapore (NUS) for its support and hospitality during the visits where this project started and developed, and 
Imperial College London were he was based during the first part of this work.
The paper was partially written during the visit of the authors to the University of Cologne. They would like
to thank this university, Alexander von Humboldt Foundation, and George Marinescu for their support and hospitality.

This project has received funding from the European Union’s Horizon 2020
Research and Innovation Programme under the Marie Skłodowska-Curie grant agreement No
796004,
 the French government through the Programme Investissement d'Avenir (I-SITE ULNE / ANR-16-IDEX-0004 ULNE and 
LabEx CEMPI /ANR-11-LABX-0007-01) managed by the Agence Nationale de la Recherche,
the CNRS  through the program PEPS JCJC 2019,
and the NUS 
and MOE
through the grants
R-146-000-248-114
and
MOE-T2EP20120-0010.

\section{Preliminaries and comparison principles}\label{s:preliminary}

\subsection{Some definitions}\label{ss:def}

We collect here some notions that we will use throughout the paper.

\begin{definition}
Given a subset $U$ of $\P^k$ or $\C^k$ and a real-valued function $g\colon U \to \R$,
define the \emph{oscillation} $\Omega_U (g)$ of $g$ as
\[\Omega_U (g) := \sup g - \inf g\]
and  its continuity modulus $m_U(g,r)$ at distance $r$ as
\[m_U(g,r) := \sup_{x,y\in U\colon \dist(x,y)\leq r} |g(x)-g(y)|.\]
We may drop the index $U$ when there is no possible confusion. 
\end{definition}

The following semi-norms  will be the
main building blocks
for all the semi-norms that we will construct  and study in the sequel.

\begin{definition}\label{defi_logp_cont}
The semi-norm  $\norm{\cdot}_{\log^p}$
is defined for every $p>0$ and $g\colon \P^k \to \R$ as
\[\norm{g}_{\log^p} := \sup_{a,b\in \P^k} |g(a)-g(b)| \cdot (\log^\star \dist(a,b))^p=
\sup_{r>0, a\in \P^k} \Omega_{\B_{\P^k}(a,r)}(g)\cdot (1+|\log r|)^p,\]
where $\B_{\P^k}(a,r)$ denotes the ball of center $a$ and radius $r$ in $\P^k$.
The definition can be extended to functions on any metric space.
\end{definition}

\begin{definition}
The norm $\norm{\cdot}_*$ 
of a $(1,1)$-current $R$ is given by
\[
\norm{R}_* := \inf \norm{S}
\]
where the infimum is taken over all positive
closed $(1,1)$-currents $S$ such that $\abs{R}\leq S$, see the
 Notation at the beginning of the paper.
When such a current $S$ does not exist, we put $\norm{R}_* :=+\infty$.
The semi-norm $\norm{\cdot}_*$ of
 an integrable function $g\colon \P^k \to \R$ is given by
\[\norm{g}_* :=
\norm{dd^c g}_*.\]
\end{definition}
Note that
when $R$ is a real closed $(1,1)$-current the above norm is equivalent
to the usual one defined as
\[
\norm{R}_* :=  \inf ( \|S^+\| + \|S^-\| )
\]
where  the infimum is taken over all positive
closed $(1,1)$-currents $S^{\pm}$ on $\P^k$ such that $R = S^+ - S^-$.
In particular, for $R=dd^c g$, the currents
 $S^+$ and $S^-$ are cohomologous and thus
  have the same mass, i.e., $\|S^+\|=\|S^-\|$,
see \cite[App.\ A.4]{dinh2010dynamics} for details.

In this paper,
we only consider continuous functions $g$. So the above semi-norms (and the others that we will
 introduce later) are almost norms: 
they only vanish when $g$ is constant. In particular, they are norms on the space of functions
$g$ satisfying $\langle \nu,g\rangle=0$ for some fixed probability measure $\nu$.
For convenience,
 we will 
use later 
 $\nu=m_\phi$ or $\nu= \mu_\phi$
to obtain a spectral gap for the Perron-Frobenius operator
and to study the statistical
properties of $\mu_\phi$.

\subsection{Approximations for H\"older continuous functions}
We will need the following
property for H\"older
continuous functions, see
 Definition \ref{d:norm-pag} and
Remark \ref{rem:holder-pag}.

\begin{lemma}\label{l:alpha-p-gamma-Holder}
Let $0<\gamma\leq 1$ be a constant. Then, for every $\Cc^\gamma$ function $g\colon \P^k \to \R$,
$s\geq 1$, and $0<\eps\leq 1$, there exist a $\Cc^s$ function $g_\eps^{(1)}$ and
a continuous function $g_\eps^{(2)}$ such that
$$g = g_\eps^{(1)} + g_\eps^{(2)}, \qquad  \|g_\eps^{(1)}\|_{\Cc^s} \leq 
c \norm{g}_{\infty}  (1/\eps)^{s/\gamma} \qquad \text{and} \qquad
\|g_\eps^{(2)}\|_{\infty} \leq c \norm{g}_{\Cc^\gamma} \eps,$$
where $c=c(\gamma,s)$  is a positive constant independent of $g$ and $\epsilon$.
\end{lemma}

\begin{proof}
Using a partition of unity, we can reduce the problem to the case where $g$ is
supported by the unit ball of an affine chart $\C^k\subset\P^k$.
Consider a smooth non-negative function $\chi$ with support in
 the unit ball of $\C^k$ whose integral with respect to the Lebesgue measure is 1. 
For $\nu>0$, consider the function $\chi_\nu(z):=\nu^{-2k}\chi(z/\nu)$
which has integral 1 and tends to the Dirac mass at 0 when $\nu$ tends to 0. 
Define an approximation of $g$ using the standard convolution operator $g_\nu:=g*\chi_\nu$,
and define $g_\eps^{(1)}:=g_\nu$ and $g_\eps^{(2)}:=g-g_\nu$.
We consider 
$\nu:=\epsilon^{1/\gamma}$.
It remains to bound $\|g^{(1)}_\eps\|_{\Cc^s}$ and $\|g^{(2)}_\eps\|_\infty$. 

By standard properties of the convolution we have, for some constant $\kappa>0$,
\[
\|g_\eps^{(2)}\|_{\infty} 
\lesssim m(g, \kappa\nu) \lesssim \norm{g}_{\Cc^\gamma} (\kappa \nu)^\gamma
\lesssim \norm{g}_{\Cc^\gamma} \epsilon.\]
and, by definition of $g_\nu$,
 \[
\|g^{(1)}_\eps\|_{\Cc^s}
 \lesssim \norm{g}_\infty \norm{\chi_\nu}_{\Cc^s} \Leb(\B_\nu^k)\lesssim
\norm{g}_\infty \nu^{-s} \lesssim \norm{g}_\infty 
\eps^{-s/\gamma}.
\]
The lemma follows.
\end{proof}

\subsection{Dynamical potentials}\label{ss:dyn-pot}
Let $T$ denote the Green $(1,1)$-current
of $f$. It is positive closed and of unit mass. Let $S$ be any positive closed
$(1,1)$-current of mass $m$ on $\P^k$. There is a unique function $u_S\colon \P^k \to \R \cup \{-\infty\}$ 
which is p.s.h.\ modulo $mT$
 and such that
\[S = mT + dd^c u_S \qquad \text{and} \qquad \langle \mu, u_S\rangle=0.\]
Locally, $u_S$ is the difference between 
a potential of $S$ and a potential of $mT$. We call it {\it the dynamical potential} of $S$.
Observe that the dynamical potential of $T$ is zero, i.e., $u_T=0$. 

Recall that 
$T$ has H\"older continuous potentials. 
So, $u_S$ is locally the difference
between a p.s.h.\ function and a H\"older continuous one.
We refer the reader to
\cite{dinh2010dynamics, bd-eq-states-part1} for details. 
In this paper, we 
only need currents $S$ such that $u_S$ is continuous.

\subsection{Complex Sobolev functions}
In
 our study,
 we will be naturally lead to
  consider currents of the form
$i\partial u \wedge \bar \partial u$. These currents
are always positive.
In this section, we
study the regularity of $u$ under 
 the assumption
  that
$i\partial u \wedge \bar \partial u \leq dd^c v$ for some $v$ of
given regularity.
 Recall that, given a smooth bounded open set $\Omega \subset \C^k$,
 the Sobolev space $W^{1,2}(\Omega)$
 is defined  as the space of functions $u\colon \Omega \to \R$
such that $\|u\|_{W^{1,2} (\Omega)} := \|u\|_{L^2 (\Omega)} + \|\partial u\|_{L^2 (\Omega)}<\infty$, 
where the reference measure is the standard Lebesgue measure on $\Omega$.
The Poincar\'e-Wirtinger's 
inequality implies 
that
$\norm{u}_{W^{1,2}(\Omega)}\lesssim \norm{\partial u}_{L^2 (\Omega)} +\big| \int_\Omega u\  d\Leb\big| $.
 We will 
 need the following 
lemmas.

\begin{lemma}\label{l:mt}
There is a 
universal positive constant $c$ 
such that
\[\int_K |u| d\Leb \leq c \Leb(K)(\log^\star\Leb(K))^{1/2}.\]
for every 
compact set $K \subset \D_1$
with $\Leb (K)>0$
and function $u\colon\D_1\to \R$ 
such that $\norm{u}_{W^{1,2} (\D_1)}\leq 1$.
\end{lemma}
\proof
By Trudinger-Moser's inequality \cite{moser1971sharp},
 there are positive constants $c_0$ and $\alpha$ such that 
$$\int_{\D_1} e^{2\alpha |u|^2} d\Leb \leq c_0.$$
Let $m$ denote the
restriction of the measure
$\Leb$ to $K$ multiplied by $1/\Leb(K)$. This is a probability measure. 
It follows from Cauchy-Schwarz's inequality that 
$$\int  e^{\alpha |u|^2} dm \leq \Big(\int  e^{2\alpha |u|^2} dm\Big)^{1/2} \lesssim \Leb(K)^{-1/2}.$$

Observe that the function $t\mapsto e^{\alpha t^2}$ is
convex on $\R^+$ and its inverse is the function $t\mapsto \alpha^{-1/2} (\log t)^{1/2}$.
By Jensen's inequality, we obtain
$$\int |u| dm \leq \alpha^{-1/2}\Big[\log \int  e^{\alpha |u|^2} dm\Big]^{1/2}\lesssim (\log^\star  \Leb(K))^{1/2}.$$
The lemma follows.
\endproof

 \begin{lemma}\label{l:u_chi}
Let $u\colon \D_2 \to \R$ 
be a continuous function 
 and
$\chi\colon \D_2 \to \R$
 a smooth function with compact support in $\D_2$ and 
equal to $1$ on $\overline\D_1$. 
Set $\chi_z := \partial \chi / \partial z$. Then we have, for all $0<r<s<1$,
\begin{equation}\label{eq_u_chi}
u(0)-u(r) = {1\over 2\pi} \Big\langle {i} \partial u, \chi(s^{-1}z) {r\over \overline z(\overline z-r)} d\overline z \Big\rangle
+ {1\over 2\pi} \Big \langle u,   \chi_z(s^{-1}z){r\over s \overline z (\overline z-r)} {i} dz\wedge d\overline z \Big\rangle.
\end{equation}
 \end{lemma}

\begin{proof}
Denote by $\delta_\xi$ the Dirac mass at $\xi\in\C$. 
Observe that 
\[{{i}\over 2\pi} \partial {d\overline z\over \overline z- \overline \xi} = \ddc \log|z-\xi|=\delta_\xi ,\]
where the equalities are
in the sense of currents on $\C$.
Hence, for $|\xi|< s$,
\[{{i}\over 2\pi} \partial \Big[{\chi(s^{-1}z)d\overline z\over \overline z-\overline \xi}\Big] = 
{\chi_z(s^{-1}z) {i} dz\wedge d \overline z\over 2\pi s (\overline z-\overline \xi)} +
\chi (s^{-1} z)\delta_\xi
=
{\chi_z(s^{-1}z) i dz\wedge d \overline z\over 2\pi s (\overline z-\overline \xi)} +\delta_\xi.
\]
Applying this identity for $\xi=0$ and $\xi=r$, and
 since $u(0)-u(r)=\langle u, \delta_0-\delta_r\rangle$, we obtain
\begin{eqnarray*}
u(0)-u(r) &= &
\frac{1}{2\pi} \Big\langle u, {i}\partial  \Big[\chi(s^{-1}z)\Big({1\over \overline z}-{1\over \overline z-r}\Big) d\overline z \Big]\Big\rangle
- \frac{1}{2\pi s} \Big\langle u,   \chi_z(s^{-1}z)\Big({1\over \overline z}-{1\over \overline z-r}\Big) {i}dz\wedge d\overline z \Big\rangle \\
& = & 
 \frac{1}{2\pi} \Big\langle {i} \partial u, \chi(s^{-1}z) {r\over \overline z(\overline z-r)} d\overline z \Big\rangle
+ \frac{1}{2\pi} \Big \langle u,   \chi_z(s^{-1}z){r\over s\overline z (\overline z-r)} {i} dz\wedge d\overline z \Big\rangle.
\end{eqnarray*}
The assertion is proved.
\end{proof}

The following is a main result in this section. It will be a crucial technical
tool in the construction of the norms with respect to which
the transfer operator has a spectral gap.

\begin{proposition}\label{p:mod-cont-mixed}
Let $u\colon \B_5^k\to \R$ 
be continuous and such that $\norm{\partial u}_{L^2(\B^k_5)}<\infty$. 
Assume that
 $i\partial u\wedge \dbar u\leq \ddc v$ where $v\colon \B_5^k \to \R$ 
 is  continuous, p.s.h., and
  such that 
\begin{equation}\label{e:hp-mixed}
\int_0^1 m_{\B_4^k}(v,t)  (\log \log^\star t)^4t^{-1} dt <+\infty.
\end{equation}
Then there is a positive constant $c$ such that, for all
$0<r\leq 1/2$, we have
\begin{align}\label{e:omega-mixed}
m_{\B_1^k}(u,r) & \leq c\Big(\int_0^{r^{1/2}} m_{\B_4^k}(v,t) (\log \log^\star t)^2t^{-1} dt\Big)^{1/2} \\ 
& \qquad  + c m_{\B_4^k}(v,r)^{1/3} \Omega_{\B_4^k}(v)^{1/6} (\log^\star r)^{1/2} 
+ c\Omega_{\B_4^k}(v)^{1/2}r^{1/2} (\log^\star r)^{1/2}.  \nonumber
\end{align}
\end{proposition}

\proof
Let $x,y \in \B_1^k$ be such that $\dist(x,y)\leq r\leq 1/2$. We need
to bound $|u(x)-u(y)|$ by the RHS of \eqref{e:omega-mixed}. We can assume without loss of generality that $\dist(x,y)=r$. 
By a change of coordinates and restricting
  to the complex line through $x$ and $y$
   we can reduce the problem to the case of dimension 1. More precisely, we can assume 
  that $x=0$ and $y=r$ in $\C$ and that $u,v$ are defined on $\D_4$.
  By subtracting constants, we can assume that $v(0)=0$ and $\int_{\D_3} u(z) idz\wedge d\overline z=0$. 
By multiplying $u$ and $v$ by suitable constants $\gamma$ and $\gamma^2$,
 we can assume that  $m_{\D_3}(v,1) = 1/8$, which implies that $|v|\leq 1$ on $\D_3$. 
In order to establish \eqref{e:omega-mixed} 
it is enough to show that $|u(0)-u(r)|$ is bounded by a constant times
\begin{equation} \label{e:good-bound}
\Big(\int_0^{r^{1/2}} m_{\D_3}(v,t)  (\log \log^\star t)^2t^{-1} dt\Big)^{1/2}+  m_{\D_3}(v,r)^{1/3}  (\log^\star r)^{1/2} 
+ r^{1/2} (\log^\star r)^{1/2}.
\end{equation}

Since $v$
is bounded, 
by Chern-Levine-Nirenberg's inequality \cite{chern1969intrinsic} the mass 
of $\ddc v$ on $\D_2$ is bounded by a constant.
Thus, by the hypotheses on $u$ and $v$, 
the 
$L^2$-norm of $\partial u$ on $\D_2$ is bounded by a constant and therefore,
by Poincar\'e-Wirtinger's inequality,
$\norm{u}_{W^{1,2}(\D_2)}$ is also bounded by a constant.

Fix a smooth function $0\leq \chi(z)\leq 1$ with compact support in $\D_2$ and
such that $\chi=1$ on $\overline\D_{1}$. Define $\chi_z:={\partial \chi /\partial z}$ and set
$$s:=r \min\big\{r^{-1/2}, m_{\D_3}(v,r)^{-1/3}\big\}.$$
We have
\begin{equation}\label{e:bound-s}
\sqrt{2} r\leq s \leq r^{1/2} < 1
\end{equation} 
because $m_{\D_3}(v,r)\leq m_{\D_3}(v,1) = 1/8$ and $0<r\leq 1/2$. 
The functions $u$ and $\chi$ satisfy the assumptions of
Lemma \ref{l:u_chi}. Thus, \eqref{eq_u_chi} holds for the above $s$ and $r$.
The second term in the 
RHS of \eqref{eq_u_chi}
 is an integral over $\D_{2s}\setminus \D_{s}$ because
$\chi_z$ has support in $\D_2\setminus \D_{1}$. 
Moreover, for $z \in \D_{2s}\setminus \D_{s}$, we have 
${r \over s \overline z(z-r)} = O(rs^{-3})$ because of \eqref{e:bound-s}. 
Using that 
$\norm{u}_{W^{1,2} (\D_2)}\lesssim 1$ and that $\Leb (\D_{2s})\lesssim s^{2}$,
Lemma \ref{l:mt}
implies that the considered term has modulus bounded by a constant times 
$$rs^{-1}(\log^\star s)^{1/2}\lesssim \max\big\{r^{1/2}, m_{\D_3}(v,r)^{1/3}\big\} (\log^\star r)^{1/2}.$$
The last expression is bounded by the sum in \eqref{e:good-bound}. 

In order to conclude, it remains to bound the first term in the RHS of \eqref{eq_u_chi}. 
Choose a smooth decreasing function $h(t)$ defined for $t>0$ and such that
$h(t):=(-\log t) (\log(-\log t))^2$ for $t$ small enough and $h(t)=1$
for $t$ large enough.  Define  $\eta(z):=  h(|z|)+h(|z-1|)$.
We will also use the function $\tilde v(z):= v(z)- r^{-1}v(r)\Re(z)$. 
This function satisfies $\ddc \tilde v=\ddc v$ and $\tilde v(0)=\tilde v(r)=0$. 
By Cauchy-Schwarz's inequality we have for the first term in the RHS of \eqref{eq_u_chi}
\[\Big|\Big\langle {i} \partial u, \chi(s^{-1}z) {r\over \overline z(\overline z-r)} d\overline z \Big\rangle\Big|^2
\leq 
\big\langle i\partial u \wedge \dbar u, \chi(s^{-1}z)\eta(r^{-1}z) \big\rangle
\int {\chi(s^{-1}z) \over \eta(r^{-1}z)}{r^2\over |z^2(z-r)^2|} idz\wedge d\overline z.\]
Using the change of variable $z\mapsto rz$, the fact that $0\leq\chi\leq 1$,
and the definition of $\eta$ we see that the last integral is bounded by 
$$\int _\C{ idz\wedge d\overline z \over \big[h(|z|)+h(|z-1|)\big] |z^2(z-1)^2|}\cdot$$
Using polar coordinates for $z$ and for $z-1$ and the definition of $h$
 it is not difficult to see that the last integral is finite.
Therefore, since $i\partial u\wedge\dbar u\leq \ddc v=\ddc\tilde v$, we get 
\[\Big|\Big\langle i\partial u, \chi(s^{-1}z) {r\over \overline z(\overline z-r)} d\overline z \Big\rangle\Big|^2 \lesssim 
\big\langle \ddc \tilde v , \chi(s^{-1}z)\eta(r^{-1}z) \big\rangle.\]

Define $\hat v(z):=\tilde v(sz)$. The RHS in the last expression is then equal to
\[\big\langle \ddc \hat v , \chi(z)\eta(r^{-1}sz) \big\rangle= 
\big\langle \ddc \hat v , \chi(z)h(r^{-1}s|z|) \big\rangle + \big\langle \ddc \hat v , \chi(z)h(|r^{-1}sz-1|) \big\rangle.\]
In order to conclude the proof of
 the proposition, it is enough to show that each term in the last sum is bounded by a constant times
\begin{equation}\label{e:mixed-step}
\int_0^{s} m_{\D_3}(v,t) (\log^\star |\log t|)^2t^{-1} dt + m_{\D_3}(v,r)^{2/3} \log^\star r.
\end{equation}

We will only consider the first term. The second term can be
treated in a similar way using the coordinate
$z':= z-rs^{-1}$. Since $h$ is decreasing, the first term we consider is bounded by
$$\big\langle \ddc \hat v , \chi(z)h(|z|) \big\rangle.$$

\begin{claim}
 We have 
\begin{equation}\label{e:claim_dwbd} 
\big\langle \ddc \hat v , \chi(z)h(|z|) \big\rangle 
= \int_{\D_2\setminus\{0\}} \hat v(z) \ddc [\chi(z) h(|z|)]. 
\end{equation}
\end{claim}

\medskip

We assume the claim for now and 
conclude the proof of the proposition.  Notice that the assumption
\eqref{e:hp-mixed} will be used in the proof of this claim.

Using the definitions of $h$ and $\chi$, 
 we can bound the RHS of \eqref{e:claim_dwbd}
 by a constant times
\[\begin{aligned}
\int_{\D_2\setminus\{0\}}  |\hat v(z) z^{-2}| (\log\log^\star |z|)^2 idz\wedge d\overline z
&
=   \int_{\D_{2s}\setminus\{0\}}  |\tilde v(z) z^{-2}| (\log\log^\star |z/s|)^2 idz\wedge d\overline z \\
& \lesssim  \int_{\D_{2s}\setminus\{0\}}  |\tilde v(z) z^{-2}| (\log\log^\star |z|)^2 idz\wedge d\overline z,
\end{aligned}
 \]
where we used the change of variable $z\mapsto sz$
and the fact that
$\log^\star |z/s|\lesssim \log^\star |z|$ for\break $0 < |z|< 2s<2$.
Moreover, by the definition of $\tilde v$ and using that $v(0)=\tilde v(0)=0$, we have for $|z|<2s$
\[|\tilde v(z)|\leq m_{\D_3}(v,|z|)+ r^{-1} |v(r)\Re(z)|\leq m_{\D_3}(v,|z|)+ m_{\D_3}(v,r)r^{-1}|z|.\]
Therefore, using 
polar coordinates, we see that the last integral is bounded by a constant times
\[\int_0^{2s}   m_{\D_3}(v,t) (\log\log^\star t)^2 t^{-1} dt + m_{\D_3}(v,r) r^{-1}s (\log \log^\star (2s))^2.\]
The first term in this sum is bounded by a constant times the integral in
\eqref{e:mixed-step}
because $m_{\D_3}(v,t') \leq 4 m_{\D_3}(v,t)$ for 
$s/2\leq t\leq s\leq t'\leq 2s$. The second one is bounded by a constant
times the second term in \eqref{e:mixed-step} by the definition of $s$
and \eqref{e:bound-s}.
The proposition follows.
\endproof

\begin{proof}[Proof of the claim]
Observe that $h(|z|)$ tends to infinity when $z$ tends to 0. 
Let $\vartheta:\R\to\R$ be a smooth increasing concave
function such that $\vartheta(t)=t$ for $t\leq 0$ and $\vartheta(t)=1$ for $t\geq 2$.
Define $\vartheta_n(t):=\vartheta(t-n)+n$. This is a sequence of
smooth functions increasing to the 
identity. 
Define $l(z):=\chi(z)h(|z|)$. 
Using an integration by parts, we see that the LHS of \eqref{e:claim_dwbd} 
 is equal to
\begin{eqnarray*}
\lim_{n\to\infty}  \big\langle \ddc \hat v , \vartheta_n(l(z)) \big\rangle  & = &   \lim_{n\to\infty}   \int_{\D_3} \hat v(z) 
\ddc \vartheta_n(l(z)) \\
& = &    \lim_{n\to\infty}   \int_{\D_3} \hat v(z) \vartheta_n'(l(z)) \ddc l(z) 
  +    \lim_{n\to\infty}   \int_{\D_3} \hat v(z) \vartheta_n''(l(z)) dl(z)\wedge d^cl(z).
\end{eqnarray*}
The first term in the last sum converges to the RHS of the identity in the claim using Lebesgue's dominated convergence theorem and \eqref{e:hp-mixed}. We need to show that the second term tends to 0.
Since $\chi(z)=1$ for $z$ near 0, for $n$ large enough, the considered term has an absolute value bounded by a constant times
\[\lim_{n\to\infty}   \int_{\{h(|z|)>n\}} |\hat v(z)|  i \partial h(|z|)\wedge \dbar h(|z|)\lesssim \lim_{n\to\infty}   \int_{\{h(|z|)>n\}} |\hat v(z)
  z^{-2}| (\log\log^\star |z|))^4 idz\wedge d\overline z.\]
Using the arguments as at the end of the proof of 
Proposition \ref{p:mod-cont-mixed} 
and the 
assumption \eqref{e:hp-mixed} on $v$,
 we see that the last integrand is an integrable function on $\D_1$.
Since the set $\{h(|z|)>n\}$ decreases to $\{0\}$ when
 $n$ tends to infinity, the last limit is zero according to Lebesgue's dominated convergence theorem. This ends the proof of the claim.
\end{proof}

\begin{corollary} \label{c:mod-cont-bis-mixed}
Let $S_0$ be a positive closed $(1,1)$-current on $\P^k$
 of unit mass,
whose dynamical potential $u_S$
satisfies $\norm{u_S}_{\log^p}\leq 1$ for some $p>3/2$.
 Let $\Fc(S_0)$ denote the set of all continuous
 functions $g\colon \P^k\to \R$ 
 such that $i\partial  g\wedge \bar \partial g \leq S_0$.
  Then for any positive number $q< \frac{p}{3}-\frac{1}{2}$
  we have $\norm{g}_{\log^{q}}\leq c$ for some positive constant $c=c(p,q)$
  independent of $S_0$. 
 In particular, 
the family   $\Fc(S_0)$ is equicontinuous.
\end{corollary}

\begin{proof}
Notice that
\eqref{e:hp-mixed} is satisfied for all $v$
such that $\norm{v}_{\log^p}<\infty$ for some $p>1$. It follows
that if $u$ and $v$ are as in
Proposition \ref{p:mod-cont-mixed} and $v$ is $\log^p$-continuous for some $p> 3/2$
then
$u$ is $\log^{q}$-continuous on $\B_1^k$
for all $q$ as in the statement, 
with $\norm{u}_{\B_1^k,\log^q} \leq c \norm{v}_{\B_5^k,\log^p}^{1/2}$
for some positive constant $c$ independent of $u,v$.
The result is thus deduced from Proposition \ref{p:mod-cont-mixed}
by means of a finite cover of $\P^k$.
\end{proof}

\section{Some semi-norms and equidistribution properties} \label{s:norm}

In this section,
we
consider the action of the operator
 $(f^n)_*$ on functions and currents. We 
 also introduce the
 semi-norms which are crucial in our study. Some results and ideas here are of independent interest.
 Recall that we always assume that $f$ satisfies the Assumption {\bf (A)} in the Introduction.

\subsection{Bounds with respect to the semi-norm $\norm{\cdot}_{\log^p}$}\label{subsection_logp_prelim}

In this section,
 we study the action of 
the operator $f_*$ on 
functions with bounded semi-norm $\norm{\cdot}_{\log^p}$.
We first prove that, with respect to this
semi-norm, the operator  $f_*$ is Lipschitz.

\begin{lemma}\label{l:log-Lipschitz}
For every constant $A>1$, there exists a positive constant $c=c(A)$
such that
for every $n\geq 0$, $p>0$, 
and
continuous 
function 
$g\colon \P^k \to \R$,
 we have
$$\big\| d^{-kn} (f^n)_* g\big\|_{\log^p}\leq c^p A^{pn} \norm{g}_{\log^p}.$$
\end{lemma}

\begin{proof}
We have
$$\big\| d^{-kn} (f^n)_* g\big\|_{\log^p} =
\sup_{x,y\in \P^k} d^{-kn}\abs{(f^n)_* g (x)- (f^n)_* g (y)} \cdot (\log^\star \dist(x,y))^p.$$
We need to bound the RHS by $c^p A^{pn} \norm{g}_{\log^p}$. 

Applying \cite[Cor.\,4.4]{dinh2010equidistribution} inductively to some iterate of $f$,  we see
that the Assumption {\bf (A)} implies:

\medskip\noindent
{\bf (A')} for every constant $\kappa>1$, there 
are an integer $n_\kappa \geq 0$ and 
a constant $c_\kappa>0$
independent of $n$
 such that for all $x,y\in\P^k$ and $n\geq n_\kappa$ 
 we can write
$f^{-n}(x)=\{x_1,\ldots, x_{d^{kn}}\}$ and $f^{-n}(y)=\{y_1,\ldots, y_{d^{kn}}\}$
(counting multiplicity) with the property 
that
$$\dist(x_j,y_j)\leq c_\kappa \dist(x,y)^{1/\kappa^n} \quad \text{for } \ j=1,\ldots, d^{kn}.$$

Fix $\kappa <A$.
We have, for $n\geq n_\kappa$,
\begin{eqnarray*}
d^{-kn} \abs{(f^n)_* g (x)- (f^n)_* g (y)} (\log^\star \dist(x,y))^p
&\leq &  \max_j \abs{g (x_j)- g (y_j)} (\log^\star \dist(x,y))^p\\
&\leq &
\max_j \frac{\norm{g}_{\log^p}}{(\log^\star \dist(x_j,y_j))^p} (\log^\star \dist(x,y))^p \\
&=& \max_j \Big({\log^\star \dist(x,y) \over \kappa^n \log^\star \dist(x_j,y_j)}\Big)^p\norm{g}_{\log^p}\kappa^{pn} .
\end{eqnarray*}

We need
 to check that the expression in the last parentheses is bounded by a constant. Fix a large constant $M>0$. 
Since $\log^\star \dist(x_j,y_j)$ is bounded from
below by 1,  when $\log^\star \dist(x,y)$ is bounded by $2M\kappa^n$
 the considered 
expression is bounded by some constant $c$ as desired. Assume
now that   $\log^\star \dist(x,y)\geq 2M\kappa^n$. Since $M$ is large, we deduce that 
$\log\dist(x,y) \leq -2M\kappa^n+1\leq -M\kappa^n$. Hence, by 
{\bf (A')},
since $M$ is large, we have 
$$\log\dist(x_j,y_j) \leq \log c_\kappa +\kappa^{-n}\log\dist(x,y)\leq {1\over 2} \kappa^{-n}\log\dist(x,y).$$
It is now clear that $\kappa^n \log^\star \dist(x_j,y_j)\geq {1\over 2} \log^\star \dist(x,y)$
which implies that the considered expression is bounded,
 as desired. This implies the lemma for $n\geq n_\kappa$.

As the multiplicity
of $f^n$ at a point is at most $d^{kn}$, we also have (see again \cite[Cor.\,4.4]{dinh2010equidistribution}):

\medskip\noindent
{\bf (A'')}
there is a constant $c_0>0$ such that for every $n\geq 0$,
for all $x,y\in\P^k$,
 we can write
$f^{-n}(x)=\{x_1,\ldots, x_{d^{kn}}\}$ and $f^{-n}(y)=\{y_1,\ldots, y_{d^{kn}}\}$
(counting multiplicity) with the property 
that
$$\dist(x_j,y_j)\leq c_0 \dist(x,y)^{1/d^{kn}} \quad \text{for } \ j=1,\ldots, d^{kn}.$$

 Hence, when $n\leq n_\kappa$, it is enough to use
 {\bf (A'')} instead of {\bf (A')}. Since 
$n_\kappa$ is fixed 
 it is clear that 
$\frac{\log^\star \dist(x,y)}{\log^\star \dist(x_j,y_j)} \lesssim d^{k n_\kappa}$,
which is bounded.
 The proof is complete.
\end{proof}

We will need the following result which is an improvement of 
\cite[Th.\,1.1]{dinh2010equidistribution}
in the case where $f$ satisfies the  Assumption {\bf (A)}.
By duality,
this result implies an exponential equidistribution
of $d^{-kn} (f^{n})^*\nu$
towards $\mu$ for every probability measure $\nu$.
 The assumption 
 {\bf (A)}
is
 necessary here
to get the estimate in the norm $\norm{\cdot}_\infty$. 

\begin{theorem}\label{t:like-equi}
Let $f$ be an endomorphism of $\P^k$ of algebraic degree $d\geq 2$
and satisfying the Assumption {\bf (A)}. Consider a real number $p >0$.
Let $g\colon \P^k \to \R$  be such that $\norm{dd^c g}_*\leq 1$,
$\langle \mu, g \rangle=0$ and $\norm{g}_{\log^p}\leq 1$. 
Then, for every constant $\eta>d^{-p/(p+1)}$, there is a positive
constant $c$ independent of $g$ such that for every $n\geq 0$
\[
\big\| d^{-kn} (f^n)_* g\big\|_\infty \leq c\eta^n.	
\]
\end{theorem}
\begin{proof}
Set $g_n:=d^{-kn} (f^n)_* g$. Recall (see, e.g.,
 \cite{skoda1972sous,dinh2010exponential})
that there exists a positive
constant $c_0$ 
independent of $g$ and $n$ such that
\begin{equation}\label{eq-unif-integral}
\int_{\P^k} e^{d^n |g_n|}  \leq c_0,
\end{equation}
where the integral above is taken with respect to the Lebesgue measure
associated to the volume form $\omega_\FS^k$ on $\P^k$.

Fix a constant $A>1$ such that  $\eta > \pa{A/d}^{p/(p+1)}$.
Suppose by contradiction
that for infinitely many $n$ there exists a point $a_n\in \P^k$ 
such that $|g_n (a_n)|\geq 3 \eta^n$ for some $g$ as above. Choose
 $r:=e^{- c_A A^n\eta^{-n/p}}$ with $c_A$ the constant given by
Lemma \ref{l:log-Lipschitz}
(we write $c_A$ instead of $c$ in order to
avoid confusion). 
By that lemma, when $\dist(z,a_n)<r$, we have
\[
|g_n(z)| \geq |g_n(a_n)| -|g_n(z)-g_n(a_n)| \geq 3\eta^n -  c_A^pA^{pn}(1+|\log r|)^{-p} \geq
\eta^n.
 \]
This implies that
\[
c_0\geq \int_{\P^k} e^{d^n |g_n|} 
\geq \int_{\dist(z,a_n)<r} e^{d^n |g_n(z)|}  
 \gtrsim r^{2k}  e^{d^n \eta^n} \gtrsim e^{-2kc_A A^n\eta^{-n/p}+d^n\eta^n}.
\] 
By the choice of $A$, the last expression diverges when $n$ tends to infinity. 
This is a  contradiction.
The theorem follows.
\end{proof}

\subsection{The semi-norm $\norm{\cdot}_p$}
In this section, we combine the semi-norms  $\norm{\cdot}_{\log^p}$ and $\norm{\cdot}_*$
to build a new semi-norm $\norm{\cdot}_p$
 and study its first properties. For our convenience, we will use dynamical potentials of currents, but this is avoidable. 

For every positive closed $(1,1)$-current $S$ on $\P^k$ we first define
$$\|S\|'_p:=\|S\|+\|u_S\|_{\log^p},$$
where $u_S$ is
the dynamical potential of $S$.
When 
$R$ is any $(1,1)$-current
we define
\[
\norm{R}_p = \min \norm{S}'_p,
\]
where the minimum is taken over all positive closed $(1,1)$-currents $S$
such that  
$\abs{R}\leq S$, and we set $\norm{R}_p:=\infty$ when no such $S$ exists. 
Finally, for all $g\colon \P^k \to \R$, define
$$\|g\|_p:=\|\ddc g\|_p.$$
The following lemma shows in particular that the norm $\norm{\cdot}_p$ is equivalent to the norm $\norm{\cdot}'_p$
when both are defined. We will thus 
just consider the norm $\norm{\cdot}_p$ 
 in the sequel.
 
\begin{lemma} \label{l:current-norm-p}
Let $S$ be a positive closed $(1,1)$-current on $\P^k$
and let $g\colon \P^k\to \R$ be a continuous function.
Then
\[\|u_S\|_p\leq 2\|S\|'_p, \qquad 
\norm{S}_p \leq \norm{S}'_p \leq c \norm{S}_p,
\qquad \text{and} 
\qquad
\norm{g}_{\log^p}\leq c \norm{g}_p
\]
for some positive constant $c=c(p)$ independent of $S$ and $g$. 
\end{lemma}

\proof
Define $m:= \|S\|$. We have
$\|S\|'_p\geq m$.
Since $u_T=0$, we have $\|T\|_p = \|T\|'_{p} =1$ and
$$\|u_S\|_p
=\|\ddc u_S\|_p
=\|S-mT\|_p\leq \|S\|_p+m\|T\|_p \leq \|S\|'_p+m\leq 2\|S\|'_p.$$
This proves the first assertion in the lemma.

We prove now the second assertion.
The first inequality is true by definition. For the second one,
it is enough
to prove that
if
 $\tilde S$ is also positive closed, and such that $S\leq \tilde S$,
 then $\norm{S}'_p\leq c \big\|\tilde S\big\|'_p$
for some constant $c$ independent of
 $S,\tilde S$.
It is clear that $\norm{S}\leq \big\|\tilde S\big\|$.
 So we can assume that $\big\|\tilde S\big\| =1$ and we only need to check that 
 $\|u_S\|_{\log^p}$
  is bounded by $c(1+\|u_{\tilde S}\|_{\log^p})$ for some constant $c$.

 We cover $\P^k$ with
 a finite family of open sets of the
form $\Phi_j(\B_{1/2}^k)$ where $\Phi_j$ is an injective holomorphic map from $\B_4^k$ to $\P^k$.
Let $h_j$ denote a potential of $T$ on $\Phi_j(\B_4^k)$. This is a H\"older continuous function.
By definition of dynamical potential, $v_j:=u_{\tilde S}+h_j$ is a potential
of $\tilde S$ on $\Phi_j(\B_4^k)$. 
By   \cite[Corollary 2.5]{bd-eq-states-part1},
 we have that $\Omega(u_S)$ is bounded by a constant.
 Taking into account the distortion
of the maps $\Phi_j$,
{a comparison principle}
\cite[Corollary 2.7]{bd-eq-states-part1} implies 
 that for all $r$ smaller than some constant $r_0>0$
\[
m_{\P^k}(u_S, r)
\lesssim \max_j m_{V_j} (v_j, c\sqrt r) + \sqrt{r} \lesssim (1+\|u_{\tilde S}\|_{\log^p})|\log r|^{-p}.
\]
This proves the second assertion.

Finally, let us consider the last inequality in the statement. By linearity, we can assume that
there exists a positive closed $(1,1)$-current $\tilde S$ of mass $\big\|\tilde S\big\|=1$
such that $\abs{dd^c g}\leq \tilde S$ and prove that $\norm{g}_{\log^p}\leq c (1+\big\|u_{\tilde S}\big\|_{\log^p})$
for some positive constant $c$ independent of $g$ and $\tilde S$. Observe that
$dd^c g + \tilde S$ is a positive closed current and $dd^c g + \tilde S \leq 2 \tilde S$. So we can apply the arguments
in the previous paragraph to $g+u_{\tilde S}, 2u_{\tilde S}$ instead of $u_S, u_{\tilde S}$. 
We obtain $\|g+u_{\tilde S}\|_{\log^p}\lesssim 1+\|u_{\tilde S}\|_{\log^p}$. 
This implies the last assertion, and completes
the proof of the lemma.
\endproof

The following lemma can be applied to a non-convex $\Cc^2$ function $\chi$ as 
it can be written as the difference of two convex functions. 
We can also apply it to $\chi$ Lipschitz and convex because such a function can be approximated by smooth convex functions with bounded first derivatives.
The statement for the smooth convex case that we need is however simpler.

\begin{lemma} \label{l:norm-p-compose}
There is a positive constant $c=c(p)$ such that for all
continuous
 functions $g,h\colon \P^k\to \R$ with finite $\norm{\cdot}_p$ 
 semi-norms and any $\Cc^2$ 
convex function $\chi:\R\to\R$ we have 
$$\norm{\partial g \wedge \bar \partial h}_p
\leq c (\Omega(g)  
 \|h\|_p  + \|g\|_p 
 \Omega (h))
\quad 
\text{and} 
\quad 
\|\chi(g)\|_p\leq c
\|\chi'(g)\|_\infty  \|g\|_p.
$$
\end{lemma}
\proof
We can write 
$\abs{dd^c g}\leq S$
with $\|S\|\lesssim 
\|g\|_p$ and $\|u_{S}\|_{\log^p}\lesssim \|g\|_p$.
We first prove the second inequality.  Set
$A:= \|\chi'(g)\|_\infty $.
We have
\[
\ddc \chi (g) =
\chi'(g) \ddc g + {1\over \pi} \chi''(g) i \partial g\wedge \dbar g 
= \Big[
\chi'(g)
dd^c g 
 +AS+{1\over \pi} \chi''(g) i \partial g\wedge \dbar g\Big]
 -[AS].
\]
Write the last expression as $R^+-R^-$ where $R^+$ (resp.\ $R^-$) is the expression 
in the
first (resp.\ second) brackets.
Using the definition of $A$, the 
inequality $\abs{\ddc g } \leq S$, the convexity of $\chi$,
and the fact that $i \partial g\wedge \dbar g$ is always positive,
we deduce that both $R^+$ and $R^-$ are positive currents.
Clearly, 
the current $R^-$ is closed and its mass is
 $\lesssim A\norm{g}_p$. 
The current $R^+$ is cohomologous to $R^-$ because $\ddc \chi (g)$ is an exact current.
It follows that $R^+$ is also a positive closed current of mass $\lesssim A\norm{g}_p$. 

We have $\|u_{R^-}\|_{\log^p}=A\|u_S\|_{\log^p}\lesssim A\norm{g}_p$. This and 
the above estimate for the mass
imply that  $\|R^-\|_p\lesssim A\norm{g}_p$. 
On the other hand, the above identities imply that $\chi(g)+u_{R^-}$
differs from $u_{R^+}$ by a constant.
We deduce that 
\[\|u_{R^+}\|_{\log^p} \leq \|u_{R^-}\|_{\log^p} + \|\chi(g)\|_{\log^p}  \lesssim \|u_{R^-}\|_{\log^p} 
+ A \|g\|_{\log^p}     \lesssim A\norm{g}_p.\]
Therefore, we also have $\|R^+\|_p\lesssim A\norm{g}_p$. It is now
clear that $\|\ddc\chi(g)\|_p\lesssim A\norm{g}_p$. Hence, we get the second assertion in the lemma.

For the first assertion, we first consider the case where $g=h$. We can replace $g$ by $g-\min g$ in order to
assume that $\min g=0$ and hence $\|g\|_\infty=\Omega(g)$.
The above computation gives
\[0\leq i \partial g \wedge \bar \partial g =\pi\big( \ddc g^2 -2g \ddc g \big) \lesssim \ddc g^2 + 2\|g\|_\infty S.\]
We thus have
\[\|i\partial g \wedge \bar \partial g\|_p
\lesssim \|\ddc g^2\|_p+\norm{g}_\infty\|S\|_p\lesssim \|\ddc g^2\|_p+\norm{g}_\infty \norm{g}_p.\]
We obtain the desired estimate by applying the second assertion in the lemma to $\|\ddc g^2\|_p$,
 using the function 
$\chi(t):=t^2$.

Finally, let us consider the first inequality for $g$ and $h$.
As above, we can assume that $\min h=0$ and hence $\|h\|_\infty=\Omega(h)$.
It follows from
Cauchy-Schwarz's 
inequality that 
\[
\abs{\partial g \wedge \bar \partial h \pm \partial h \wedge \bar \partial g}
\lesssim \frac{\norm{h}_{\infty}}{\norm{g}_{\infty}} i \partial g \wedge \bar \partial g +
\frac{\norm{g}_{\infty}}{\norm{h}_{\infty}} i \partial h \wedge \bar \partial h.
\]
The assertion thus follows from the particular case considered above.
This completes the proof of the lemma.
\endproof

\subsection{The dynamical norm $\norm{\cdot}_{\ap}$}  \label{ss:norm-alpha-p}
In this section, we define the main norms
$\norm{\cdot}_{\ap}$ for $(1,1)$-currents
that we will use to quantify the convergence \eqref{e:intro-cv-rho}. Based on the results
 in the previous sections,
 we will see later that  these 
 norms satisfy the
 inequalities
\[\norm{\cdot}_q \lesssim  \norm{\cdot}_{\ap} \lesssim \norm{\cdot}_p
\]
for some explicit
$q$ depending on $p,\alpha$, and $d$. 
In particular, the new norms are at the same
time weaker than the previous 
norm $\norm{\cdot}_p$, but still inherit
the main properties of a similar norm 
 $\norm{\cdot}_{q}$ which are obtained in the previous section.

\begin{definition}
Given a positive closed $(1,1)$-current $S$ on $\P^k$ and a real number $\alpha$ such that 
$d^{-1}\leq \alpha<1$, we define the current $S_\alpha$ by
\[
S_{\alpha} :=\sum_{n=0}^\infty  \alpha^n  
\frac{(f^n)_* (S)}{d^{(k-1)n}} \cdot
\]
For any $(1,1)$-current $R$ on $\P^k$ and real number $p>0$, we define 
\begin{equation}\label{e:def_ap_current}
\norm{R}_{\ap} := \inf \set{c\in\R \colon \exists S \mbox{ positive closed}
\colon \norm{S}_p \leq 1, |R|\leq c  S_\alpha }
\end{equation}
and we set $\norm{R}_{\ap} :=\infty$ if such a number $c$ does not exist. 
\end{definition}

Note 
 that when $\norm{R}_{\ap}$ is finite, by compactness, the infimum in 
\eqref{e:def_ap_current}
 is actually a minimum.
We have the following lemma where the assumption $d^{-1}\leq \alpha< d^{-1/(p+1)}$ is equivalent to  $0<q_0\leq p$.

\begin{lemma}\label{l:alpha-p-q}
Let $\alpha$ and $p$ be positive and such that $d^{-1}\leq \alpha<d^{-1/(p+1)}$. 
Then, for every $0<q<q_0 := {\abs{\log \alpha}\over \log d}(p+1)-1$, there are positive constants $c_1=c_1(p,\alpha)$ and $c_2=c_2(p,\alpha,q)$
such that, for every $(1,1)$-current $R$,
\[\norm{R}_{\ap}\leq c_1\norm{R}_p \qquad 
\mbox{ and } \qquad
\norm{R}_q\leq c_2\norm{R}_{\ap}.
\]
\end{lemma}

\begin{proof}
The first inequality holds by the definition of $\norm{\cdot}_\ap$ and Lemma \ref{l:current-norm-p}.
We prove the second  inequality. 
Consider a current $R$ such that $\|R\|_{\ap}=1$. We have to show that $\|R\|_q$ is bounded by a constant.

From the definition of $\norm{\cdot}_{\ap}$, we can find a
positive closed current $S$ such that $\|S\|_p=1$ and $|R|\leq S_\alpha$. 
By the definition of the norm $\norm{\cdot}_q$ and Lemma \ref{l:current-norm-p}
applied to $S_\alpha$,
  it is enough to show that $\|S_\alpha\|'_q$ is bounded. 
Denote by $u_\alpha$ the dynamical potential of $S_\alpha$. Since the mass of $S_\alpha$ is bounded,
we only need to show that $\|u_\alpha\|_{\log^q}$ is bounded. 
By definition of $S_\alpha$, we have 
\[
u_{\alpha} =\sum_{n=0}^\infty  \alpha^n  
\frac{(f^n)_* u_S}{d^{(k-1)n}} \cdot
\]
It follows that,  for every positive number $N$,
$$m(u_\alpha,r) \leq \sum_{n\leq N}  \alpha^n d^{-(k-1)n} \|(f^n)_* u_S\|_{\log^p} (\log^\star r)^{-p}+2\sum_{n>N}  \alpha^n d^{-(k-1)n} 
\norm{(f^n)_* u_S}_\infty.$$
Fix constants $A>1$ close enough to 1, $\eta>d^{-p/(p+1)}$ close
enough to $d^{-p/(p+1)}$ and $\alpha'>\alpha A^p$ close enough
to $\alpha$. In particular, by
the assumption on $\alpha$ and the choice of $\eta$, 
 we have
 that $\alpha d\eta$ is close to $\alpha d^{1/(p+1)}$ and smaller than 1.
By Lemma \ref{l:log-Lipschitz}
and Theorem \ref{t:like-equi}
we know that $\norm{(f^n)_* u_S}_{\log^p} \lesssim d^{kn} A^{pn}$ and 
$\norm{(f^n)_* u_S}_{\infty} \lesssim  d^{kn}\eta^n$.
This,  the above estimate on $m(u_\alpha,r)$, and the fact that $\alpha'd>\alpha d\geq 1$ imply that
\[
m (u_\alpha, r) \lesssim \sum_{n\leq N}
 \alpha^n d^n  A^{pn} (\log^\star r)^{-p} +  \sum_{n>N} \alpha^n d^n \eta^n \lesssim (\alpha'd)^N (\log^\star r)^{-p}  + (\alpha d \eta)^N.
\]
Finally,  choose $N =  \frac{p+1}{\log d} \log \log^\star r$. 
Observe that if we replace $\alpha'$ by $\alpha$ and  $\alpha d\eta$ by $\alpha d^{1/(p+1)}$, the last sum is equal to $2(\log^\star r)^{-q_0}$. So,
this sum is bounded by a constant times $(\log^\star r)^{-q}$ for $q<q_0$ because
$\alpha'$ is chosen close to $\alpha$ and $\alpha d\eta$ is close to $\alpha d^{1/(p+1)}$. This concludes the proof 
of the lemma.
\end{proof}

The following
shifting property of the norm 
$\norm{\cdot}_{\ap}$
is very useful when we work with the action of $f$, and is the key property that 
we need of
this norm.

 \begin{lemma}\label{l:f-alpha-p}
For every $n\geq 0$ and every 
$(1,1)$-current $R$
on $\P^k$, we have
\[\big\|d^{-kn} \pa{f^n}_* 
R\big\|_{\ap} \leq \frac{1}{d^n \alpha^n} 
\norm{
R}_{\ap}.\]
\end{lemma}
 
 \begin{proof}
 We can assume that $ \norm{R}_{\ap}=1$, so that there
 is a positive closed current $S$ with $\|S\|_p=1$ and
$ |R|\leq S_\alpha$, see the Notation at the beginning of the paper.
Consider any function $\xi\colon\P^k\to \C$ such that $|\xi|\leq 1$ and define $\xi_n:=\xi\circ f^n$.
Since $|\xi_n|\leq 1$, we have
\[
\Re \Big(\xi d^{-kn}\pa{f^n}_* R\Big)
=
\frac{1}{d^n}
\Re \Big(\frac{\pa{f^n}_* (\xi_n R)}{d^{(k-1)n}}\Big)
\leq
\frac{1}{d^n\alpha^n}
\sum_{j=0}^{\infty}
\alpha^{n+j} \frac{(f^n)_*}{d^{(k-1)n} } \frac{(f^j)_*}{d^{(k-1)j}}S  \leq \frac{1}{d^n\alpha^n} S_\alpha.
\]
The lemma follows.
\end{proof}

\subsection{The dynamical Sobolev semi-norm $\norm{\cdot}_{\sap}$}\label{ss:apmixed}
We can now define the 
first semi-norm
 for functions $g\colon \P^k \to \R$
with respect to which we 
will be able to prove the existence of
a spectral gap for the transfer operator. We can also define this norm for $1$-forms.

\begin{definition}\label{def:norm_sap}
Let $\alpha$ be a real number 
 such that 
$d^{-1}\leq \alpha<1$.
For any function $g\colon\P^k \to \R$ we set
\begin{equation*}
\norm{g}_{\sap} :=
\norm{i\partial g \wedge \bar \partial g}_{\ap}^{1/2}.
\end{equation*}
\end{definition}

The following two lemmas give
 the main 
properties of the semi-norm $\norm{\cdot}_{\sap}$
that we will need in Section 5, together with Lemma \ref{l:f-alpha-p}.
Recall that $q_0$ is defined in Lemma \ref{l:alpha-p-q}.
Note that the hypothesis $p>3/2$ ensures that $d^{-1}<d^{-5/(2p+2)}$ and the hypothesis
on $\alpha$ ensures that 
$q_0>3/2$, and hence that
$q_1$ is positive.

\begin{lemma}\label{l:cpt_apmixed}
Let $\alpha$ and $p$ be positive numbers 
such that $p>3/2$ and $d^{-1}\leq \alpha<d^{-5/(2p+2)}$. 
Then, for every $0<q<q_1:= \frac{q_0}{3} - \frac{1}{2}$,
 there are
 positive constants $c_1=c_1(p,\alpha,q)$ and $c_2= c_2(p,\alpha)$ 
such that for every $g\colon \P^k \to \R$ we have
\[
\norm{g}_{\log^q} \leq c_1 \norm{g}_{\sap},
\quad
\norm{g}_{\sap}\leq c_2 \norm{g}_p,
\quad
\mbox{ and }
\quad
\norm{g}_{\sap}\leq c_2 \norm{g}_{\Cc^1}.
\]
\end{lemma}

\begin{proof}
We can assume that $\norm{g}_{\sap}\leq 1$. By the definition
of the norm $\norm{\cdot}_{\sap}$ and Lemma \ref{l:alpha-p-q}, $\norm{i\partial g \wedge \bar \partial g}_{q'}$ 
is bounded by a constant for 
 any $q' <  q_0$. Therefore, we have 
 $i\partial g \wedge \bar \partial g\leq R$ for some positive closed current
$R$ such that $\norm{R}$ and 
$\norm{u_R}_{\log^{q'}}$ are bounded by a constant.
The first inequality follows from Corollary \ref{c:mod-cont-bis-mixed}. The second assertion follows from Lemmas \ref{l:alpha-p-q} 
and \ref{l:norm-p-compose}. The last assertion follows from Definition
\ref{def:norm_sap}.
\end{proof}

\begin{lemma}\label{l:ap_gh}
Let $\alpha$ and $p$ be positive numbers such that $d^{-1}\leq \alpha<1$. 
Then   for all
functions $g,h\colon\P^k\to\R$ we have
\[
\norm{gh}_{\sap} \leq 
\sqrt{2} \big( \norm{g}_{\sap} \norm{h}_\infty + \norm{g}_\infty
\norm{h}_{\sap} \big).
\]
\end{lemma}

\begin{proof}
Using an expansion of 
$i\partial (gh) \wedge \bar \partial (gh)$
and Cauchy-Schwarz's inequality,
 we have
\begin{eqnarray*}
\norm{i \partial (gh) \wedge \bar \partial (gh)}_{\ap}
 & \leq &
\norm{h}_\infty^2 \norm{i\partial g \wedge \bar \partial g}_{\ap} + \norm{g}_\infty^2 \norm{i\partial h \wedge \bar \partial h}_{\ap}\\
& &+ \norm{g}_\infty \norm{h}_\infty \norm{i\partial g \wedge \bar \partial h + i\partial h \wedge \bar \partial g}_{\ap}\\
 &\leq&
\norm{h}_\infty^2 \norm{i\partial g \wedge \bar \partial g}_{\ap} + \norm{g}_\infty^2 \norm{i\partial h \wedge \bar \partial h}_{\ap}\\
& &+ \norm{g}_\infty \norm{h}_\infty  \Big(  \frac{\norm{h}_\infty}{\norm{g}_\infty}\norm{i\partial g \wedge \bar \partial g}_{\ap} 
+\frac{\norm{g}_\infty}{\norm{h}_\infty} \norm{i\partial h \wedge \bar \partial h}_{\ap} \Big) \\
 &\leq &
2\norm{h}_\infty^2 \norm{i\partial g \wedge \bar \partial g}_{\ap} + 2\norm{g}_\infty^2 \norm{i\partial h \wedge \bar \partial h}_{\ap}.
\end{eqnarray*}
The assertion follows from Definition \ref{def:norm_sap}.
\end{proof}

\subsection{The semi-norm $\norm{\cdot}_{\sapg}$} \label{ss:alpha-p-gamma}

The following semi-norm defines the
final space of functions that we will use 
in our study of the transfer operator.
 We use here some ideas
 from the theory of interpolation between Banach spaces, see also \cite{triebel1995interpolation}.

 \begin{definition}\label{d:norm-pag}
For all real numbers $d^{-1}\leq \alpha <1$, $\gamma>0$ and $p>0$, we define for a continuous function $g\colon\P^k\to \R$
\begin{equation}\label{e:def:apg}
\begin{aligned}
\norm{g}_{\sapg} := \inf\Big\{ &c\geq 0\colon \forall \, 0<\eps  \leq {1} \ \exists\, g_\eps^{(1)}, g_\eps^{(2)}\colon \\
&g= g_\eps^{(1)} + g_\eps^{(2)},
\big\|g_\eps^{(1)}\big\|_{\sap}\leq c (1/\eps)^{1/\gamma},
 \big\|g_\eps^{(2)}\big\|_\infty \leq  c \eps
\Big\}.
\end{aligned}
\end{equation}
When such a number $c$ does not exist, we set $\norm{g}_{\sapg} :=\infty$. 
\end{definition}

\begin{remark}\label{rem:holder-pag}
Lemma 
\ref{l:alpha-p-gamma-Holder}
applied
for $s=1$  implies that 
$\norm{\cdot}_{\sapg}
\lesssim \norm{\cdot}_{\Cc^\gamma}$
because
$\norm{\cdot}_{\sap}
\lesssim \norm{\cdot}_{\Cc^1}$, see Lemma 
\ref{l:cpt_apmixed}.
\end{remark}

The following two 
lemmas are the counterparts of Lemmas
 \ref{l:cpt_apmixed} and \ref{l:ap_gh} 
for the semi-norm $\norm{\cdot}_{\sapg}$. Recall that $q_1$ is defined in Lemma \ref{l:cpt_apmixed}.

\begin{lemma}\label{l:logq-alpha-p-gamma}
For all positive numbers 
$p, \alpha,\gamma, q$
satisfying $p>3/2$, $d^{-1}\leq \alpha<d^{-5/(2p+2)}$ and
$q<q_2:= \frac{\gamma}{\gamma+ 1}q_1$, 
there is a positive constant $c =c (p,\alpha, \gamma, q)$ such that
\[
\norm{g}_{\log^q} \leq c \norm{g}_{\sapg} \quad \mbox{ and }
\quad \norm{g}_{\sapg} \leq  \norm{g}_{\sap}
\]
for every continuous function $g\colon \P^k \to \R$. Moreover, if $\chi:I\to\R$
is a Lipschitz function with Lipschitz constant $\kappa$ on an interval $I\subset\R$
containing the image 
of $g$, then we have 
$$\|\chi(g)\|_\sapg \leq \kappa\|g\|_\sapg.$$  
\end{lemma}

\begin{proof}
Let us prove the first inequality. We can assume that
$\norm{g}_{\sapg}\leq 1$.
Lemma \ref{l:cpt_apmixed}
 implies that 
$g_\eps^{(1)}$ has 
$\norm{\cdot}_{\log^{q'}}$
 semi-norm bounded by a constant times $(1/\eps)^{1/\gamma}$
when $q' < q_1$.
Therefore, we have for $r>0$
\[
m(g,r) \leq m(g_\eps^{(1)},r)+ m(g_\eps^{(2)},r)
\lesssim \frac{(1/\eps)^{1/\gamma}}{(\log^\star r)^{q'}} + \eps.
\]
Choosing  
$\eps = (\log^\star r)^{-K}$ with $K = q' / (1+ 1/\gamma)$
gives
\[
m(g,r)\lesssim (\log^\star r)^{-q' + K/\gamma}+ (\log^\star r)^{-K} = 2(\log^\star r)^{- q' / (1+ 1/\gamma)}.
\]
The first assertion of the lemma follows by choosing $q'$ close enough to $q_1$.

The second inequality follows from the definition of the semi-norm
$\norm{\cdot}_{\sapg}$, by taking
$g_{\eps}^{(2)}=0$ in the decomposition $g= g_\eps^{(1)} + g_\eps^{(2)}$ for every $\eps$.

We prove now the last assertion. Since we can approximate $\chi$ uniformly
by smooth functions $\chi_n$ with $|\chi_n'|\leq 1$, we can assume for simplicity
that $\chi$ is smooth.  Define $h:=\chi(g)$ and recall that we are
assuming that
$\norm{g}_{\sapg}\leq 1$. For every $0<\eps\leq 1$, we
have the decomposition
$$g= g_\eps^{(1)} + g_\eps^{(2)} \quad \text{with} \quad \big\|g_\eps^{(1)}\big\|_{\sap}\leq  (1/\eps)^{1/\gamma} \quad \text{and} \quad 
 \big\|g_\eps^{(2)}\big\|_\infty \leq   \eps.$$
Write
$$h= h_\eps^{(1)} + h_\eps^{(2)} \quad \text{with} \quad h_\eps^{(1)}:=\chi(g_\eps^{(1)}) \quad \text{and} \quad h_\eps^{(2)}:=h- h_\eps^{(1)}.$$
We have 
$$\|h_\eps^{(2)}\|_\infty = \|\chi(g)-\chi(g_\eps^{(1)})\|_\infty
\lesssim \kappa \|g-g_\eps^{(1)}\|_\infty = \kappa \|g_\eps^{(2)}\|_\infty \leq \kappa \eps$$
and also
$$\|h_\eps^{(1)}\|_\sap = \|i \partial \chi(g_\eps^{(1)})\wedge \dbar \chi(g_\eps^{(1)})\|_\ap^{1/2} \leq \kappa \|i \partial g_\eps^{(1)}\wedge \dbar g_\eps^{(1)}\|_\ap^{1/2} =\kappa \|g_\eps^{(1)}\|_\sap \leq \kappa  (1/\eps)^{1/\gamma}.$$
It follows that $\|h\|_\sapg \leq \kappa$. This completes the proof of the lemma.
\end{proof}

\begin{lemma}\label{l:apg_gh}
For all positive numbers
$p, \alpha,\gamma$ 
 such that $d^{-1}\leq \alpha<1$
we have
\[
\norm{gh }_{\sap, \gamma} \leq 3
 \big(
 \norm{g}_{\sap, \gamma}\norm{h}_\infty + \norm{g}_\infty \norm{h}_{\sap, \gamma}
 \big)
\]
for every continuous functions $g,h\colon \P^k \to \R$. 
\end{lemma}

\begin{proof}
 We can assume that 
 $\norm{g}_{\sap, \gamma} = \norm{h}_{\sap, \gamma}=1$.
For every $0<\eps\leq 1$, we
 need to find a decomposition $gh = L_\eps^{(1)} +  L^{(2)}_\eps$ with $L^{(1)}_\eps, L^{(2)}_\eps$
 such that
 $$\|L^{(1)}_\eps\|_{\sap} \leq 3(\norm{g}_\infty + \norm{h}_\infty) (1/\eps)^{1/\gamma} \quad 
 \text{and} \quad \|L^{(2)}_\eps\|_\infty \leq 3 (\norm{g}_\infty + \norm{h}_\infty) \eps.$$
  
 \noindent 
{\bf Case 1.} Assume that $\|g\|_\infty\leq 3\eps$ and $\|h\|_\infty\leq 3\eps$. Choose $L^{(1)}_\eps=0$ and $L^{(2)}_\eps=gh$. 
Clearly, these functions satisfy the desired estimates.

  \medskip\noindent 
{\bf Case 2.} Assume now that $\|g\|_\infty+\|h\|_\infty\geq 3\eps$. 
  By the definition of the semi-norm
   $\norm{\cdot}_{\sapg}$,
  we have the 
  decompositions $g = g^{(1)}_\eps + g^{(2)}_\eps$ and $h = h^{(1)}_\eps + h^{(2)}_\eps$ with
  \[
\|g^{(1)}_\eps\|_{\sap}\leq (1/\eps)^{1/\gamma},
\quad
 \|g^{(2)}_\eps\|_\infty \leq \epsilon,
 \quad
\|h^{(1)}_\eps\|_{\sap}\leq (1/\eps)^{1/\gamma},
\quad
 \|h^{(2)}_\eps\|_\infty \leq \epsilon.
  \]
 Observe that  $\big\|g_\eps^{(1)}\big\|_\infty \leq \norm{g}_\infty +\eps$ 
and $\big\|h_\eps^{(1)}\big\|_\infty \leq \norm{h}_\infty+\eps$,
 which imply that 
$$\big\|g_\eps^{(1)}\big\|_\infty+\big\|h_\eps^{(1)}\big\|_\infty\leq 2(\|g\|_\infty+\|h\|_\infty).$$
Set
 \[L^{(1)}_\eps := g^{(1)}_\eps h^{(1)}_\eps \qquad \text{and} \qquad 
 L^{(2)}_\eps := g^{(1)}_\eps h^{(2)}_\eps+g^{(2)}_\eps h^{(1)}_\eps+g^{(2)}_\eps h^{(2)}_\eps.\]
The desired estimate for $\|L^{(1)}_\eps\|_{\sap}$ follows
from Lemma \ref{l:ap_gh} and the one for $\|L^{(2)}_\eps\|_\infty$ is obtained by a direct computation.
This ends the proof of the lemma.
\end{proof}

\section{Spectral gap for the transfer operator}\label{s:speed}

In this section we prove our main Theorem \ref{t:goal}.
Theorem \ref{t:main}
and \eqref{e:intro-cv-rho}
 give
the scaling ratio $\lam$, the density function $\rho$ as an
eigenfunction for the operator $\Ll = \Ll_\phi$, and the probability measures
$m_\phi$ and $\mu_\phi$, all under the hypothesis that $\norm{\phi}_{\log^q}<\infty$ for some $q>2$.
The semi-norms  $\norm{\cdot}_{\sap}$ and  $\norm{\cdot}_{\sapg}$ were introduced in Sections 
\ref{ss:apmixed} and \ref{ss:alpha-p-gamma}, respectively.

\subsection{Some preliminary results}\label{ss:new-recall}

For positive real numbers $q,M$, and $\Omega$ with $q>2$ and $\Omega<\log d$, consider the following set of weights
$$\Pc(q,M,\Omega):=\big\{\phi\colon\P^k\to\R \ :   \	\|\phi\|_{\log^q}\leq M,\ \Omega(\phi)\leq \Omega\big\}$$  
and the uniform topology induced by the sup norm. Observe that this family is equicontinuous. 
 The following two lemmas
were obtained in \cite{bd-eq-states-part1}, see Lemmas 4.6 and 4.7 therein.
Since we will use them several times in this section, we restate them here.
We 
 use the index $\phi$ or parameter $\phi$ for objects which depend on $\phi$, e.g., we
write $\lambda_\phi, \Ll_\phi, \rho_\phi$
 instead of $\lambda, \Ll, \rho$.

\begin{lemma} \label{l:dependence}
Let $q,M$, and $\Omega$ be positive real numbers 
such that $q>2$ and $\Omega<\log d$.
The maps $\phi\mapsto\lambda_\phi$, $\phi\mapsto m_\phi$, $\phi\mapsto \mu_\phi$,
and $\phi\mapsto \rho_\phi$ are continuous on $\phi\in \Pc(q,M,\Omega)$ with
respect to the standard topology on $\R$, the
weak topology on measures, and the uniform topology on functions. 
In particular, $\rho_\phi$ is bounded from above and below by positive
constants which are independent of $\phi\in \Pc(q,M,\Omega)$.
Moreover, $\|\lambda_\phi^{-n}\Ll_\phi^n\|_\infty$ is bounded by a
constant which is independent of $n$ and of $\phi\in \Pc(q,M,\Omega)$.
\end{lemma}

\begin{lemma} \label{l:dependence-bis}
Let $q,M$, and $\Omega$ be positive real numbers 
such that $q>2$ and $\Omega<\log d$.
Let $\Fc$ be a uniformly bounded and equicontinuous family of
real-valued functions on $\P^k$. Then the family
$$\big\{\lambda_\phi^{-n} \Ll_\phi^n(g) \ : \  n\geq 0, \ \phi\in \Pc(p,M,\Omega), \ g\in\Fc\big\}$$
is equicontinuous. Moreover,
$\big\|\lambda_\phi^{-n} \Ll_\phi^n(g)-\langle m_\phi,g\rangle\big\|_\infty$
tends to $0$ uniformly on $\phi\in \Pc(p,M,\Omega)$ and $g\in\Fc$ when $n$ goes to infinity.
\end{lemma}

\subsection{Main result and first step of the proof.} 
The following is the main result of this section.
We will use it in order to prove Theorem \ref{t:goal}
with a suitable norm and another value of $\beta$.

\begin{theorem}\label{t:spectral_gap_apg}
Let $f,\phi, \lambda, m_\phi, \rho$ 
be as in Theorem \ref{t:main} and $\Ll$ the Perron-Frobenius operator associated to $\phi$.
Let $p, \alpha,\gamma, A,\Omega$ be positive constants
and $q_2$ as in Lemma \ref{l:logq-alpha-p-gamma} such that $p>3/2$, 
$d^{-1} \leq \alpha < d^{-5/(2p+2)}$,
 $\Omega < \log (d\alpha)$, and $q_2>2$.
Assume that $\norm{\phi}_{\sap,\gamma}\leq A$ and $\Omega (\phi)\leq \Omega$.
Then we have 
$$\|\lambda^{-n}\Ll^n\|_\sapg \leq c, \quad \norm{\rho}_{\sapg}\leq c, \quad \text{and} \quad
\norm{1/\rho}_{\sapg}\leq c$$
for some positive constant $c=c(p, \alpha,\gamma, A,\Omega)$
independent of $\phi$ and $n$. Moreover, for every constant $0<\beta<1$ there is 
a positive integer $N = N(p,\alpha,\gamma, A,\Omega, \beta)$ independent of $\phi$  such that
\begin{equation}\label{e:contraction_apg}
\big\|  \lam^{-N}{\Ll^{N}}g \big\|_{\sap,\gamma}
\leq
\beta \norm{g}_{\sap,\gamma}
\end{equation}
for every function $g\colon\P^k \to \R$ with $\sca{m_\phi, g}=0$.
Furthermore, we can take
$\beta=\delta^{-N}$ for every 
 given
constant $0<\delta < (d \alpha)^{\gamma/(2\gamma+2)}$, provided that $A$ is small enough.
\end{theorem}

Notice that 
Lemma \ref{l:logq-alpha-p-gamma} 
and the assumption $q_2>2$ imply that $\norm{\phi}_{\log^q}<\infty$ 
for some $q>2$. Hence, the scaling ratio $\lam$, the
density function $\rho$, and the measures $m_\phi$ and $\mu_\phi$ are well defined by Theorem \ref{t:main}.
Notice also that  $q_2>2$ implies  that the condition $\alpha<d^{-5/(2p+2)}$
is automatically satisfied.

\smallskip

The proof of Theorem \ref{t:spectral_gap_apg} will be reduced to
a comparison between suitable currents and their norms. 
Theorem \ref{t:spectral_gap_apg}
will then follow from some interpolation techniques (see Section \ref{s:proof:t:goal}).
A crucial estimate that we will need here is the following.

\begin{proposition}\label{p:estimate_n_gen}
Let $f$ be  as in Theorem \ref{t:main}. 
Take $0<\alpha<1$ and $p>0$.
Given $n$ functions $\phi^{(j)}\colon \P^k  \to \R$ for $j=1,\dots ,n$, set
$$\Phi_m :=\alpha^{-m} d^{(k-1)m} e^{\sum_{j=1}^m \max(\phi^{(j)})} \quad 
\text{and} \quad \Ll_{m,n}:= \Ll_{\phi^{(m)}}\circ \dots \circ \Ll_{\phi^{(n)}}.$$ 
Then there
exists a positive constant $c=c(p,\alpha)$, independent of $\phi^{(j)}$,  such that 
\begin{equation*}
\norm{\Ll_{1,n} g}_{\sap} \leq 
c \norm{\Ll_{1,n}\1}_\infty^{1/2} \Phi_n^{1/2} \norm{g}_{\sap}
  +c\sum_{m=1}^{n} \big\|\phi^{(m)}\big\|_{\sap} 
m^{3/2} \Phi_m^{1/2}  
\norm{\Ll_{1,m}\1}_{\infty}^{1/2}
\norm{\Ll_{m+1,n}g}_\infty  
\end{equation*}
for every function $g\colon \P^k \to \R$.
\end{proposition}

\proof
By Definition 
\ref{def:norm_sap}
of the semi-norm $\norm{\cdot}_{\sap}$,
we need to  bound the  current  $i\partial \Ll_{1,n} g \wedge \bar \partial \Ll_{1,n} g$.

Consider the product space 
$(\P^k)^{n+1}$ and denote by $(x_0, \dots, x_n)$ its elements.
Define the manifold $\Gamma_n \subset (\P^k)^{n+1}$ by
\[
\Gamma_n :=  \big\{(x,f(x), \dots , f^{n} (x)) : x \in \P^k \big\},
\]
which can also be seen as the graph of the map $(f,f^2, \dots , f^n)$
in the product space $(\P^k)^{n+1}$.
Denote by $\pi_n$ the restriction 
to $\Gamma_n$
of the projection
of $(\P^k)^{n+1}$ to its last component.

We have, using a direct computation,
$$\partial \big( e^{\phi^{(n)} (x_0) + \dots + \phi^{(1)}(x_{n-1})} g(x_0) \big) = \Theta_1+\Theta_2$$
with
$$\Theta_1:=\h'' \partial g(x_0), \quad \Theta_2:= \h'' g(x_0) \sum_{m=0}^{n-1} \partial \phi^{(n-m)} (x_m), \quad \text{and} 
\quad    \h'':=e^{\phi^{(n)} (x_0) + \dots + \phi^{(1)} (x_{n-1})}.$$ 
Using Cauchy-Schwarz's inequality, we obtain
\begin{eqnarray*}
i\partial \Ll_{1,n} g \wedge  \bar \partial \Ll_{1,n} g &=& i(\pi_n)_*(\Theta_1+\Theta_2)\wedge (\pi_n)_*(\overline\Theta_1+\overline\Theta_2) \\
&\leq& 2i (\pi_n)_*(\Theta_1)\wedge (\pi_n)_*(\overline \Theta_1) +  2i (\pi_n)_*(\Theta_2)\wedge (\pi_n)_*(\overline \Theta_2).
\end{eqnarray*}
We need to bound the norm
$\norm{\cdot}_\ap$
 of the two terms in the last sum by the square of the RHS of the inequality in the proposition.

For the first term, using again Cauchy-Schwarz's inequality, the definition of $\Phi_m$ as
in the statement,
and Lemma \ref{l:f-alpha-p}, we get
(notice that $(\pi_n)_*(\h'')=\Ll_{1,n}\1$)
\begin{eqnarray*}
\norm{i (\pi_n)_*(\Theta_1)\wedge (\pi_n)_*(\overline \Theta_1)}_{\ap} &\leq & \norm{  (\pi_n)_*(\h'')(\pi_n)_*\big(\h'' i \partial g(x_0)\wedge \dbar g(x_0)\big)}_\ap \\
&\leq & \norm{\Ll_{1,n}\1}_\infty  e^{\sum_{j=1}^n \max \phi^{(j)}} \norm{(f^n)_*  (i \partial g \wedge \bar \partial g )}_\ap \\
&\leq& \norm{\Ll_{1,n}\1}_\infty \Phi_n \norm{i\partial g \wedge \bar \partial g }_{\ap}.
\end{eqnarray*}
This gives the desired estimate for the first term.

For the second term, observe that $i (\pi_n)_*(\Theta_2)\wedge (\pi_n)_*(\overline \Theta_2)$ is equal to
\begin{eqnarray*}
\lefteqn{ \sum_{0\leq m,m'< n}    
i(\pi_n)_*\big(\h'' g(x_0)\partial \phi^{(n-m)}(x_m)\big)\wedge (\pi_n)_*\big(\h'' g(x_0)\dbar \phi^{(n-m')}(x_{m'})\big)}\\
&\leq&
2\sum_{0\leq m'\leq m< n}    
\big|i(\pi_n)_*\big(\h'' g(x_0)\partial \phi^{(n-m)}(x_m)\big)\wedge (\pi_n)_*\big(\h'' g(x_0)\dbar \phi^{(n-m')}(x_{m'})\big)\big|.
\end{eqnarray*}
Using Cauchy-Schwarz's inequality,
 we can
bound the current in the absolute value signs by 
$$(m-m'+1)^{-2} i(\pi_n)_*\big(\h'' g(x_0)\partial \phi^{(n-m)}(x_m)\big)\wedge (\pi_n)_*\big(\h'' g(x_0)\dbar \phi^{(n-m)}(x_m)\big)$$
$$\qquad \qquad +(m-m'+1)^{2} i(\pi_n)_*\big(\h'' g(x_0)\partial \phi^{(n-m')}(x_m')\big)\wedge (\pi_n)_*\big(\h'' g(x_0)\dbar \phi^{(n-m')}(x_m')\big)$$
and deduce that $i (\pi_n)_*(\Theta_2)\wedge (\pi_n)_*(\overline \Theta_2)$ is
bounded by a constant times
$$\sum_{0\leq m<n} (n-m)^3 
i(\pi_n)_*\big(\h'' g(x_0)\partial \phi^{(n-m)}(x_m)\big)\wedge (\pi_n)_*\big(\h'' g(x_0)\dbar \phi^{(n-m)}(x_m)\big).$$
Therefore, in order to get the proposition, 
 setting $\eta:=(\pi_n)_*\big(\h'' g(x_0)\partial \phi^{(n-m)}(x_m)\big)$
we only need to show that 
$$\norm{i\eta\wedge\overline\eta}_\ap \leq
\big\|\phi^{(n-m)}\big\|_{\sap}^2  \Phi_{n-m} \norm{\Ll_{1,n-m}\1 }_{\infty} \norm{\Ll_{n-m+1,n}g}_\infty^2.$$

 Consider the map $\pi':\Gamma_n\to(\P^k)^{n-m+1}$ defined by $\pi'(x):=x':=(x_m,\ldots,x_n)$.
Denote by $\Gamma'$ the image of $\Gamma_n$ by $\pi'$.
Consider also
  the map $\pi'':\Gamma'\to\P^k$ defined by $\pi''(x'):=x_n$.
Both
$\pi':\Gamma_n\to\Gamma'$
and 
 $\pi'':\Gamma'\to\P^k$ are
 ramified coverings,  respectively of degrees $d^{km}$ and $d^{k(n-m)}$, and
we have
  $\pi_n=\pi''\circ\pi'$.

 With the notation as above, we see that 
$$\eta = \pi''_*\big(\Ll_{n-m+1,n}g(x_m)\h_m\partial \phi^{(n-m)}(x_m)\big) \quad \text{with} \quad \h_m:=e^{\phi^{(n-m)} (x_m) + \dots + \phi^{(1)} (x_{n-1})}.$$
It follows from Cauchy-Schwarz's inequality  that
\begin{eqnarray*}
i\eta\wedge\overline\eta &\leq & \|\Ll_{n-m+1,n}g\|_\infty^2 \pi''_*(\h_m) \pi''_*\big(\h_m i \partial \phi^{(n-m)}(x_m)\wedge \dbar \phi^{(n-m)}(x_m)\big) \\
&\leq& \|\Ll_{n-m+1,n}g\|_\infty^2 \|\pi''_*(\h_m)\|_\infty \|\h_m\|_\infty (f^{n-m})_*\big(i\partial \phi^{(n-m)}\wedge \dbar \phi^{(n-m)}\big).
\end{eqnarray*}
Thus, by Lemma \ref{l:f-alpha-p} and the definition of $\pi''$,
we get
$$\|i\eta\wedge\overline\eta\|_\ap \leq  \|\Ll_{n-m+1,n}g\|_\infty^2 \|\Ll_{1,n-m}\1\|_\infty \Phi_{n-m} \|i\partial \phi^{(n-m)}\wedge \dbar \phi^{(n-m)}\|_\ap.$$
This ends the proof of the proposition.
\endproof


\subsection{Proof of Theorem \ref{t:spectral_gap_apg}} \label{ss:proofs_gap}

We will need the following elementary lemma.

\begin{lemma}\label{l:app-norm-infinity}
Let $\phi, \phi^{(j)},\psi\colon \P^k \to \R$ be 
such that $\sum_{j=1}^\infty  
 \norm{\phi-\phi^{(j)}}_\infty \leq a$ for some positive
 constant $a$.
 Define $\vartheta:\R^+\to
 \R^+$ by $\vartheta(t):=t^{-1}(e^{t}-1)$ and $\vartheta(0)=1$, which is a smooth increasing function. Then we have
\begin{enumerate}
\item[{\rm(i)}] \label{item_approx_1}  $\norm{\Ll_\phi  - \Ll_\psi}_\infty \leq \vartheta(\norm{\phi-\psi}_{\infty})
  \norm{\Ll_\phi}_\infty \norm{\phi-\psi}_{\infty}$;
 \item[{\rm(ii)}]  for every $n\geq 1$,
 $\big\|\Ll^n - \Ll_{\phi^{(1)}} \circ \dots  \circ \Ll_{\phi^{(n)}} \big\|_\infty \leq \vartheta(a) a \norm{\Ll^n}_\infty $.
 \end{enumerate}
\end{lemma}

\begin{proof}
Observe that for $\phi, \psi\colon \P^k \to \R$ we have, using the definition of $\Ll_\phi$ and $\Ll_\psi$,\
 \begin{equation*}
\norm{\Ll_\phi (g) - \Ll_\psi (g)}_\infty
 \leq \big\|(1- e^{\psi-\phi} )\|g\|_\infty \big\|_\infty \norm{\Ll_\phi (\1)}_\infty
= \big\|1- e^{\psi-\phi} \big\|_\infty \norm{\Ll_\phi}_\infty
\norm{g}_\infty.
\end{equation*}
The first item in the lemma
 follows.
 For the second item,   notice that
\[\norm{\pa{\phi + \phi \circ f + \dots + \phi \circ f^{n-1}}
-
\pa{
\phi^{(n)} + \phi^{(n-1)} \circ f + \dots + \phi^{(1)} \circ f^{n-1}
}
}_\infty
\leq
\sum_{j=1}^n \big\|\phi - \phi^{(j)}\big\|_\infty.
\]
Therefore, using this estimate and the expansions of
$\Ll^n(g)$ and $\Ll_{\phi^{(1)}} \circ \dots  \circ \Ll_{\phi^{(n)}}(g)$, we obtain the result in the same way as
the first item.
\end{proof}

We continue the proof of Theorem \ref{t:spectral_gap_apg}
We first prove the following result.

\begin{proposition} \label{p:G12}
Under the hypotheses of Theorem \ref{t:spectral_gap_apg}, there exists a positive
integer $N_0= N_0(p,\alpha,\gamma, A,\Omega, \beta)$ 
independent of $\phi$ and $g$
 and such that \eqref{e:contraction_apg}
holds for all $N\geq N_0$.
\end{proposition}

By subtracting from $\phi$
a constant,
we can assume that $\phi$ belongs to the family of weights
$$\Qc_0:=\big\{\phi\colon\P^k\to \R \ : \ \min \phi=0,\ \|\phi\|_{\sap,\gamma} \leq A,\ \Omega(\phi)\leq\Omega\big\}.$$
Observe that we can apply
 Lemmas \ref{l:dependence} and \ref{l:dependence-bis} because, by
Lemma \ref{l:logq-alpha-p-gamma}
and the assumptions on $\alpha$ and $p$,
 the family $\Qc_0$ is contained in $\Pc_0(q,M,\Omega)$
 for suitable $q>2$ and $M$. 
Observe also that
$\norm{\phi}_\infty =\Omega(\phi) \leq \Omega$ and $\norm{\phi}_\infty\lesssim \norm{\phi}_{{\sap},\gamma}\leq A$.

 Consider two constants $K\geq 1$ and $K'\geq 1$
whose values depend on $\beta$ and will be specialised later.
We will not fix $A$ but we assume $A\leq A_0$ for some fixed constant $A_0>0$.
For a large part of this section we can take $A=A_0$, but at the end of the
proof of Theorem \ref{t:spectral_gap_apg} we will consider $A\to 0$. This is the
reason why we will keep the constant $A$ in the estimates below. Note that
the constants hidden in the signs $\lesssim$ below are independent of the
parameters $A, \beta, K, K',n$ and also of 
the constant $0<\eps\leq 1$ and the integer $j$ that we consider now.

Since $\|\phi\|_\sapg\leq A$,  for every  $j\geq 1$
 there are functions 
 $\phi^{(j)}$ and $\psi^{(j)}$ such that
\[\phi = \phi^{(j)}+ \psi^{(j)}, \quad \big\|\phi^{(j)}\big\|_{\sap} 
\leq A  (Kj^2)^{1/\gamma}(1/\eps)^{1/\gamma},
 \quad \text{and} \quad  \big\|\psi^{(j)}\|_{\infty} \leq A K^{-1} j^{-2} \eps. \]
Observe that $\|\phi^{(j)}\|_\infty$ is bounded by a constant since 
$\|\phi\|_\infty$ is bounded by a constant.
 
We can assume for simplicity that $\norm{g}_{\sap,\gamma}\leq 1$,
which implies that $\Omega(g)$ is bounded by a constant.
Since $\langle m_\phi,g\rangle=0$ by hypothesis,
we deduce that $\|g\|_\infty$ is bounded by a constant. 
By the definition of the semi-norm $\norm{\cdot}_{\sap, \gamma}$,  we can find two
 functions $g^{(1)}_\eps$ and $g^{(2)}_\eps$ satisfying
 \begin{equation*}
 \begin{aligned}
 g= g^{(1)}_\eps + g^{(2)}_\eps, 
 \ \  
  \big\|g^{(1)}_\eps\big\|_{\sap}\leq K'^{1/\gamma} (1/\eps)^{1/\gamma},	
  \ \ 
  \big\|g^{(2)}_\eps\big\|_\infty\leq 2K'^{-1}\eps , \ \ 
 \langle{m_\phi, g_\eps^{(1)}\rangle}=\langle{m_\phi, g_\eps^{(2)}\rangle}=0.
  \end{aligned}
 \end{equation*}
Notice that without the condition $\langle{m_\phi, g_\eps^{(2)}\rangle}=0$
we would 
not need the coefficient 2 in the above estimate of $\|g^{(2)}_\eps\|_\infty$. 
We obtain this condition by adding to $g^{(2)}_\eps$ a suitable constant and
subtracting the same constant from $g^{(1)}_\eps$. The condition 
$\langle{m_\phi, g_\eps^{(1)}\rangle}=0$ is deduced from the hypothesis
$\langle{m_\phi, g\rangle}=0$ when we  have
$\langle{m_\phi, g_\eps^{(2)}\rangle}=0$. Since $\|g\|_\infty$ is bounded by a constant, 
$\|g^{(1)}_\eps\|_\infty$ is also bounded by a constant.

Define as above $\Ll_{m,n}:=\Ll_{\phi^{(m)}} \circ \dots \circ \Ll_{\phi^{(n)}}$ and write
\begin{equation}\label{e:def:g-abc}
 \lam^{-n} \Ll^n g = \lam^{-n} \Ll_{1,n}  g^{(1)}_\eps +  \lam^{-n}
\big(\Ll^n g^{(1)}_\epsilon - \Ll_{1,n}  g^{(1)}_\eps
\big) + \lam^{-n} \Ll^n g^{(2)}_\eps
 =: G^{(a)}_{n,\eps} + G^{(b)}_{n,\eps} +G^{(c)}_{n,\eps}.
 \end{equation}

\begin{lemma} \label{l:Gbc-beta}
When $K$ and $K'$ are large enough, we have for every $n\geq 1$
$$\big\|G^{(b)}_{n,\eps}\big\|_\infty \leq {1\over 2} \beta \eps
 \quad \text{and} \quad \big\|G^{(c)}_{n,\eps}\big\|_\infty \leq {1\over 2}\beta \eps.$$
\end{lemma}
\proof
The above estimate on $\|\psi^{(j)}\|_\infty$ 
implies that $\sum \|\psi^{(j)}\|_\infty \lesssim A K^{-1}\eps$. 
Therefore, using Lemma \ref{l:app-norm-infinity} and the fact that the sequence
$\lambda^{-n}\Ll^n\1$ is bounded uniformly on $n$ and $\phi$ (see Lemma \ref{l:dependence}), 
we get
$$\big\|G^{(b)}_{n,\eps}\big\|_\infty \lesssim AK^{-1}\eps \|g^{(1)}_\eps\|_\infty \lesssim AK^{-1}\eps \leq A_0K^{-1}\eps$$
because $\|g^{(1)}_\eps\|_\infty$ is bounded by a constant. So we get the first estimate in the lemma when $K$ is large enough
(depending on $A_0$ and $\beta$).

For the second estimate, using again that the sequence $\lambda^{-n}\Ll^n\1$ is uniformly bounded, we obtain
$$\big\|G^{(c)}_{n,\eps}\big\|_\infty \lesssim \|g^{(2)}_\eps\|_\infty \leq K'^{-1}\eps.$$
The result follows provided that $K'$ is large enough.
\endproof

\begin{lemma} \label{l:Ga-beta}
When $K\geq 1$ and $K'$ are fixed, there is a constant $0<\epsilon_0\leq 1$ independent of $\phi$ and $g$ such that,
for 
 all $0<\epsilon\leq \epsilon_0$
and all $n$ large enough, also independent of $\phi$ and $g$,
$$\big\|G^{(a)}_{n,\eps}\big\|_{\sap}\leq  \beta (1/\eps)^{1/\gamma}.$$
\end{lemma}
\proof
Fix an $\epsilon_0>0$ small enough.
We will apply Proposition \ref{p:estimate_n_gen}. First, using the estimates $\sum \|\psi^{(j)}\|_\infty \lesssim A K^{-1}\eps\leq A_0$ and $\lambda\geq d^ke^{\min\phi}$, we obtain
$$\Phi_m\lesssim \alpha^{-m}d^{(k-1)m} e^{m\max \phi} \lesssim  \alpha^{-m}d^{-m} \lambda^m e^{m\Omega(\phi)}  \leq  \alpha^{-m}d^{-m} \lambda^m e^{m\Omega}.$$
By Lemmas \ref{l:app-norm-infinity} and  \ref{l:dependence}, we have 
$$\|\Ll_{1,m}\1\|_\infty \lesssim \|\Ll^m\1\|_\infty\lesssim \lambda^m 
\quad \text{and} \quad
 \|\Ll_{m+1,n}\1\|_\infty \lesssim \|\Ll^{n-m} \1\|_\infty \lesssim \lambda^{n-m}$$
and also, again by Lemma \ref{l:app-norm-infinity},
\begin{eqnarray*}
\|\Ll_{m+1,n}g^{(1)}_\eps\|_\infty 
&
\leq & \|\Ll^{n-m}g^{(1)}_\eps\|_\infty + \|\Ll_{m+1,n}g^{(1)}_\eps-\Ll^{n-m}g^{(1)}_\eps\|_\infty \\
&\lesssim& \|\Ll^{n-m}g^{(1)}_\eps\|_\infty + \|\Ll^{n-m}\1\|_\infty \|g^{(1)}_\eps\|_\infty AK^{-1}\eps \\
&\lesssim& \|\Ll^{n-m}g\|_\infty + \|\Ll^{n-m}g^{(2)}_\eps\|_\infty  + \lambda^{n-m} AK^{-1}\eps \\
&\lesssim& \|\Ll^{n-m}g\|_\infty +  \lambda^{n-m} K'^{-1}\eps+  \lambda^{n-m} AK^{-1}\eps.
\end{eqnarray*}
This, Proposition \ref{p:estimate_n_gen},
and the estimates in the definitions of $g^{(1)}_\eps$ and $\phi^{(j)}$
 allow us to bound  
$\big\|G^{(a)}_{n,\eps}\big\|_{\sap}$ by a constant times
\begin{equation}\label{e:gane}
\Big[K'^{1/\gamma} \big({e^\Omega\over d\alpha}\big)^{n/2}   
+ \sum_{m=1}^{n} AK^{1/\gamma}m^{2/\gamma+3/2} \big({e^\Omega\over d\alpha}\big)^{m/2} \Big(\|\lambda^{-n+m}\Ll^{n-m}g\|_\infty + (K'^{-1}+AK^{-1})\eps_0\Big) \Big] \eps^{1/\gamma}.
\end{equation}

Recall that $g$ belongs to a uniformly bounded and equicontinuous family of functions, see Lemma \ref{l:logq-alpha-p-gamma}. 
It follows from 
 Lemma \ref{l:dependence-bis}
 that
$\|\lambda^{-n+m}\Ll^{n-m}g\|_\infty$ tends to 0, uniformly on $\phi$ and $g$, when $n-m$ tends to infinity.
This and  the fact that $e^\Omega< d\alpha$
 imply that, when $n$ tends to infinity, the sum between brackets in \eqref{e:gane} 
 converges to
$$\sum_{m=1}^\infty AK^{1/\gamma}m^{2/\gamma+3/2} \big({e^\Omega\over d\alpha}\big)^{m/2} (K'^{-1}+AK^{-1})\eps_0.$$
This sum is smaller than $\beta$ because $\epsilon_0$ is chosen small enough.
Therefore, we get the estimate in the lemma for $n$ large enough, independently of $\phi$ and $g$, because the last convergence is uniform on $\phi$ and $g$.
\endproof
 
\proof[Proof of Proposition \ref{p:G12}]
Take $N$ large enough, independent of $\phi$. It suffices to show that
we can write
$$\lam^{-N} \Ll^N g = G^{(1)}_{N,\eps} + G^{(2)}_{N,\eps} \quad 
\text{with} \quad \big\|G^{(1)}_{N,\eps}\big\|_{\sap}\leq  \beta (1/\eps)^{1/\gamma} \quad
\text{and} \quad \big\|G^{(2)}_{N,\eps}\big\|_\infty \leq \beta \eps.$$
We apply Lemmas \ref{l:Gbc-beta} and \ref{l:Ga-beta} to $n:=N$. When $\epsilon\leq \epsilon_0$, it
is enough to choose $G^{(1)}_{N,\eps}:=G^{(a)}_{N,\eps}$ and $G^{(2)}_{N,\eps}:=G^{(b)}_{N,\eps}+G^{(c)}_{N,\eps}$.
Assume now that $\epsilon_0\leq \epsilon\leq 1$ and choose $G^{(1)}_{N,\eps}:=0$
and $G^{(2)}_{N,\eps}:=\lambda^{-N}\Ll^Ng$. With $N$ large 
enough,
we have $\|G^{(2)}_{N,\eps}\|_\infty\leq \beta\eps_0 \leq \beta \eps$ 
because 
$\|\lambda^{-n}\Ll^ng\|_\infty$ tends to 0 uniformly on $\phi$ and $g$ when $n$
goes to infinity, see
 Lemma \ref{l:dependence-bis}. 
Thus, we have the desired decomposition of 
$\lam^{-N} \Ll^N g$ and hence the property \eqref{e:contraction_apg} for all
$N$ large enough.
\endproof

\begin{proposition} \label{p:rho-inverse}
Under the hypotheses of Theorem \ref{t:spectral_gap_apg}, there is
a positive constant $c= c(p,\alpha,\gamma, A,\Omega)$ independent of $\phi$ and $n$ such that
$$\|\lambda^{-n}\Ll^n\|_\sapg \leq c, \quad \norm{\rho}_{\sapg}\leq c, \quad \text{and} \quad
\norm{1/\rho}_{\sapg}\leq c.$$
\end{proposition}
\proof
We prove the first inequality. We will use the above computations for $K=K'=\epsilon_0=1$. 
Consider any function $g\colon\P^k\to\R$ such that $\|g\|_\sapg\leq 1$. We do
not assume  that $\langle m_\phi,g\rangle=0$. As before,  for any $0<\eps\leq 1$ we can write 
$$g= g^{(1)}_\eps + g^{(2)}_\eps \quad  \text{with} \quad 
  \big\|g^{(1)}_\eps\big\|_{\sap}\leq  (1/\eps)^{1/\gamma} \quad \text{and} \quad
  \big\|g^{(2)}_\eps\big\|_\infty\leq \eps.$$
We also consider as above the decomposition 
$$\lam^{-n} \Ll^n g = G^{(1)}_{n,\eps} + G^{(2)}_{n,\eps} \quad \text{with} \quad
G_{n,\epsilon}^{(1)}:=G_{n,\epsilon}^{(a)} \quad  \text{and} \quad G_{n,\epsilon}^{(2)}:=G_{n,\epsilon}^{(b)}+G_{n,\epsilon}^{(c)},$$
see \eqref{e:def:g-abc}. 
The computations in Lemmas \ref{l:Gbc-beta} and \ref{l:Ga-beta} give 
$\|G_{n,\epsilon}^{(2)}\|_\infty\lesssim \epsilon$ and 
$\|G_{n,\epsilon}^{(1)}\|_\sap\lesssim \epsilon^{1/\gamma}$. 
We use here the fact that $\|\lambda^{-n+m}\Ll^{n-m}g\|_\infty$
is bounded by a constant,
see 
 Lemma \ref{l:dependence-bis}. 
Therefore, $\|\lambda^{-n}\Ll^n g\|_\sapg$ is bounded by a constant.
Thus, the first inequality in the proposition holds.

Consider now the second inequality. Observe that 
$$\rho=\lim_{n\to\infty} \lambda^{-n}\Ll^n \1 = \1+\sum_{n=0}^\infty \lambda^{-n}\Ll^n g \quad \text{with} \quad g:=\lambda^{-1}\Ll\1-\1.$$
The 
$\lam^{-1}\Ll^*$-invariance of $m_\phi$ implies that $\langle m_\phi,g\rangle=0$. Therefore,
Proposition \ref{p:G12} and the first inequality in the present proposition imply
that $\|\lambda^{-n}\Ll^n g\|_\sapg\lesssim \beta^{n/N}$. We deduce that $\|\rho\|_\sapg$ is bounded by a constant.

For the last inequality in the lemma, observe that $\rho$ is bounded from above and
below by positive constants which are independent of $\phi$, see
 Lemma \ref{l:dependence}.
The result is then a consequence of Lemma \ref{l:logq-alpha-p-gamma} applied to the function $\chi(t):=1/t$. 
The proof is complete.
\endproof 
 
\proof[End of the proof of Theorem \ref{t:spectral_gap_apg}] 
By Propositions \ref{p:G12} and \ref{p:rho-inverse}, it only remains to prove the last assertion in this theorem.
We continue to use the computations in Lemmas \ref{l:Gbc-beta}
and \ref{l:Ga-beta} and take $K=1$, $K'=\delta'^{N}$, and
$\epsilon_0=1$ for some 
$\delta'$ and $d'$ such that $\delta<\delta'=d'^{\gamma/(2\gamma+2)}$ and $d'<d\alpha$.
As above, we consider the decomposition
$$\lam^{-N} \Ll^N g = G^{(1)}_{N,\eps} + G^{(2)}_{N,\eps} \quad \text{with} \quad
G_{N,\epsilon}^{(1)}:=G_{N,\epsilon}^{(a)} \quad  \text{and} \quad G_{N,\epsilon}^{(2)}:=G_{N,\epsilon}^{(b)}+G_{N,\epsilon}^{(c)}.$$
Take $A\to 0$, which also implies that $\Omega(\phi)\to 0$. So we can fix an $\Omega$ as small as needed
and assume that 
$d' < e^{-\Omega}d \alpha$.
Then, the estimates in Lemmas \ref{l:Gbc-beta} and \ref{l:Ga-beta} give
$$\|G^{(1)}_{N,\eps}\|_\sap \leq c \big(K'^{1/\gamma} d'^{-N/2} +O(A)\big)\eps^{1/\gamma}
 \quad \text{and} \quad 
\|G^{(2)}_{N,\eps}\|_\infty \leq c(K'^{-1}+O(A))\epsilon,$$
where $c$ is a positive constant. With 
$N$ large enough and $A$ small enough, since $\delta<\delta'$ and $K'=\delta'^{N}$, we get
$$\|G^{(2)}_{N,\eps}\|_\infty \leq \delta^{-N}\epsilon \quad \text{and} \quad 
\|G^{(1)}_{N,\eps}\|_\sap \leq \delta^{-N}\eps^{1/\gamma}.$$
In other words, we can take $\beta=\delta^{-N}$. This completes the proof of the theorem.
\endproof


\subsection{Proof of Theorem \ref{t:goal}}\label{s:proof:t:goal}

The statement is a consequence of Theorem \ref{t:spectral_gap_apg}, namely, 
of the estimate \eqref{e:contraction_apg}. Note that the constant $\beta$ in Theorem
\ref{t:goal} is not the one in \eqref{e:contraction_apg}. 
Given $\phi$ as in the statement, we first choose 
$\alpha$ sufficiently close to $1$ so
 that $\Omega (\phi)< \log (\alpha d)$. Then, we choose
$p$ large enough  so that $q< q_2$,
where $q$ and $\gamma$ are as in the statement and $q_2$
is defined in Lemma \ref{l:logq-alpha-p-gamma}
(this also implies that $\alpha < d^{-5/(2p+2)}$ since $q>2$).
Recall that the semi-norm $\norm{\cdot}_\sapg$ is almost a norm. Define
$$\norm{\cdot}_{\diamond_1}:=\norm{\cdot}_\infty+\norm{\cdot}_\sapg.$$
This is now a norm, which is independent of $\phi$.
By Lemmas \ref{l:logq-alpha-p-gamma} and \ref{l:alpha-p-gamma-Holder}, we have
$$\norm{\cdot}_\infty+\norm{\cdot}_{\log^q}
\lesssim \norm{\cdot}_{\diamond_1} \lesssim \norm{\cdot}_{\Cc^\gamma}.$$
By Lemma \ref{l:dependence}, the quantities $\|\lam^{-n}\Ll^n\|_\infty$, $\|\rho\|_\infty$, and $\|1/\rho\|_\infty$
are bounded by a constant when $\norm{\phi}_{\diamond_1}\leq A$. 
By Theorem \ref{t:spectral_gap_apg}, 
$\|\lam^{-n}\Ll^n\|_\sapg$, $\|\rho\|_\sapg$, and $\|1/\rho\|_\sapg$ are also
bounded by a constant. We deduce that $\|\lam^{-n}\Ll^n\|_{\diamond_1}$, $\|\rho\|_{\diamond_1}$,
and $\|1/\rho\|_{\diamond_1}$ satisfy the same property. 

Let $N$ and $\beta_0$ be as in Theorem \ref{t:spectral_gap_apg}
(we write $\beta_0$ instead of $\beta$
to distinguish it from 
the constant that we use now for Theorem \ref{t:goal}). Fix
a constant $\beta$ such that  $\beta_0^{1/N}<\beta<1$ and consider the following norms
$$\norm{g}_\diamond:=|c_g|+\norm{g'}_\sapg \quad \text{and} \quad 
\norm{g}_{\diamond_2}:= |c_g|+ \sum_{n=0}^{\infty} \beta^{-n}\norm{\lam^{-n}\Ll^n g'}_\sapg$$
for every function $g\colon\P^k\to\R$, where $c_g:=\langle m_\phi,g\rangle$ and $g':=g-c_g\rho$. 

\begin{lemma} \label{l:diamond}
We have $\|gh\|_{\diamond_1}\leq 3 \|g\|_{\diamond_1}\|h\|_{\diamond_1}$ for all functions $g,h\colon\P^k\to\R$. Moreover, 
both of the 
norms $\norm{\cdot}_{\diamond}$ and $\norm{\cdot}_{\diamond_2}$ are equivalent to $\norm{\cdot}_{\diamond_1}$. 
\end{lemma}
\proof
The first assertion is a direct consequence of Lemma \ref{l:apg_gh}. We prove now the second assertion.
Since $m_\phi$ is a probability measure
and $\norm{\rho}_{\sapg}$ is bounded, we have 
$$\norm{g}_{\diamond}= |c_g|+\|g-c_g\rho\|_\sapg  \leq |c_g| + \|g\|_\sapg+ |c_g| \|\rho\|_\sapg \lesssim \|g\|_\infty+\|g\|_\sapg =\|g\|_{\diamond_1}.$$ 
Conversely, assume that 
 $\norm{g}_{\diamond}\leq 1$, then $|c_g|\leq 1$ and $\norm{g'}_\sapg\leq 1$. It follows that
 $\norm{g}_\sapg$ is bounded by a constant because it is bounded by $\norm{g'}_\sapg+|c_g|\norm{\rho}_\sapg$.
 By Lemma \ref{l:logq-alpha-p-gamma}, $\Omega(g)$ is also bounded by a constant. This and
  the inequality $|\langle m_\phi,g\rangle|= |c_g|\leq 1$ 
  imply that $\|g\|_\infty$ is bounded by a constant. 
  We deduce that $\norm{\cdot}_{\diamond}$ is equivalent to $\norm{\cdot}_{\diamond_1}$.

Observe that $\norm{\cdot}_{\diamond}\leq \norm{\cdot}_{\diamond_2}$. 
To complete the proof, it is enough to show that $\|g\|_{\diamond_2}\lesssim \|g\|_{\diamond}$ for every function $g$. 
Recall that $\rho$ is invariant by $\lambda^{-1}\Ll$ and $\langle m_\phi,\rho\rangle=1$.
Therefore, we have $\langle m_\phi,g'\rangle =0$. Theorem \ref{t:spectral_gap_apg} and Proposition \ref{p:rho-inverse} imply
that $\|\lambda^{-n}\Ll^n g'\|_\sapg\lesssim \beta_0^{n/N}\|g'\|_\sapg$.
 Hence
$$\|g\|_{\diamond_2} \lesssim |c_g| + \|g'\|_\sapg\sum_{n=0}^\infty\beta^{-n}\beta_0^{n/N} \lesssim |c_g| + \|g'\|_\sapg =\|g\|_{\diamond}.$$
The last infinite sum is finite because $\beta>\beta_0^{1/N}$. This ends the proof of the lemma.
\endproof

Consider now a function $g$ with $c_g=\langle m_\phi,g\rangle=0$,
which implies $g=g'$. We also have $\langle m_\phi,\lambda^{-1}\Ll g\rangle =0$ because $m_\phi$ is invariant by $\lambda^{-1}\Ll^*$. 
From the definition of $\norm{\cdot}_{\diamond_2}$, we get
\[
\norm{\lam^{-1}\Ll g }_{\diamond_2}
 = \beta  ( \norm{g}_{\diamond_2}  - \norm{g}_{\diamond}  )
\leq \beta \norm{g}_{\diamond_2}.
\]
This is the desired contraction. Finally, the last assertion in Theorem \ref{t:goal} is a direct consequence of the last assertion
in Theorem \ref{t:spectral_gap_apg} by taking $\alpha$ close enough to 1. The proof of Theorem \ref{t:goal} is now complete.

\subsection{Spectral gap in the limit case}\label{subsection_contraction}
The semi-norm $\norm{\cdot}_\sap$ can be seen as the limit of the semi-norm
$\norm{\cdot}_\sapg$ as $\gamma$ goes to infinity.
In order to complete our study, we will prove here a spectral gap with respect to this limit norm.
The following is an analogue of Theorem \ref{t:spectral_gap_apg}.

\begin{theorem}\label{t:spectral_gap_ap}
Let $f,\phi, \lambda, m_\phi, \rho$ 
be as in Theorem \ref{t:main} and $\Ll$ the Perron-Frobenius operator associated to $\phi$.
Let $p, \alpha,A,\Omega$ be positive constants and $q_1$ as in Lemma \ref{l:cpt_apmixed} such
that $p>3/2$, $d^{-1} \leq \alpha < d^{-5/(2p+2)}$,
 $\Omega < \log (d\alpha)$, and $q_1>2$.
Assume that $\norm{\phi}_{\sap}\leq A$ and $\Omega (\phi)\leq \Omega$.
Then we have 
$$\|\lambda^{-n}\Ll^n\|_\sap \leq c, \quad \norm{\rho}_{\sap}\leq c, \quad \text{and} \quad
\norm{1/\rho}_{\sap}\leq c$$
for some positive constant $c=c(p, \alpha, A,\Omega)$ independent of $\phi$ and $n$. Moreover, for every constant $0<\beta<1$ there is 
a positive integer $N = N(p,\alpha, A,\Omega, \beta)$ independent of $\phi$  such that
\begin{equation*}
\big\|  \lam^{-N}{\Ll^{N}}g \big\|_{\sap}
\leq
\beta \norm{g}_{\sap}
\end{equation*}
for every function $g\colon\P^k \to \R$ with $\sca{m_\phi, g}=0$.
Furthermore, we can take
$\beta=\delta^{-N}$ for every
 given constant $0<\delta < (d \alpha)^{1/2}$, provided that $A$ is small enough.
\end{theorem}

Notice that 
Lemma \ref{l:cpt_apmixed} and the assumption $q_1>2$ imply that $\norm{\phi}_{\log^q}$ is finite 
for some  $q>2$. Hence, the scaling ratio $\lam$, the density function $\rho$, and
the measures $m_\phi$ and $\mu_\phi$ are well defined by Theorem \ref{t:main}. Notice also that $q_1 >2$ implies that the condition
$\alpha <d^{-5/(2p+2)}$ is automatically satisfied.

\begin{proof}
The proof follows the same lines as the one of
Theorem \ref{t:spectral_gap_apg}. It is however simpler
because the definition of the semi-norm 
$\norm{\cdot}_{\sap}$ is simpler than the one of $\norm{\cdot}_{\sapg}$.
In particular, we do not need any decomposition of $\lambda^{-n}\Ll^ng$.
Applying directly Proposition  \ref{p:estimate_n_gen} with $\phi^{(j)}:= \phi$ for
all $j\geq 1$ and recalling that $\|\1_n \|_\infty\lesssim \lam^n$
we obtain
\begin{equation*}
\norm{\lam^{-n}\Ll^n g}_{\sap}
 \lesssim
 \norm{g}_{\sap} 
\big(\frac{e^{\Omega}}{d\alpha}\big)^{n/2}
+ 
\norm{\phi}_{\sap}
\sum_{m=1}^n  
  m^{3/2} 
  \big(\frac{ e^{\Omega}  }{d \alpha}\big)^{m/2}
  \|\lam^{-n+m}\Ll^{n-m}g\|_\infty.
\end{equation*}
With this estimate, the rest of the proof is the same as that of
Theorem \ref{t:spectral_gap_apg}.
\end{proof}

As in the last section, we obtain the following counterpart of Theorem \ref{t:goal} as a consequence of the last result. 

\begin{theorem}\label{t:goal-bis} 
Let $f,p,\alpha,A,\Omega,\phi,\rho,  \lambda, m_\phi$ and $\Ll$ be as in Theorem \ref{t:spectral_gap_ap}.  
Then there is an explicit norm $\norm{\cdot}_{\diamond_0}$, depending on $f,p,\alpha,\phi$ and equivalent to 
$\norm{\cdot}_\infty+\norm{\cdot}_\sap$, such that
when $\norm{\phi}_\sap\leq A$ and $\Omega(\phi)\leq\Omega$ we have 
$$\|\lam^{-n}\Ll^n\|_\sap\leq c, \quad \norm{\rho}_\sap\leq c, \quad \norm{1/\rho}_\sap\leq c, \quad \text{and} \quad 
\big\|\lam^{-1} \Ll g \big\|_{\diamond_0} \leq \beta \norm{g}_{\diamond_0}$$
for every  $g\colon\P^k\to\R$ with $\sca{m_\phi, g}=0$, and for
some positive constants $c=c(f,p,\alpha,A,\Omega)$ and  $\beta =\beta(f,p,\alpha,A,\Omega)$
with $\beta<1$, both independent of $\phi,n,$ and $g$.
Furthermore, given any constant $0<\delta<(d\alpha)^{1/2}$, when $A$ is small enough,
the norm  $\norm{\cdot}_{\diamond_0}$
 can be chosen so that we can take $\beta = 1/\delta$. 
 \end{theorem}

Note that Lipschitz functions have finite $\norm{\cdot}_\sap$ semi-norm
(this follows from Lemma \ref{l:cpt_apmixed}, since Lipschitz functions can be uniformly
approximated by $\Cc^1$ ones whose norm is dominated
by the Lipschitz constant, see also the proof of Lemma \ref{l:logq-alpha-p-gamma}). 
So the last theorem can be applied to Lipschitz functions.
For such functions we can take any $p$ large enough and $\alpha$ close to 1.
The rate of contraction is then almost equal to $d^{-1/2}$ when $A$
is small enough (i.e., when $\phi$ is close to a constant function).
This rate is likely optimal as it corresponds to known
results obtained in the setting of zero weight, see \cite{dinh2010dynamics}.

\section{Statistical properties of equilibrium states}\label{s:statistic}

In this section we prove Theorem \ref{t:statistic}.
The definitions of the statistical properties that we consider are all recalled in Appendix \ref{a:abstract-result},
as well as criteria ensuring their validity in an abstract setting.
We work under the hypotheses of Theorems \ref{t:goal} and \ref{t:statistic} and
with the equivalent norms $\norm{\cdot}_{\diamond_1}$ and $\norm{\cdot}_{\diamond_2}$
as in Theorem \ref{t:goal}, see
Section \ref{s:proof:t:goal}.
 For simplicity, we define 
 the operator $L$ 
 as $L(g):=(\lam \rho)^{-1}\Ll(\rho g)$.

\subsection{Exponential equidistribution of preimages of points}

The following result gives 
a quantitative version of 
the equidistribution of preimages in Theorem \ref{t:main}.
Because of Lemma \ref{l:alpha-p-gamma-Holder} 
and of the definition of the norm $\norm{\cdot}_{\diamond_1}$,
it applies in 
particular
 to H\"older continuous
test functions, 
 see Remark \ref{rem:holder-pag}.

\begin{theorem}
Under the hypotheses of Theorem \ref{t:goal}, for every $x\in\P^k$, as $n$ tends to infinity
 the points in $f^{-n}(x)$, with
suitable weights, are equidistributed exponentially fast
 with respect to the conformal measure $m_\phi$.
More precisely, we have
\[\Big|\Big\langle \lambda^{-n} \sum_{f^n(a)=x} e^{\phi(a)+\cdots+\phi(f^{n-1}(a))} \delta_a - \rho(x) m_\phi, g \Big\rangle \Big|
\leq c\beta^n \|g\|_{\diamond_1},\]
for all $g \colon \P^k \to \R$,
where $0<\beta<1$ is the constant in Theorem \ref{t:goal} and $c$ is a positive
constant independent of $x$ and  $g$.
\end{theorem}
\proof
For simplicity, assume that $\|g\|_{\diamond_1}=1$. 
Define $g':=g-c_g\rho$ with $c_g:=\langle m_\phi,g\rangle$. By Lemma \ref{l:diamond}, both $|c_g|$ and $\|g'\|_{\diamond_1}$ are bounded by a constant.
Since the norms $\norm{\cdot}_{\diamond_1}$ and $\norm{\cdot}_{\diamond_2}$ are equivalent,
we deduce from Theorem \ref{t:goal} that 
$\|\lam^{-n}\Ll^n(g')\|_{\diamond_1}\lesssim \beta^n$. 
Since
\[ \Big \langle \lambda^{-n} \sum_{f^n(a)=x} e^{\phi(a)+\cdots+\phi(f^{n-1}(a))} \delta_a, g \Big \rangle=  \lam^{-n} (\Ll^n g)(x),\]
the LHS of the inequality in the theorem is bounded by
$$\norm{ \lam^{-n}\Ll^n (g) -  c_g \rho }_\infty = \norm{ \lam^{-n}\Ll^n (g - c_g \rho)}_\infty = \norm{ \lam^{-n}\Ll^n (g')}_\infty \lesssim \beta^n.$$
This implies the result.
\endproof

We also have the following useful result. 

\begin{proposition}\label{p:exp-est}
There exists a positive constant $c$ such that,
 for
every $g\colon \P^k \to \R$ with $\norm{g}_{\diamond_1} \leq 1$,
setting
 $c_{\rho g}:=\langle m_\phi,\rho g\rangle=\langle\mu_\phi,g\rangle$ we have
\[
\sca{\mu_\phi, e^{\beta^{-n} 
   \abs{L^n (g) - c_{\rho g} }}} \leq c \quad \mbox{ and } \quad 
\norm{L^n g -  c_{\rho g}}_{\diamond_1} \leq c  \beta^n.
\]
 \end{proposition}
 
\begin{proof}
Clearly, the first inequality follows from
the second one 
 by increasing
  the constant $c$ if necessary.
For the second inequality, define $\tilde g=\rho g$ and $\tilde g':=\tilde g-c_{\tilde g} \rho$. 
We have $\langle m_\phi,\tilde g'\rangle=0$.
By Lemma \ref{l:diamond} and Theorem \ref{t:goal}, we have that
$\norm{\tilde g}_{\diamond_1}$ is bounded by a constant because it is at most equal to
$3\norm{\rho}_{\diamond_1}\norm{g}_{\diamond_1}$.
We deduce that $c_{\tilde g}$ and hence $\norm{\tilde g'}_{\diamond_1}$ are both bounded by a constant.
Using again Lemma \ref{l:diamond} and Theorem \ref{t:goal}, we also obtain
$$\norm{L^n g - c_{\rho g}}_{\diamond_1} = \norm{\lam^{-n} \rho^{-1} {\Ll^n\tilde g} - c_{\tilde g}}_{\diamond_1}
= \norm{\rho^{-1} \lam^{-n} \Ll^n\tilde g'}_{\diamond_1}\leq 3 \norm{\rho^{-1}}_{\diamond_1}\norm{\lam^{-n} \Ll^n\tilde g'}_{\diamond_1}\lesssim \beta^n.$$
The result follows.
\end{proof}

\subsection{Exponential mixing of the equilibrium state}

 The speed of mixing
 in
 Theorem \ref{t:main}
  is not controlled
   for $g_0, \dots, g_r, g, l \in L^2 (\mu_\phi)$.
  We establish here some uniform exponential
bound for the speed of mixing of the system $\pa{\P^k, f,\mu_\phi}$
for more regular observables.
In the case of H\"older continuous
 weight $\phi$ and observable $g$, this was established in
 \cite{haydn1999convergence} for $k=1$
 (see also \cite{denker1996transfer} for a uniform sub-exponential speed) and in
  \cite{szostakiewicz2014stochastics} for $k>1$,
 see also \cite{dinh2010dynamics} for the case  when $\phi$ is constant.

\begin{theorem}
Under the hypotheses of Theorem \ref{t:goal}, for every integer $r\geq 0$, there is a positive constant $c=c(r)$ such that
$$\Big| \sca{\mu_\phi, g_0 \pa{g_1 \circ f^{n_1}} \cdots \pa{g_r \circ f^{n_r}}}
- \prod_{j=0}^r \sca{\mu_\phi, g_j}\Big| \leq c \beta^n
\Big( \prod_{j=0}^{r-1} \norm{g_j}_{\diamond_1}  \Big) \norm{g_r}_{L^1(\mu_\phi)}$$
for $0 = n_0 \leq n_1 \leq \dots \leq n_r$ and $ n:= \min_{0\leq j < r} (n_{j+1} - n_j)$. Here, the constant $0<\beta<1$ is the one from Theorem \ref{t:goal}.
\end{theorem}

\begin{proof}
We prove the theorem by induction on $r$. 
 Since when  $r=0$ the assertion  is
 trivial, we can assume  that the theorem holds for $r-1$ and prove it for $r$.

The fact that the statement holds for $r-1$
and the invariance of $\mu_\phi$
 imply the result for $r$ when $g_0$ is constant.
So, by subtracting from $g_0$ the constant $c_{\rho g_0}:  =\sca{m_\phi,\rho g_0}=\sca{\mu_\phi, g_0}$,
 we can assume that $\sca{\mu_\phi, g_0}=0$.
Setting $g' := g_1 \pa{g_2 \circ f^{n_2-n_1}} \cdots \pa{g_r \circ f^{n_r- n_1}}$
and recalling that $L (\cdot) = (\lam \rho)^{-1} \Ll (\rho \cdot)$, computing
as at the end of the proof of Theorem \ref{t:main} we get
\[
\begin{aligned}
\abs{\sca{\mu_\phi, g_0 (g' \circ f^{n_1}) }}
&= \abs{\sca{\mu_\phi,  L^{n_1}(g_0 (g' \circ f^{n_1}))}}
= \abs{\sca{\mu_\phi, L^{n_1}(g_0) g'  } }\\
&\leq \norm{L^{n_1}(g_0)}_{\infty}  \norm{g'}_{L^1 (\mu_\phi)}\\
& \leq \norm{L^{n_1} (g_0)}_{\infty} \norm{g_1}_\infty \dots \norm{g_{r-1}}_\infty \norm{g_r\circ f^{n_r- n_1}}_{L^1 (\mu_\phi)} \\
& = \norm{L^{n_1} (g_0)}_{\infty} \norm{g_1}_\infty \dots \norm{g_{r-1}}_\infty \norm{g_r}_{L^1 (\mu_\phi)}.
\end{aligned}\]
The assertion follows from Proposition \ref{p:exp-est} applied to $g_0$ instead of $g$.
\end{proof}

\subsection{Central Limit Theorem (CLT)}
We prove here the CLT for observables $g$ such that $\norm{g}_{\diamond_1}<\infty$
and which are 
not coboundaries, see Appendix \ref{a:abstract-result} for the definitions.
In the case of H\"older continuous
 weight $\phi$ and observable $g$, this was established in
\cite{denker1996transfer} for $k=1$  and in 
  \cite{szostakiewicz2014stochastics} for $k>1$,
 see also \cite{dinh2010dynamics} for the case  when $\phi$ is constant.
We start with the following version that does not need the introduction of
perturbed operators. More refined versions 
of this theorem will be given later in Sections
\ref{ss:local_clt} and \ref{ss:asip_consequences}.

\begin{theorem}
Under the hypotheses of Theorem \ref{t:goal}, consider a function $g\colon \P^k \to \R$ 
such that $\norm{g}_{\diamond_1}<\infty$ and
$\sca{\mu_\phi, g}=0$.
Assume that $g$ is not a coboundary. Then $g$ satisfies the CLT
with variance $\sigma>0$ given by
\begin{equation}\label{eq_def_sigma}
\sigma^2 :=\sca{\mu_\phi, g^2} +2\sum_{n\geq 1} \sca{\mu_\phi, g \cdot \pa{g \circ f^n}}.
\end{equation}
Moreover, the Berry-Esseen theorem holds for the observable $g$.
\end{theorem}

We refer to Definition \ref{d:clt} for the statement of the CLT
and to Theorem \ref{t:clt-spectral} for the Berry-Esseen theorem.
Recall that $g$ is a coboundary if there exists $h \in L^2 (\mu_\phi)$ such that 
$g = h\circ f - h$, and this is the case if and only if $\sigma =0$. It is not difficult to check that a coboundary cannot satisfy the CLT as in Definition \ref{d:clt}.

\begin{proof}
Let $\Bc$ denote the Borel $\sigma$-algebra on $\P^k$. 
By Gordin Theorem \ref{t:a_clt}, we only need to show that
\[
\sum_{n\geq 1} \norm{\Ex \pa{ g | f^{-n}\mathcal{B}}}^2_{L^2 (\mu_\phi)}< \infty.
\]
We will transform the above series into a series of the norms $\norm{L^n g}_{L^2 (\mu_\phi)}$, that
we will bound by means of Proposition
\ref{p:exp-est}.
Recall that $\sca{\mu_\phi, (L^n g) h}= \sca{\mu_\phi, g (h\circ f^n)}$ 
for all $g,h\colon \P^k \to \R$.
So
\[
\sca{\mu_\phi, g  \pa{h \circ f^n}}
 = \sca{\mu_\phi, (L^n g)  h}
 = \sca{\mu_\phi, \pa{(L^n g) \circ f^n}  \pa{h \circ f^n} },
\]
where in the last passage we used the $f_*$-invariance of $\mu_\phi$.
This precisely says that $\Ex (g, f^{-n}\Bc) = L^n (g) \circ f^n$.
Note that the norm of $f^*$ on $L^2 (\mu_\phi)$ is  1.
So we have  
\[
\sum_{n\geq 1} \norm{\Ex \pa{ g | f^{-n}\Bc}}^2_{L^2 (\mu_\phi)}
\leq
\sum_{n\geq 1} \norm{ L^n (g)}^2_{L^2 (\mu_\phi)}
{\leq \sum_{n\geq 1} \norm{ L^n (g)}^2_{\infty}.}
\]
The last sum converges because of Proposition \ref{p:exp-est} and the hypothesis $\sca{\mu_\phi, g}=0$.
For the Berry-Esseen theorem, see Theorem \ref{t:clt-spectral} and Section
\ref{ss:perturbed_operators} below.
This completes the proof.
\end{proof}

We have the following characterization of coboundaries
with bounded $\norm{\cdot}_{\diamond_1}$ norm that will be used in Section \ref{ss:local_clt}.

\begin{proposition}\label{p:coboundaries}
Let $g$ be a coboundary and assume
 that $\norm{g}_{\diamond_1}<\infty$.
 Then there exists
a function $\tilde h$ such that  $\|\tilde h\|_{\diamond_1}<\infty$
and $g = \tilde h \circ f - \tilde h$ on the small Julia set of $f$.
\end{proposition}

\begin{proof}
By the definition of coboundary, there exists a function $h \in L^2 (\mu_\phi)$ such that $g = h \circ f - h$ in
$L^2 (\mu_\phi)$. This identity implies that $\langle \mu_\phi,g\rangle=0$. 
By adding a constant to $h$ we can assume that $\sca{\mu_\phi,h}=0$.
Since
$$L(h\circ f) = \lam^{-1} \rho^{-1} \Ll (\rho (h \circ f)) =  \lam^{-1}\rho^{-1}\Ll (\rho)  h = h$$
a direct computation gives
$$Lg + \dots + L^{n}g= h - L^n h.$$
 We claim that
$L^n h \to 0$ in $L^2 (\mu_\phi)$ when $n$ tends to infinity.
Indeed, fix any $\eps>0$ and choose a continuous function $h'$
with
$\norm{h-h'}_{L^2(\mu_\phi)}<\eps$ and
 $\sca{\mu_\phi,h'}=0$.
 By \cite[Lemma 4.3]{bd-eq-states-part1} 
 there exists a positive
 constant $c$
 such that
$\norm{L^n}_{L^2(\mu_\phi)}\leq c$ 
   for all $n\geq 0$. Moreover, 
$\norm{L^n h'}_\infty =  \norm{L^n h' -  \sca{\mu_\phi, h'}}_\infty \to 0$
   as $n\to \infty$ by Theorem \ref{t:main} and the definition of $L$. Hence,
  \[
  \norm{L^n h}_{L^2 (\mu_\phi)} 
  \leq \norm{L^n (h-h')}_{L^2(\mu_\phi)} +
  \norm{L^n h'}_{L^2 (\mu_\phi)} \lesssim \eps\]
as $n\to \infty$.  By taking $\eps\to 0$, this proves that 
$L^n h \to 0$ in $L^2 (\mu_\phi)$.
Proposition \ref{p:exp-est} and the identity $\langle \mu_\phi,g\rangle=0$ imply that the LHS of the last identity converges to
a function  $\tilde h$ with finite $\norm{\cdot}_{\diamond_1}$ norm. 
We conclude that $\tilde h = h$ $\mu_\phi$-almost everywhere. Therefore, we have $g = \tilde h \circ f - \tilde h$
$\mu_\phi$-almost everywhere, and thus everywhere on $\supp(\mu_\phi)$ as both sides are continuous functions.
The proposition follows because the support of $\mu_\phi$ is equal to the small Julia set, see Theorem \ref{t:main}.
\end{proof}

\subsection{Properties of perturbed Perron-Frobenius operators}\label{ss:perturbed_operators}
 The next statistical
properties will be proved by means of spectral methods, and more precisely by
the introduction of suitable (complex) perturbations of the operator $\Ll= \Ll_\phi$.
This method was originally developed by Nagaev \cite{nagaev1957some} in the context of Markov chains. 
More details are given in Appendix \ref{a:abstract-result}.
We  introduce here some notations and prove some preliminary results. We mainly follow 
the approach as presented in
\cite{broise1996transformations,dinh2007thermodynamics, rousseau1983theoreme}.
 
\begin{definition}\label{d:l:perturbed}
Given functions $\phi,g\colon \P^k \to \R$, $h\colon \P^k \to \C$, and a parameter $\theta \in \C$ we set
\[
\Ll_{\phi +  \theta g} h :=  \Ll_{\phi + \theta g} \Re h + i  \Ll_{\phi + \theta g} \Im h,
\]
where the operator in the RHS is the natural extension of \eqref{e:L} in the case with
complex weight.
\end{definition}

Since from now on we fix
 $\phi$ and $g$, 
we will just denote the above operator by $\Ll_{[\theta]}$ when no
possible confusion arises. In particular, we have
 $\Ll_{[0]}=\Ll_\phi$.
By means of Definition \ref{d:l:perturbed}, we 
 extend
the operator $\Ll$ to complex weights and complex test functions. We naturally
extend the  norms  $\norm{\cdot}_{\diamond_1}$ and  $\norm{\cdot}_{\diamond_2}$
  to these function spaces by setting
\[
\norm{h}_{\diamond_1} := \norm{\Re h}_{\diamond_1} + \norm{\Im h}_{\diamond_1} \quad \text{and} \quad 
\norm{h}_{\diamond_2} := \norm{\Re h}_{\diamond_2} + \norm{\Im h}_{\diamond_2}. 
\]

We will be in particular
interested in
the case where $\theta$ is small or
pure imaginary. 
We may develop our study given in the previous sections
for a complex weight directly, but we prefer to use a shortcut which only requires to treat the real case.
The next lemma collects the main properties of the family of operators 
$\Ll_{[\theta]}$ that we need.

\begin{lemma}\label{l:perturbed}
Assume that
 $ \norm{g}_{\diamond_1}$ is finite.
Then the following assertions hold for both of the
norms $\norm{\cdot}_{\diamond_1}$ and  $\norm{\cdot}_{\diamond_2}$.
\begin{enumerate}
\item[{\rm (i)}] $\Ll_{[0]} = \Ll_\phi$;
\item[{\rm (ii)}] For every $\theta \in \C$,
 $\Ll_{[\theta]}$ is a bounded operator;
\item[{\rm (iii)}] \label{item_analyticity} The map $\theta\mapsto \Ll_{[\theta]}$ is analytic in $\theta$;
\item[{\rm (iv)}] For every $n\in \N$, $\theta \in \C$, and $h\colon \P^k \to \C$, we have
\begin{equation}\label{e:ltn-l0n}
\Ll_{[\theta]}^n h = \Ll_{[0]}^n ( e^{\theta S_n(g) } h),
\end{equation}
where
\begin{equation}\label{e:birkhoff-sum}
S_0 (g) :=0
\quad \mbox{ and  } \quad 
 S_n (g) := \sum_{j=0}^{n-1} g \circ f^j \mbox{ for } n\geq 1.
 \end{equation}
 \end{enumerate}
\end{lemma}

\begin{proof}
The first item is true by definition. For the second and the third items,
it is enough to prove them for the norm $\norm{\cdot}_{\diamond_1}$
 since the norm
$\norm{\cdot}_{\diamond_2}$ is equivalent, see Lemma \ref{l:diamond}.
By Lemma \ref{l:diamond} we have,
 for all $h\colon \P^k \to \C$ and $n\in \N$,
\[
\norm{\Ll_{[0]} (g^n h)}_{\diamond_1}
 \leq  \norm{\Ll_{[0]}}_{\diamond_1} \norm{g^n h}_{\diamond_1}
\leq  3^n \norm{\Ll_{[0]}}_{\diamond_1} \norm{g}_{\diamond_1}^n \norm{h}_{\diamond_1}.
\]
So, the series $\sum_{n\geq 0} (\theta^n /n!) \Ll_{[0]} (g^n h)$ converges normally in norm $\norm{\cdot}_{\diamond_1}$.
The limit is equal to
$\Ll_{[\theta]} (h)$ and we have
\[
\norm{\Ll_{[\theta]} (h)}_{\diamond_1}
\leq 
e^{3 |\theta| \norm{g}_{\diamond_1}}
\norm{\Ll_{[0]}}_{\diamond_1} \norm{h}_{\diamond_1}.
\]
This proves the second and the third items.

The last item follows by a direct induction. The case $n=0$ is trivial and we have
\[\begin{aligned}
\Ll_{[0]}^n ( e^{\theta S_n(g) } h) & = \Ll_{[0]}
\Big(\Ll_{[0]}^{n-1} \big(e^{\theta  \cdot g \circ f^{n-1} }e^{\theta S_{n-1} (g)} h \big) \Big) \\
& = \Ll_{[0]} \Big(e^{\theta g} \Ll_{[0]}^{n-1} (e^{\theta S_{n-1}(g) } h) \Big) 
= \Ll_{[\theta]} \Ll_{[0]}^{n-1} (e^{\theta S_{n-1}(g) } h) 
=
\Ll_{[\theta]} \Ll_{[\theta]}^{n-1} h .
\end{aligned}\]
The proof is complete.
\end{proof}

Recall that the operator $\lam^{-1}\Ll_{[0]}$
has  $\rho$ as its unique (up to a multiplicative constant)\break
eigenfunction of eigenvalue 1. It 
 is a contraction with respect to the norm $\norm{\cdot}_{\diamond_2}$ (which is equivalent to $\norm{\cdot}_{\diamond_1}$)
on the space of functions whose integrals with respect to $m_\phi$ are zero,
see Theorem \ref{t:goal}. The following is then a consequence of the Rellich
perturbation method described in \cite[Ch.\ VII]{dunford1958linear}, see also
\cite[Proposition 5.2]{broise1996transformations} and \cite{kato2013perturbation,le1982theoremes}.
Note that the last assertion of Theorem \ref{t:statistic} is a
direct consequence of the analyticity of $\alpha$  given by the fourth item.

\begin{proposition}\label{p:decomposition}
Assume that $\norm{g}_{\diamond_1}$ is finite and let 
$0<\beta<1$ be 
the constant in Theorem \ref{t:goal}. Then, for all
$\beta < \beta'<1$,
the following holds  for $\theta$ sufficiently small and all $n\in \N$:
there exists a decomposition 
\[
\lam^{-n}\Ll_{[\theta]}^n =  \alpha(\theta)^n \Phi_\theta + \Psi^n_\theta
\]
as operators on $\{h\ : \ \norm{h}_{\diamond_1}<\infty\}$
  such that
\begin{enumerate}
\item[{\rm(i)}] $\alpha (\theta)$  is the (only) largest eigenvalue
of $\Ll_{[\theta]}$, $\alpha(0)=1$ and $|\alpha (\theta)|>\beta'$;
\item[{\rm(ii)}] $\Phi_\theta$ is the projection on the (one  dimensional) 
eigenspace associated to $\alpha(\theta)$
and we have $\Phi_0 (h) = \sca{m_\phi, h} \rho$;
\item[{\rm(iii)}] $\Psi_\theta$ is a bounded operator on $\{h \ : \ \norm{h}_{\diamond_1}<\infty \}$
whose spectral radius
 is $<\beta'$ and
\[
\Psi_\theta \circ \Phi_\theta = \Phi_\theta \circ \Psi_\theta =0;
\]
\item[{\rm(iv)}] the maps $\theta\mapsto \Psi_\theta, \theta \mapsto \Phi_\theta$, and
$\theta \mapsto \alpha(\theta)$ are analytic.
\end{enumerate}
\end{proposition}

The last property that
we will need is the second order expansion of $\alpha (\theta)$ for $\theta$ near $0$. It is a consequence of the above results.

\begin{lemma}\label{l:development}
Assume that $\norm{g}_{\diamond_1}$ is finite and $\langle \mu_\phi,g\rangle=0$. 
Let $\sigma>0$ be
as in \eqref{eq_def_sigma} and $\alpha(\theta)$ be given by Proposition \ref{p:decomposition}. Then we have
\[
\alpha (\theta) = e^{  \frac{\sigma^2 \theta^2}{2}  + o(\theta^2)} = 1 + \frac{\theta^2 \sigma^2}{2} + o(\theta^2).
\] 
\end{lemma}

\begin{proof}
We know from Lemma \ref{l:perturbed} 
and Proposition \ref{p:decomposition}
that $\alpha (0)=1$ and that $\theta \mapsto \alpha (\theta)$
is analytic near $0$.
 So, we only need to prove that
 $\alpha'(0)=0$ and $\alpha''(0)= \sigma^2$. 
 
 We start with the first equality. First of all
 notice that, for all $n\in \N$
 and $h\colon \P^k \to \C$, the invariance of $m_\phi$ implies 
 \begin{equation}\label{e:simplify}
\sca{  \mu_\phi,h} = \sca{ m_\phi, \rho h} = \sca{  m_\phi, \lam^{-n}\Ll_{[0]}^n (\rho h)}.
 \end{equation}
 Applying these identities to $e^{S_n (g)/n}$ instead of $h$ (with $S_n (g)$ as in \eqref{e:birkhoff-sum}) 
 and using \eqref{e:ltn-l0n}, we get
\[
\langle  \mu_\phi, e^{S_n (g)/n} \rangle = \langle  m_\phi, \lam^{-n} \Ll^n_{[1/n]} (\rho)\rangle.
\]
Then, the decomposition given in Proposition
 \ref{p:decomposition} (applied with $1/n$ instead of $\theta$) gives
 \begin{equation} \label{e:mu-Sn}
\langle  \mu_\phi, e^{S_n (g)/n}\rangle =\alpha(1/n)^n \sca{  m_\phi, \Phi_{1/n} (\rho) }
+
\big\langle  m_\phi, \Psi^n_{1/n}(\rho)\big\rangle.
\end{equation}
By Proposition \ref{p:decomposition}, the second term in the RHS goes to zero exponentially fast
as $n\to \infty$ .
Using that $\alpha(\theta)=1+\alpha'(0)\theta +o(\theta)$ and
$\Phi_\theta(\rho)=\rho+o(1)$ as $\theta\to 0$
 and the identity $\sca{m_\phi, \rho}=1$, we get
\[
\lim_{n\to \infty} \langle  \mu_\phi, e^{S_n (g)/n} \rangle = e^{ \alpha'(0)}.
\]
Recall that $\mu_\phi$ is mixing and hence ergodic. Moreover, $g$ is continuous and  $\sca{\mu_\phi,g}=0$. 
Birkhoff's ergodic theorem then implies that $S_n (g) /n\to 0$ $\mu_\phi$-almost surely as $n$ goes to infinity. 
We conclude that $\alpha'(0)=0$ as desired.

 Let us now prove that $\alpha'' (0)=\sigma^2$. The $f$-invariance
 of $\mu_\phi$ and a direct computation give 
 \[
\sigma^2 = \lim_{n\to \infty} 
\Big\langle \mu_\phi,
\Big( \frac{S_n g}{\sqrt{n}} \Big)^2
\Big\rangle
 \quad \text{and}\quad
\Big\langle \mu_\phi,
\Big( \frac{S_n g}{\sqrt{n}} \Big)^2
\Big\rangle
= \frac{\partial^2}{\partial \theta^2} \sca{ \mu_\phi, e^{(\theta/\sqrt{n}) S_n g}}\Big|_{\theta=0}.
\]
We use again \eqref{e:ltn-l0n} and
Proposition \ref{p:decomposition} (applied with $\theta/\sqrt{n}$
instead of $\theta$) to get
\[
\sca{\mu_\phi, e^{(\theta/\sqrt{n}) S_n g}}
= \alpha (\theta/\sqrt{n})^n \sca{m_\phi, \Phi_{\theta/\sqrt{n}}(\rho)} +
\sca{m_\phi, \Psi^n_{\theta/\sqrt{n}} (\rho)}.
\]
A direct computation, together with the
identities $\alpha(0)=1$, $\alpha'(0)=0$, $\langle m_\phi,\rho\rangle=1$ and the
properties
 in Proposition \ref{p:decomposition}, gives
$$\frac{\partial^2}{\partial \theta^2} \sca{ \mu_\phi, e^{(\theta/\sqrt{n}) S_n g}}\Big|_{\theta=0} =\alpha''(0) +o(1).$$
It is now clear that $\alpha''(0)=\sigma^2$. This completes the proof of the lemma.
 \end{proof}

\subsection{Local Central Limit Theorem (LCLT)}\label{ss:local_clt}

We establish here  an improvement of the CLT, see Definition \ref{d:lclt}, for observables
satisfying a further cocycle condition. 
Our result is new for $k=1$, $\phi$ non-constant, and for $k> 1$, even when
 $\phi =0$; for $k=1$ and $\phi=0$, see \cite{dinh2007thermodynamics}. We need the following definition.

\begin{definition}  \label{d:cocycle}
Let $g\colon\P^k\to\R$ be a measurable function. We say that $g$ is a {\it multiplicative cocycle} if there exist
$t>0$, $s\in \R$, 
and a measurable function $\xi\colon \P^k \to \C$, not equal to zero $\mu_\phi$-almost everywhere,
 such that $e^{i t g (z)}  \xi (z) = e^{i s} \xi (f(z))$. We say
 that $g$ is a {\it $(\Cc^0,\phi)$-multiplicative cocycle} (resp. 
\emph{$(\norm{\cdot}_{\diamond_1},\phi)$-multiplicative cocycle}) if there exist
$t>0$, $s\in \R$, 
and $\xi\colon \P^k \to \C$, not identically zero
on the small Julia set of $f$, which is continuous (resp. with finite
$\norm{\cdot}_{\diamond_1}$ norm), such that $e^{i t g (z)}  \xi (z) = e^{i s} \xi (f(z))$ on  the small Julia set of $f$.
\end{definition}

Recall that the supports of $\mu_\phi$ and
$m_\phi$ are both equal to the small Julia set of $f$, see Theorem \ref{t:main}.

 \begin{theorem} \label{t:LCLT}
 Under the hypotheses of Theorem \ref{t:goal}, let $g\colon \P^k \to \R$ be such that
 $\norm{g}_{\diamond_1}$ is finite, $\sca{\mu_\phi, g}=0$, and $g$ is not a $(\norm{\cdot}_{\diamond_1},\phi)$-multiplicative cocycle.
 Then $g$ satisfies the LCLT with variance $\sigma$ given by \eqref{eq_def_sigma}.
 \end{theorem}

 \begin{remark}\label{r:lclt}
 The LCLT is a refined version of the CLT. Notice that it requires a stronger assumption on the observable $g$.
 \end{remark}
 
We first need some properties of multiplicative cocycles. Note that the following lemma still holds if 
we only assume that $\phi$ and $g$ have bounded $\norm{\cdot}_{\log^q}$ norms for some $q>2$.

\begin{lemma} \label{l:cocycle}
There is a positive 
constant $c$
independent of $g,t$,
and $n$ such that $\|\lam^{-n}\Ll_{[it]}^n\|_\infty\leq c$. 
Let $K$ be a compact subset of $\R\setminus \{0\}$ (e.g., a
singleton in $\R\setminus \{0\}$). Let $\Fc$ be a uniformly bounded
and equicontinuous family of functions on $\P^k$. Then the family
$$\Fc^K_{\N}:=\big\{\lambda^{-n}\Ll_{[it]}^n h \  : \  t\in K, \ h\in\Fc, \ n\in\N\big\}$$
is also uniformly bounded and equicontinuous. Furthermore, if $K\subset (0,\infty)$ and 
$g$ is not a $(\Cc^0,\phi)$-multiplicative cocycle, then
$\|\lambda^{-n}\Ll_{[it]}^n h\|_\infty$ tends to $0$ when $n$
goes to infinity, uniformly in $t\in K$ and $h\in\Fc$.
\end{lemma}

\proof
Define $\phi_t:=\phi+itg$. Observe that $|e^{\phi_t}|=e^\phi$.
It follows that $\|\Ll_{[it]}^n\|_\infty\leq \|\Ll^n\|_\infty \leq c\lam^n$ for some
positive
constant $c$, according to
 \eqref{e:intro-cv-rho}.
Using this, we can follow the
proof of 
 \cite[Lemma 3.9]{bd-eq-states-part1},
with $\phi_t$ instead of $\phi$, and
obtain that $\Fc^K_{\N}$ is uniformly bounded and equicontinuous. 
It remains to prove the last assertion in the lemma.

Let $\Fc^K_\infty$ denote the family of the limit functions 
of all sequences $\lam^{-n_j}\Ll_{[it_j]}^{n_j}(h_j)$ with 
$t_j\in K$, $h_j\in \Fc$, 
and $n_j$ going to infinity. 
By Arzel\`a-Ascoli theorem, this is a uniformly bounded
and equicontinuous family of functions
which is compact for the uniform topology. 
Define 
$$M:=\max \big\{|l(a)/\rho(a)| \ : \  l\in\Fc^K_\infty,\  a \text{ in the small Julia set} \big\}.$$

\medskip\noindent
{\bf Claim 1.} If $M=0$, then $\Fc^K_\infty$ only contains the zero function.

\proof[Proof of Claim 1]
Assume that $M=0$. Consider a function $l$ in $\Fc^K_\infty$. We show that $l=0$. 
 Consider functions
$l_{-n}\in\Fc^K_\infty$ such that $l=\lam^{-n}\Ll_{[it]}^nl_{-n}$ for some $t\in K$ 
(take $t$ as a limit of $t_j$ above). 
Since $M=0$, the function $l_{-n}$ vanishes on the support of $m_\phi$ (which is the small Julia set) 
and we have $c_{|l_{-n}|}=\langle m_\phi, |l_{-n}|\rangle=0$.
Moreover, we also have
$$|l|=|\lam^{-n}\Ll_{[it]}^nl_{-n}|\leq \lam^{-n} \Ll^n|l_{-n}|.$$
By 
 Lemma \ref{l:dependence-bis}
applied to
the family $\{|h|: \ h\in \Fc^K_\infty\}$, the last function converges to 0
when $n$ tends to infinity. The claim follows.
\endproof

\noindent
{\bf Claim 2.} There is $l\in \Fc^K_\infty$ such that $|l/\rho|=M$ on the small Julia set.

\proof[Proof of Claim 2]
Choose a function $l$ in $\Fc^K_\infty$ such that $\max|l/\rho|=M$.
Consider $l_{-n}$ and $t$ as in the proof of Claim 1. Let $a$ in the small Julia set be 
such that $|l(a)/\rho(a)|=M$. Then we have
\begin{align*}
|l(a)|=|\lam^{-n}\Ll_{[it]}^n l_{-n}(a)| &=\Big|\lam^{-n}\sum_{b\in f^{-n}(a)} e^{\phi_t(b)+\cdots+\phi_t(f^{n-1}(b))} l_{-n}(b)\Big| \\
& \leq \lam^{-n}\sum_{b\in f^{-n}(a)} e^{\phi(b)+\cdots+\phi(f^{n-1}(b))} 
M \rho(b)= M\rho(a).
\end{align*}
So the last inequality is actually an equality. In particular,
we have $|\l_{-n}/\rho|=M$ on $f^{-n}(a)$. Recall that when $n$
goes to infinity, $f^{-n}(a)$ tends to the small Julia set. Therefore, if $l_{-\infty}$
is a limit of $l_{-n}$, we have $|l_{-\infty}/\rho|=M$ on the small Julia set.
Replacing $l$ by $l_{-\infty}$ gives the claim.
\endproof

\noindent
{\bf Claim 3.} There are $t\in K$ and $l\in \Fc^K_\infty$ such that
$|\lam^{-N}\Ll_{[it]}^N l|=M \rho$ on the small Julia set for every $N\geq 0$.

\proof[Proof of Claim 3]
Choose $l$ satisfying Claim 2
and $t$ as in its proof. 
Define $l_{-n}$ and $l_{-\infty}$ as in the proof of that claim.
Fix an $N\geq 0$. For every $n\geq N$, on the small Julia set we have
\begin{align*}
M\rho & = |l| = |\lam^{-n}\Ll_{[it]}^n l_{-n}| =|\lam^{-n+N}\Ll_{[it]}^{n-N}(\lam^{-N}\Ll_{[it]}^{N}l_{-n})| \\
& \leq \lam^{-n+N} \Ll^{n-N} |\lam^{-N}\Ll_{[it]}^{N}l_{-n}| 
\leq \lam^{-n+N} \Ll^{n-N} \lam^{-N}\Ll^{N}(M\rho) = M\rho.
\end{align*}
So the two last inequalities are in fact equalities and we deduce
that $|\lam^{-N}\Ll_{[it]}^{N}l_{-n}| =M\rho$. It follows that 
$|\lam^{-N}\Ll_{[it]}^{N}l_{-\infty}| =M\rho$ for every $N\geq 0$.
Replacing $l$ by $l_{-\infty}$ gives the claim.
\endproof

Assume now that $g$ is not a $(\Cc^0,\phi)$-multiplicative cocycle. By Claim 1, we
only need to show that $M=0$. Assume by contradiction that $M\not=0$. Consider $t$ and $l$ as in Claim 3. 
Define $\xi(a):=l(a)/\rho(a)$ for $a\in\P^k$ and
$\vartheta(a):=e^{itg(a)} \xi(a)/\xi(f(a))$ for $a$ in the small Julia set.
These functions are continuous and we have $|\xi(a)|=M$ 
and $|\vartheta(a)|=1$ on the small Julia set.
We have for $a$ in the small Julia set 
\begin{align*}
M\rho(a) &= |\lam^{-n}\Ll_{[it]}^n l(a)| =\Big|\lam^{-n}\sum_{b\in f^{-n}(a)} e^{\phi_t(b)+\cdots+\phi_t(f^{n-1}(b))} l(b)\Big| \\
& \leq \lam^{-n}\sum_{b\in f^{-n}(a)} e^{\phi(b)+\cdots+\phi(f^{n-1}(b))} 
M \rho(b)= M\rho(a).
\end{align*}
So, the last inequality is an equality. Using the function $\xi$,
we can rewrite this equality as (we remove the factors $\lam^{-n}$ and
also the factors $|\xi(f^n(b))|$ and $M$ as they are both equal to $|\xi(a)|$ and independent 
of $b \in f^{-n}(a)$)
$$\Big|\sum_{b\in f^{-n}(a)} \vartheta(b)\ldots \vartheta(f^{n-1}(b))e^{\phi(b)+\cdots+\phi(f^{n-1}(b))}\rho(b)\Big|
=  \sum_{b\in f^{-n}(a)} e^{\phi(b)+\cdots+\phi(f^{n-1}(b))} \rho(b).$$
As $|\vartheta|=1$, we deduce that if $b$ and $b'$ are two points in $f^{-n}(a)$ then 
$$\vartheta(b)\ldots \vartheta(f^{n-1}(b))=\vartheta(b')\ldots \vartheta(f^{n-1}(b')).$$
This and a similar equality for $f(b),f(b'),n-1$ instead
of $b,b',n$
 imply that $\vartheta(b)=\vartheta(b')$. We conclude
that $\vartheta$ is constant on $f^{-n}(a)$ for every $n$. As
$f^{-n}(a)$ tends to the small Julia set
when $n$ going to infinity,
it follows
that $\vartheta$ is constant.
From the definition of $\vartheta$, and since $\xi$ is continuous,
we obtain that $g$ is a $(\Cc^0,\phi)$-multiplicative
cocycle for a suitable real number $s$ such that $\vartheta=e^{is}$.
This is a contradiction. So we have $M=0$, as desired.
\endproof

Recall that $\phi$ and $g$ have bounded $\norm{\cdot}_{\diamond_1}$ norms.

\begin{lemma} \label{l:non-cocycle}
Let $K$ be a compact subset of $\R$.
There is a positive constant $c$
 such that $\|\lam^{-n}\Ll_{[it]}^n\|_{\diamond_1}\leq c$ 
for every $n\geq 0$ and $t\in K$. 
If $K\subset \R \setminus \{0\}$
 and $g$ is not a $(\norm{\cdot}_{\diamond_1},\phi)$-multiplicative
cocycle,  then there are constants $c>0$ and $0<r<1$ such
that $\|\lam^{-n}\Ll^n_{[it]}\|_{\diamond_1}\leq cr^n$ for every $t\in K$ and $n\geq 0$.
\end{lemma}
\proof
Consider the functional ball
$\Fc:=\{h\ :\  \norm{h}_{\diamond_1}\leq 1\}$.
By Arzel\`a-Ascoli theorem
this ball is compact for the uniform topology. Define
$\Fc^K_\infty$ as in the proof of Lemma \ref{l:cocycle}. Using
that lemma, we can follow the proof of Proposition \ref{p:rho-inverse} and
obtain that $\|\lam^{-n} \Ll_{[it]}^n\|_\sapg$ is bounded uniformly on $n$
and $t\in K$. It follows that a similar property holds for the norm $\norm{\cdot}_{\diamond_1}$. 
This gives the first assertion in the lemma.
We also obtain that the family $\Fc^K_\infty$ is
bounded in the $\norm{\cdot}_{\diamond_1}$ norm and,
 using again
Arzel\`a-Ascoli theorem, we  obtain that it is compact in the uniform topology.

Consider now the second assertion and assume that $g$ is
not a $(\norm{\cdot}_{\diamond_1},\phi)$-multiplicative cocycle. 
We first show that $\Fc^K_\infty$ reduces to $\{0\}$. Assume by contradiction that this is not true. 
Consider $t$ and $l$ as in Claim 3 in the proof of Lemma \ref{l:cocycle}.
Recall that both $\norm{l}_{\diamond_1}$ and  $\norm{1/\rho}_{\diamond_1}$ are
finite, see Theorem \ref{t:goal}.
By Lemma \ref{l:apg_gh}, the function $\xi:=l/\rho$ satisfies the same property and
we conclude, as at the end of the proof of Lemma \ref{l:cocycle},
that $g$ is a $(\norm{\cdot}_{\diamond_1},\phi)$-multiplicative cocycle.
This contradicts the hypothesis.

So $\Fc^K_\infty$ is reduced to $\{0\}$.  By definition of $\Fc^K_\infty$, we obtain that $\lam^{-n}\Ll_{[it]}^nh$ converges to 0 uniformly on $h\in\Fc$ and $t\in K$. Using this property, 
we can follow the proof of Proposition \ref{p:G12} (take $\beta=1/4$) to obtain that $\|\lam^{-N} \Ll_{[it]}^Nh\|_\sapg\leq 1/4$ for $N$ large enough and for all $h\in\Fc$ and $t\in K$. When $N$ is large enough, we also have $\|\lam^{-N}\Ll_{[it]}^N h\|_\infty\leq 1/4$. Therefore, we have 
$\|\lam^{-N} \Ll_{[it]}^N\|_{\diamond_1}\leq 1/2$ which implies the desired property with $r=2^{-1/N}$. 
\endproof

We have the following characterizations of multiplicative cocycles.

\begin{proposition}\label{p:mul_cocycle}
Let $g\colon \P^k \to\R$ be such that $\norm{g}_{\diamond_1}$ is finite.
Then the following properties are equivalent:
\begin{enumerate}
\item[{\rm(ia)}] $g$ is a multiplicative cocycle;
\item[{\rm(ib)}] $g$ is a $(\Cc^0,\phi)$-multiplicative cocycle;
\item[{\rm(ic)}] $g$ is a $(\norm{\cdot}_{\diamond_1},\phi)$-multiplicative cocycle;
\item[{\rm(iia)}] there exists a number $t>0$ such that the spectral radius with respect to the norm $\norm{\cdot}_{\diamond_1}$
of $\lam^{-1}\Ll_{[it]}$ is $\geq 1$.
\item[{\rm(iib)}] there exists a number $t>0$ such that the spectral radius with respect to the norm $\norm{\cdot}_{\diamond_1}$
of $\lam^{-1}\Ll_{[it]}$ is equal to $1$.
\end{enumerate}
Moreover, every coboundary
 with finite $\norm{\cdot}_{\diamond_1}$ norm  is a
$(\norm{\cdot}_{\diamond_1},\phi)$-multiplicative cocycle.
\end{proposition}

\begin{proof}
The last assertion is a direct consequence of Proposition \ref{p:coboundaries}.
Indeed, if $g$ is such a coboundary, we have $g=\tilde h\circ f -\tilde h$ on the
small Julia set for some $\tilde h$ with $\|\tilde h\|_{\diamond_1}<\infty$.
So $g$ is a $(\norm{\cdot}_{\diamond_1},\phi)$-multiplicative cocycle as in
Definition \ref{d:cocycle} for $t:=1$, $s:=0$ and $\xi:=e^{i \tilde h}$. We prove now the equivalence of the above five properties.
It is clear that (ic) $\Rightarrow$ (ib) $\Rightarrow$ (ia) and
(iib) $\Rightarrow$ (iia). Lemma \ref{l:non-cocycle} implies that (iia) $\Rightarrow$ (ic).
The following implications complete the proof.

\medskip\noindent
{\bf (iia) $\Rightarrow$ (iib).}   By Lemma \ref{l:non-cocycle}, $\|\lam^{-n}\Ll_{[it]}^n\|_{\diamond_1}$ is bounded uniformly on $n$. 
Thus, the spectral radius of $\lam^{-1}\Ll_{[it]}$ with respect to the norm $\norm{\cdot}_{\diamond_1}$ is at most equal to 1. 
We easily deduce that (iia) $\Rightarrow$ (iib).

\medskip\noindent
{\bf (ia) $\Rightarrow$ (iia).} Assume by contradiction that (ia) is true
and (iia) is not true. Consider $t,s$, and $\xi$ as in Definition \ref{d:cocycle}.
If $\xi(z)\not=0$ we divide it by $|\xi(z)|$. This allows us to assume that
$\xi$ is bounded. Define $h:=\xi\rho$. Using the cocycle property of $g$
and the $\lam^{-1}\Ll$-invariance of $\rho$, we have 
$$\lam^{-1}\Ll_{[it]} h = \lam^{-1}\Ll (e^{itg}\xi\rho)=e^{is} \lam^{-1}\Ll((\xi\circ f)\rho)=e^{is}\xi \lam^{-1} \Ll(\rho) =e^{is} h.$$
It follows that $\lam^{-n}\Ll_{[it]}^nh = e^{ins} h$ for $n\geq 1$.
Fix an arbitrary positive constant $\epsilon$ and choose a smooth function
$\tilde h$ such that $\|h-\tilde h\|_{L^1(m_\phi)}\leq \epsilon$. Since (iia) is
not true, $\lam^{-n} \Ll_{[it]}^n\tilde h$ converges uniformly to 0. Moreover, we have
$$|h|=|\lam^{-n}\Ll_{[it]}^nh|\leq |\lam^{-n}\Ll_{[it]}^n \tilde h| + |\lam^{-n}\Ll_{[it]}^n(h-\tilde h)| \leq 
|\lam^{-n}\Ll_{[it]}^n \tilde h| + \lam^{-n}\Ll^n|h-\tilde h|.$$
This and
 the $\lam^{-1}\Ll^*$-invariance of $m_\phi$ imply that
$$\|h\|_{L^1(m_\phi)} \leq \|\lam^{-n}\Ll_{[it]}^n\tilde h\|_{L^1(m_\phi)} + \|h-\tilde h\|_{L^1(m_\phi)}
\leq \|\lam^{-n}\Ll_{[it]}^n\tilde h\|_{L^1(m_\phi)} + \epsilon.$$
By taking $n$ going to infinity,
we obtain that $\|h\|_{L^1(m_\phi)} \leq\epsilon$, which implies that $h=0$ and hence $\xi=0$, both 
$m_\phi$-almost everywhere. This contradicts the requirement on $\xi$ in
Definition \ref{d:cocycle} and ends the proof of the proposition.
\end{proof}

 \begin{proof}[Proof of Theorem \ref{t:LCLT}] We follow here the proof of
  \cite[Th.\,C]{dinh2007thermodynamics}, which is based on \cite[Theorem 10.17]{breiman}. 
  Let $\psi$ be any real-valued function
 in $L^1 (\R)$ whose Fourier transform 
 $$\hat \psi (x) := \frac{1}{\sqrt{2\pi}}  \int_{\R} \psi (t) e^{-itx} dt$$
is a continuous
function with support contained in some interval $[-\delta, \delta]$. Define 
\begin{equation} \label{e:An}
A_n (x) :=
\sigma\sqrt{2 \pi n}
\sca{\mu_\phi, \psi (x+ S_n (g)}
-
e^{-x^2 / (2\sigma^2 n) } \int_{-\infty}^\infty
\psi (t)  dt.
 \end{equation}
 By
 \cite[Th.\,10.7]{breiman}, it suffices to show that $A_n(x)$ converges to 0 as $n$ goes to infinity, uniformly on $x$.
 
\medskip\noindent 
{\bf Claim.} We have 
\begin{equation} \label{e:An-bis}
A_n (x) =\int_{-\delta \sigma \sqrt{n}}^{\delta \sigma \sqrt{n}} 
\hat \psi \Big(\frac{t}{\sigma \sqrt{n}}\Big)
 e^{itx / (\sigma\sqrt{n})}
\big\langle m_\phi, \lam^{-n}\Ll^n_{[it/(\sigma\sqrt{n})]} \rho\big\rangle dt
-
\hat \psi (0)
\int_{-\infty}^\infty e^{itx/ (\sigma\sqrt{n})} e^{-t^2 /2} dt.
\end{equation}

\proof[Proof of Claim]
The second term in the RHS of \eqref{e:An} is equal to the second term in the RHS of \eqref{e:An-bis} because
$$e^{-x^2/(2\sigma^2 n)}={1\over \sqrt{2\pi}}\int_{-\infty}^\infty e^{itx/(\sigma\sqrt{n})} e^{-t^2/2} dt \qquad \text{and} \qquad \hat\psi(0)={1\over \sqrt{2\pi}} \int_{-\infty}^\infty \psi(t)dt.$$
It remains to compare the first terms. Using the identities 
$$\big \langle m_\phi,e^{itS_n(g)}h \big \rangle
= \big \langle m_\phi,\lam^{-n}\Ll^n(e^{itS_n(g)}h)\big \rangle = \big \langle m_\phi,\lam^{-n}\Ll_{[it]}^nh \big \rangle
\quad \text{and} \quad \psi(z)={1\over \sqrt{2\pi}} \int_{-\delta}^\delta \hat\psi(t) e^{itz}dt$$
and the fact that $\mu_\phi=\rho m_\phi$, we obtain
\begin{eqnarray*}
\sigma\sqrt{2 \pi n}
\sca{\mu_\phi, \psi (x+ S_n (g)} &=& \sigma\sqrt{n}\int_{-\delta}^\delta \hat\psi(t) e^{itx} \big\langle\mu_\phi,e^{itS_n(g)}\big\rangle dt  \\
 &=& \sigma\sqrt{n}\int_{-\delta}^\delta \hat\psi(t) e^{itx} \big\langle m_\phi,\lam^{-n}\Ll^n_{[it]}\rho\big\rangle dt .
\end{eqnarray*}
The last expression is equal to the first term in the RHS of \eqref{e:An-bis}
by the change of variable $t\mapsto t/(\sigma\sqrt{n})$. This ends the proof of the claim.
\endproof 
 
Notice that for any constant $\delta_0 >0$ we have the following partial
estimate for the second term in the RHS of \eqref{e:An-bis}:
\begin{equation}\label{e:a3}
\lim_{n\to \infty}
 \Big| 
\int_{|t|>\delta_0 \sigma\sqrt{n} } e^{itx/ (\sigma\sqrt{n})} e^{-t^2 /2}dt
\Big|
\leq \lim_{n\to \infty} \int_{|t|>\delta_0 \sigma \sqrt{n} }  e^{-t^2 /2} dt =0.
\end{equation}
Moreover, it follows from Proposition \ref{p:decomposition}, Lemma \ref{l:development},
and the identity $\langle m_\phi,\rho\rangle=1$, that
\begin{equation}\label{e:a1_1}
\lim_{n\to \infty}
\hat \psi \Big(\frac{t}{\sigma \sqrt{n}}\Big)
\big\langle m_\phi, \lam^{-n}\Ll^n_{[it/(\sigma\sqrt{n})]} \rho\big\rangle -
\hat \psi (0) e^{-t^2 /2} =0 \quad \mbox{for all } t \in \R,
\end{equation}
and, for $\delta_0<\delta$ sufficiently small,
\begin{equation}\label{e:a1_2}
\Big|
\hat \psi \Big(\frac{t}{\sigma \sqrt{n}}\Big)
\big\langle m_\phi, \lam^{-n}\Ll^n_{[it/(\sigma\sqrt{n})]} \rho\big\rangle
 -
\hat \psi (0) e^{-t^2 /2}\Big|
\lesssim e^{-t^2 /4}
\big\|\hat \psi\big\|_{\infty}
\quad \mbox{when } |t|\leq \delta_0 \sigma \sqrt{n}.
\end{equation}
Since the RHS in \eqref{e:a1_2} defines an integrable
function of $t \in \R$, \eqref{e:a1_1} and \eqref{e:a1_2} imply that
\begin{equation*}
\lim_{n\to \infty}
\int_{-\delta_0 \sigma \sqrt{n}}^{\delta_0 \sigma \sqrt{n}}
 \hat \psi \Big(\frac{t}{\sigma \sqrt{n}}\Big) e^{itx / (\sigma\sqrt{n})}
\big\langle m_\phi, \lam^{-n}\Ll^n_{[it/(\sigma\sqrt{n})]} \rho\big\rangle dt
-
\hat \psi (0)
\int_{-\delta_0\sigma\sqrt{n}}^{\delta_0\sigma\sqrt{n}} e^{itx/ (\sigma\sqrt{n})} e^{-t^2 /2} dt = 0.
\end{equation*}
Using this, \eqref{e:a3}, and the above claim, it is
enough to prove that (compare with the first term in the RHS of \eqref{e:An-bis})
\[
A'_n (x) :=\int_{\delta_0 \leq  |t/(\sigma\sqrt{n})|<\delta} \hat \psi \Big(\frac{t}{\sigma \sqrt{n}}\Big) e^{itx / (\sigma\sqrt{n})}
\big\langle m_\phi, \lam^{-n}\Ll^n_{[it/(\sigma\sqrt{n})]} \rho
\big\rangle dt
\]
converges to 0.
Here is where we use the assumption that $g$ is not a multiplicative
cocycle. By Lemma \ref{l:non-cocycle} and Proposition \ref{p:mul_cocycle},
we can find two constants $c>0$ and $0<r<1$ such that 
$\big\|\lam^{-n}\Ll^n_{[it/(\sigma\sqrt{n})]}\big\|_{\diamond_1}\leq cr^n$
 for $|t/(\sigma \sqrt{n})|\in [\delta_0, \delta]$.
 This shows that the term $\sca{\cdot,\cdot}$ in the definition of $A'_n(x)$  goes to zero exponentially fast, and we deduce 
 that  $A'_n  (x)=O(\sqrt{n} r^n)$ as $n$ goes to infinity. This completes the proof of the theorem.
\end{proof}

\subsection{Almost Sure Invariant Principle (ASIP) and consequences}\label{ss:asip_consequences}
We can now prove the ASIP
for observables which are not coboundaries,
see Definition \ref{d:asip}. 
The ASIP
was proved by
Dupont \cite{dupont2010bernoulli}
 in the case where  $\phi=0$
for observables which are H\"older continuous, or
admit analytic singularities, 
by using \cite{philipp1975almost}, see also
Przytycki-Urba{\'n}ski-Zdunik
 \cite{przytycki1989harmonic}
for $k=1$
and $\phi=0$.

\begin{theorem}
Under the hypotheses of Theorem \ref{t:goal}, let $g\colon\P^k\to\R$ be
such that $\norm{g}_{\diamond_1}$ is finite
and $\sca{\mu_\phi, g}=0$. Assume that
$g$ is not a coboundary. Then $g$ satisfies the
ASIP with error rate $o(n^{q})$ for all $q>1/4$.
\end{theorem}

\begin{proof}
We will prove that the assumptions of Theorem \ref{t:asip-spectral} 
are satisfied by the system $f\colon\P^k\to\P^k$ and the  operators $\mathcal T_t := \lam^{-1} \Ll_{[it]}$
for $t\in\R$. These operators act on the functional
space $\mathcal H :=\{h \ : \ \norm{h}_{\diamond_1}<\infty\}$
which is endowed with the norm 
$\norm{\cdot}_{\diamond_1}$. We consider here the probability measures
$\nu:=\mu_\phi$, $\nu^* := m_\phi$, the function $\xi:= \rho$,
and the constant $p$ (of Theorem \ref{t:asip-spectral})
large enough.
We only need to check the strong coding condition in  Theorem \ref{t:asip-spectral}.
Indeed, the spectral description is a consequence of 
Theorem \ref{t:goal}, 
the weak regularity follows from  Lemma \ref{l:perturbed} or Lemma
\ref{l:non-cocycle}, and the last condition is satisfied since $g$ is continuous.

Let us prove that the strong coding condition holds
for all $t_0,\ldots,t_{n-1} \in \R$. 
As in the proof of \eqref{e:ltn-l0n}, we obtain
\begin{eqnarray*}
\Ll_{[0]}^n ( e^{i \sum_{l=0}^{n-1} t_l  g\circ f^l }\rho) & = & \Ll_{[0]}
\pa{\Ll_{[0]}^{n-1} \big(e^{i t_{n-1} g \circ f^{n-1}} e^{i \sum_{l=0}^{n-2} t_l  g\circ f^l } \rho\big)}\\
& = &
\Ll_{[0]} \pa{e^{i t_{n-1} g} \Ll_{[0]}^{n-1} \big(e^{i \sum_{l=0}^{n-2} t_l  g\circ f^l } \rho\big) } \\
& = & \Ll_{[i t_{n-1}]}  \Ll_{[0]}^{n-1} \big(e^{i \sum_{l=0}^{n-2} t_l  g\circ f^l } \rho\big) .
\end{eqnarray*}
Hence, by induction, we get 
$$\Ll_{[0]}^n ( e^{i \sum_{l=0}^{n-1} t_l  g\circ f^l }\rho) =  \Ll_{[i t_{n-1}]} \circ\cdots \circ  \Ll_{[i t_0]}\rho.$$
The identity \eqref{e:simplify} and the last one imply 
$$\langle \mu_\phi, e^{i \sum_{l=0}^{n-1} t_l  g\circ f^l }
\rangle = \langle m_\phi, \lam^{-n} \Ll_{[0]}^n (e^{i \sum_{l=0}^{n-1} t_l  g\circ f^l }\rho)\rangle
=\sca{m_\phi, \lam^{-n} \Ll_{[i t_{n-1}]}\circ \cdots \circ \Ll_{[i t_0]} \rho}.$$
This is the desired strong coding condition. Applying Theorem \ref{t:asip-spectral} gives the result.
\end{proof}

We also have the following direct consequence
of the ASIP and Theorem \ref{t:clt-spectral}, 
see Definitions
\ref{d:lil}, \ref{d:asclt}, and Theorem
\ref{t:a_lil_asclt}. The LIL was established in \cite{szostakiewicz2014stochastics} in the case
where both the weight $\phi$
and the observable $g$
 are H\"older continuous.

\begin{corollary}\label{c:consequences-asip}
Under the hypotheses of Theorem \ref{t:goal}, let $g$ be such
that $\norm{g}_{\diamond_1}$ is finite and $\sca{\mu_\phi, g}=0$. Assume that
$g$ is not a coboundary. Then $g$ satisfies the ASCLT, the LIL,
and the CLT with error rate $O(n^{-1/2})$.
\end{corollary}

\subsection{Large Deviation Principle (LDP)}
We conclude the statistical study of $(\P^k, f, \mu_\phi)$
with the following property, see Definition \ref{d:ldp}.
This is new in this generality 
for all $k\geq 1$, even for $\phi=0$
(see \cite{comman2011large} for the case when $k=1$ and some kind of weak hyperbolicity is assumed).
The LDP in particular implies the Large Deviation Theorem, which is proved
in \cite{dinh2010exponential} in the case $\phi=0$,  see also \cite{pollicott1996large,dinh2010dynamics}.

\begin{theorem}
Under the hypotheses of Theorem \ref{t:goal}, let $g\colon \P^k \to \R$ be such that $\norm{g}_{\diamond_1}$ is finite and
 $\sca{\mu_\phi, g}=0$. 
Assume that $g$ is not a coboundary.
 Then $g$ satisfies the LDP.
\end{theorem}

\begin{proof}
We apply Theorem \ref{t:ge_ldp} for the random variables
$\mathcal S_n:=S_n(g)$ (see \eqref{e:birkhoff-sum})
with respect to the probability $\mu_\phi$, and for $\Lambda (\theta) := \log \alpha(\theta)$, where
$\alpha (\theta)$ is defined in Proposition \ref{p:decomposition}. 
So, $\Lambda(\theta)$ is analytic and well defined for $\theta$ small enough. 
Since $g$ is not a coboundary, we have $\sigma>0$ and
the convexity 
of $\Lambda$
follows
from the expansion
 of $\alpha(\theta)$ given in Lemma \ref{l:development}.
Finally, as in \eqref{e:mu-Sn}, using $\theta$ instead of $1/n$, we get 
 \[
\langle  \mu_\phi, e^{\theta S_n (g)}\rangle =\alpha(\theta)^n \sca{  m_\phi, \Phi_{\theta} (\rho) }
+
\big\langle  m_\phi, \Psi^n_{\theta}(\rho)\big\rangle
\] 
which implies (note that $\sca{  m_\phi, \Phi_{\theta} (\rho) }$ does not vanish for $\theta$ small as it is equal to 1 when $\theta=0$)
$$\lim_{n\to\infty} {1\over n} \log \langle  \mu_\phi, e^{\theta S_n (g)}\rangle =\log \alpha(\theta)=\Lambda(\theta).$$
Thus, Condition
 \eqref{e:for_ldp} in Theorem \ref{t:ge_ldp} is satisfied. Applying that theorem gives the result.
\end{proof}

\appendix

\section{Abstract statistical theorems} \label{a:abstract-result}

We provide here
the precise definitions for 
the statistical properties studied in Section \ref{s:statistic}, 
as well as
the criteria, used in the previous section, 
ensuring their validity in an abstract setting.
We will consider in what follows a dynamical system $T\colon (X,\mathcal F)\to (X,\mathcal F)$, where
$\mathcal F$ is a given $\sigma$-algebra,
and $\nu$ a
probability measure which is invariant with respect to $T$. Given any observable
(real-valued measurable function) 
 $g \in L^1 (\nu)$ and $n\geq 1$
 we denote by $S_n (g)$
 the Birkhoff sum $S_n (g):= \sum_{j=0}^{n-1} g \circ T^j$.
 We are interested in comparing the behaviour of the sequence $S_n (g)$ with that
 of a sum of independent identically distributed random variables $Z_n$ 
 of mean $\sca{\nu,g}$. 
Notice that the invariance of $\nu$ precisely implies
that the functions $g\circ T^j$, considered as random variables, are identically distributed.
 The goal is to prove that, under suitable assumptions, the sequence
  of \emph{weakly dependent} random variables $g\circ T^j$ and 
 their sums $S_n (g)$
 enjoy many of the statistical properties of the sequences $Z_n$ and $\sum_{j=0}^{n-1}Z_j$.
 For simplicity, we only consider observables $g$ of zero mean, i.e., such that
 $\langle \nu,g\rangle=0$. 
 It is easy to extend the discussion to the case of non-zero mean.
 
 In what follows, we will denote by $\Ex(h | \mathcal G)$ the conditional expectation of 
 an $\mathcal F$-measurable function $h$
 with respect to a $\sigma$-algebra $\mathcal G\subseteq \mathcal F$. 
 The $\sigma$-algebra 
$T^{-n}\mathcal F$ 
is generated by the sets of the  form  $T^{-n} (B)$ with $B$ in $\mathcal F$.
 We will say that $h$ is
 a \emph{coboundary} if there exists $\tilde h \in L^2 (\nu)$ such that
 $h = \tilde h \circ T  - \tilde h$. A coboundary belongs to $L^2(\nu)$ and has zero mean.

\begin{definition}[CLT - Central Limit Theorem]\label{d:clt}
Let $T\colon(X,\mathcal F)\to (X,\mathcal F)$ be a dynamical system
and $\nu$ an invariant probability measure. 
Let $g$ be a real-valued integrable function with $\langle\nu,g\rangle=0$.
We say that $g$ satisfies the Central Limit Theorem (CLT)
with variance $\sigma>0$ if
 for any interval $I \subset \R$ we have
\begin{equation}\label{e:clt}
\lim_{n\to \infty} \nu \Big\{
\frac{1}{\sqrt{n}} S_n (g) \in I \Big\}
=
\frac{1}{\sqrt{2\pi} \sigma} \int_I e^{-\frac{t^2}{2\sigma^2}} dt.
\end{equation}
\end{definition}

We give two criteria to ensure
that an observable $g$ satisfies the CLT. The first one, due to Gordin (see also \cite{liverani1996central}),
is more classical and easier to use.

\begin{theorem}[Gordin \cite{gordin1969central}]\label{t:a_clt}
Let $T\colon(X,\mathcal F)\to (X,\mathcal F)$ be a dynamical system
and $\nu$
an invariant probability measure. 
Let $g$ be a real-valued function in $L^2(\nu)$ which
 is not a coboundary.
Assume that
\begin{equation} \label{e:Gordin-cond}
\sum_{n\geq 1} \norm{\Ex \pa{ g | T^{-n} \mathcal F}}^2_{L^2 (\nu)}< \infty.
\end{equation}
Then $g$ satisfies the CLT
with variance $\sigma>0$ given by
 \begin{equation}\label{eq_sigma}
\sigma^2 := \sca{\nu, g^2} +2\sum_{n\geq 1} \sca{\nu, g \cdot \pa{g \circ T^n}}.
\end{equation}
\end{theorem}

Note that under the hypothesis \eqref{e:Gordin-cond}, the observable 
$g$ is not a coboundary if and only if the constant $\sigma$ defined in \eqref{eq_sigma} is non-zero.
Note also that this hypothesis implies that $g$ has zero mean because
each term in the infinite sum in \eqref{e:Gordin-cond} is larger than or equal to $|\langle \nu,g\rangle|^2$
thanks to Cauchy-Schwarz's inequality and the invariance of $\nu$.

The second criterion is 
based on the theory of perturbed operators, which
also allows to study the speed
of convergence in \eqref{e:clt}, i.e., to get 
a version of Berry-Esseen theorem.
Recall that the spectrum 
of a linear operator $\mathcal T$, from a complex Banach space to itself, 
is the closure of the set of  $z\in \C$
such that
the operator $zI- \mathcal T$ is not invertible. The essential spectrum
is obtained from the spectrum by removing its isolated points corresponding to eigenvalues
of finite multiplicity. 
The spectral radius $r(\mathcal T)$
and the essential spectral radius $\ress (\mathcal T)$
are the
radii of the smallest disks centered at the origin which contain the spectrum
and the essential spectrum, respectively.
The following criterion, based on ideas on 
Nagaev and Guivarc'h \cite{nagaev1957some,guivarc1988theoremes},
is a translation in our
setting of \cite[Th.\,3.7]{gouezel2015limit}.

\begin{theorem}[Nagaev-Guivarc'h-Gou\"ezel]\label{t:clt-spectral}
Let $T\colon(X,\mathcal F)\to (X,\mathcal F)$ be a dynamical system
 and $\nu$
an invariant probability measure. 
Let $g$ be a real-valued integrable function.
Assume that there exist a complex Banach space
$\Hh$, a family of operators $\mathcal T_t$
acting on $\mathcal H$, for  $t\in [-\delta,\delta]$ with some constant $\delta>0$, and elements
$\xi \in \Hh$ and $\nu^* \in \Hh^*$ such that
\begin{enumerate}
\item[{\rm(i)}] coding: for all $n\in \N$ and all
$t\in [-\delta,\delta]$, we have $\sca{\nu, e^{i t S_n (g)}} = \sca{\nu^*, \mathcal T_t^n \xi}$;
\item[{\rm(ii)}] spectral description: we have $\ress (\mathcal T_0) <1$, and $\mathcal T_0$ has a unique
eigenvalue of modulus $\geq 1$; moreover, this eigenvalue has multiplicity $1$ 
and is located at $1$;
\item[{\rm(iii)}] regularity: the family $t\mapsto\mathcal T_t$ is of class $C^3$. 
\end{enumerate} 
Assume also that the largest eigenvalue $\alpha(t)$ of $\mathcal T_t$ for $t$ near $0$ admits the order $2$ expansion 
$\alpha(t)=e^{iAt-\sigma^2 t^2/2+o(t^2)}$ for some $A,\sigma\in\R$ with $\sigma>0$.
Then $g-A$ satisfies the CLT with variance $\sigma$
and
the Berry-Esseen theorem holds, i.e.,
the speed of convergence in \eqref{e:clt} is of order $O(n^{-1/2})$. 
\end{theorem}

The following
is an improvement of the CLT, see also Remark \ref{r:lclt}.

\begin{definition}[LCLT - Local Central Limit Theorem]\label{d:lclt}
Let $T\colon(X,\mathcal F)\to (X,\mathcal F)$ be
a dynamical system, and $\nu$
an invariant probability measure. 
Let $g$ be a real-valued integrable function with $\langle\nu,g\rangle=0$. We say that
$g$
satisfies 
the Local Central Limit Theorem (LCLT)
with variance $\sigma >0$
 if for every bounded interval $I\subset \R$
the convergence
\[
\lim_{n\to \infty} \Big|
\sigma \sqrt{n} \nu \{x+S_n(g) \in I \} - \frac{1}{\sqrt{2\pi}} e^{-x^2 / (2\sigma^2 n)} |I|
\Big|=0
\]
holds uniformly in $x \in \R$.
Here $|I|$ denotes the length of $I$.
\end{definition}

Partial sums of independent and identically distributed random variables 
can be compared with Brownian motions of suitable variance, and one can get precise estimates
for the fluctuations around the mean value. In the context of dynamical
systems, we have the following analogous definition.

\begin{definition}[ASIP - Almost Sure Invariant Principle]\label{d:asip}
Let $T\colon(X,\mathcal F)\to (X,\mathcal F)$ be a dynamical system
 and $\nu$
an invariant probability measure. 
Let $g$ be a real-valued integrable function with $\langle\nu,g\rangle=0$. We say that
$g$
satisfies 
the Almost Sure Invariance Principle (ASIP) of variance $\sigma >0$
 and error rate $n^\gamma$
if there exist, on some probability space $\mathcal{X}$, a sequence of random variables $(\mathcal S_n)_{n\geq0}$ and a 
Brownian motion $\mathcal W$ of variance $\sigma$
such that
\begin{enumerate}
\item $\mathcal S_n =\mathcal W(n)+ o(n^{\gamma})$   almost everywhere on $\mathcal{X}$;
\item $S_n(g)$ and $\mathcal S_n$ 
have the same distribution for any $n \geq 0$. 
\end{enumerate}
\end{definition}

General criteria that allow one to establish the ASIP in various contexts
are given in \cite{philipp1975almost}.
We state here a criterion by
Gou\"ezel, see \cite[Th.\ 1.2]{gouezel2010almost} and \cite[Th.\ 5.2]{gouezel2015limit}.
Notice that  Gou\"ezel's
result actually holds in the more general case
of random variables with values in $\R^d$.

\begin{theorem}[Gou\"ezel]\label{t:asip-spectral}
Let $T\colon(X,\mathcal F)\to (X,\mathcal F)$ be a dynamical system
and $\nu$
an invariant probability measure. 
Let $g$ be a real-valued integrable function.
Assume that there exist a complex Banach space
$\Hh$,  a family of operators $\mathcal T_t$
acting on $\mathcal H$, for $t\in [-\delta,\delta]$ with some constant $\delta>0$, 
and elements
$\xi \in \Hh$
 and $\nu^* \in \Hh^*$ such that
\begin{enumerate}
\item[{\rm(i)}] strong coding: for all $n\in \N$ and  all $t_0, \dots, t_{n-1} \in [-\delta,\delta]$, we have
\[\big\langle
\nu, e^{i \sum_{j=0}^{n-1} t_j  g\circ T^j }
\big\rangle 
= \sca{\nu^*, \mathcal T_{t_{n-1}}\circ \cdots\circ  \mathcal T_{t_0} \xi};\]
\item[{\rm(ii)}] spectral description: we have $\ress (\mathcal T_0) <1$, and $\mathcal T_0$
has a unique eigenvalue of modulus $\geq 1$; moreover, this eigenvalue has multiplicity $1$ and is
located at $1$;
\item[{\rm(iii)}]  weak regularity: either $t\mapsto {\mathcal T}_t$ is continuous, or $\norm{\mathcal T_t^n}_{\Hh\to\Hh}\leq c$ for some
constant $c>0$ and for all $t\in [-\delta,\delta]$ and $n\geq 0$;
\item[{\rm(iv)}] there exist $p>2$ and $c>0$ such that $\norm{g\circ T^n}_{L^p (\nu)}\leq c$
for all $n\geq 0$.
\end{enumerate}
Assume also that the largest eigenvalue $\alpha(t)$ of $\mathcal T_t$ for $t$ near $0$ admits the order $2$ expansion 
$\alpha(t)=e^{iAt-\sigma^2 t^2/2+o(t^2)}$ for some $A,\sigma\in\R$ with $\sigma>0$.
Then $g-A$ satisfies the ASIP
of variance $\sigma$ and 
error rate $o(n^q)$ for all $q$ such that
\[
q> \frac{p}{4p-4} = \frac{1}{4} + \frac{1}{4p-4} \cdot
\]
\end{theorem}

The following properties are general 
consequences of the ASIP, see for instance \cite{philipp1975almost,lacey1989note,chazottes2007almost}.

\begin{definition}[LIL - Law of iterated logarithms]\label{d:lil}
Let $T\colon(X,\mathcal F)\to (X,\mathcal F)$ be a dynamical system
 and  $\nu$
an invariant probability measure. 
Let $g$ be a real-valued integrable function such that $\langle\nu,g\rangle=0$. We say that
$g$
satisfies 
the Law of Iterated Logarithms (LIL) with variance $\sigma>0$
if
\[
\limsup_{n\to \infty} \frac{S_n(g)}{\sigma\sqrt{2n \log \log n}}
=1 \quad \nu\mbox{-almost everywhere}.
\]
\end{definition}

\begin{definition}[ASCLT - Almost sure Central Limit Theorem]\label{d:asclt}
Let $T\colon(X,\mathcal F)\to (X,\mathcal F)$ be a dynamical system
 and  $\nu$ an invariant probability measure. 
Let $g$ be a real-valued integrable function such that $\langle\nu,g\rangle=0$. We say that
$g$
satisfies 
the Almost Sure Central Limit Theorem (ASCLT) with variance $\sigma$
$>0$
if, for $\nu$-almost every point  $x\in X$,
\[
\frac{1}{\log n} \sum_{j=1}^n \frac{1}{j} \delta_{j^{-1/2} S_j (g)(x)} \to \mathcal N (0,\sigma).
\]
In particular, $\nu$-almost surely 
\[
\frac{1}{\log n} \sum_{j=1}^n \frac{1}{j} \1_{\{ j^{-1/2} S_j (g) \leq t_0\}} 
\to \frac{1}{\sqrt{2\pi} \sigma} \int_{-\infty}^{t_0} e^{- \frac{t^2}{2\sigma^2}} dt
\]
for all $t_0 \in \mathbb R$.
\end{definition}

\begin{theorem}\label{t:a_lil_asclt}
Let $T\colon(X,\mathcal F)\to (X,\mathcal F)$ be a dynamical system
and $\nu$
 an invariant probability measure. 
Let $g$ be a real-valued integrable function with $\langle\nu,g\rangle=0$. Assume that
$g$ satisfies the ASIP with variance $\sigma>0$
and error rate $n^\gamma$ for some
 $\gamma < 1/2$. Then $g$
satisfies the LIL and the ASCLT, both with the same variance $\sigma$.
\end{theorem}

The last property that
we recall is the Large Deviation Principle. 
It gives very precise estimates on the measure of the set where the partial sums
are far from the mean value. It implies in particular the Large Deviation Theorem,
which only requires an upper bound
for the measure in \eqref{e:ldp} below.

\begin{definition}[LDP - Large Deviation Principle]\label{d:ldp}
Let $T\colon(X,\mathcal F)\to (X,\mathcal F)$ be a dynamical system
 and $\nu$
an invariant probability measure. 
Let $g$ be a real-valued integrable function such that $\langle\nu,g\rangle=0$. We say that
$g$
satisfies 
the Large Deviation Principle (LDP)
if there exists a non-negative, strictly convex function $c$ which is defined on a neighbourhood of $0\in \R$,
vanishes only at $0$, and such that, for all $\eps>0$ sufficiently small,
\begin{equation}\label{e:ldp}
\lim_{n\to \infty}
 \frac{1}{n}\log \nu 
\Big\{
x\in X \colon
\frac{ S_n (g)(x)}{n}  >\epsilon \Big\}=-c(\eps).
\end{equation}
\end{definition}

The following result, established in \cite{bougerol1985products}, 
is
a local version of G\"artner-Ellis
 Theorem \cite{gartner1977large, ellis1984large, dembo1998large} which 
 can be used to prove the LDP, see \cite[Lem.\ XIII.2]{hennion2001limit}.

\begin{theorem}[Bougerol-Lacroix]\label{t:ge_ldp}
For all $n\geq 1$ denote by $\P_n$ a probability distribution, by $\Ex_n$ the
corresponding expectation operator and by $\mathcal S_n$ a real-valued random variable. Assume that on some interval
$[-\theta_0, \theta_0],$ with $\theta_0 >0$, we have
\begin{equation}	\label{e:for_ldp}
\lim_{n\to \infty} \frac{1}{n} \log \Ex_n (e^{\theta \mathcal S_n}) = \Lambda (\theta),
\end{equation}
where $\Lambda$ is a $\Cc^1$, strictly
convex function
satisfying $\Lambda'(0)=0$.
Then,  for all  $0 < \eps < \Lambda (\theta_0)/\theta_0$, we have
\[
\lim_{n\to\infty} \frac{1}{n} \log \P_n (\mathcal S_n > n\eps) = - c(\eps)<0,
\]
where $c (\eps) :=\sup \{\theta \eps - \Lambda (\eps)\colon 
|\theta | \leq \theta_0\}$
is the Legendre transform of $\Lambda$,
which is non-negative, strictly
convex, and vanishes only at $0$.
\end{theorem}

Notice that the symmetric statement for the measure of the set where $n^{-1}S_n (g) <-\eps$ 
in \eqref{e:ldp} 
can be obtained by applying the result above to the sequence of random variables $-\mathcal S_n$.

\printbibliography

\end{document}